\documentclass[psamsfonts]{amsart}

\usepackage{enumerate}
\usepackage{url}
\usepackage[font=small,labelfont=bf]{caption}
\usepackage[all,arc]{xy}
\usepackage{tabularx}
\usepackage{adjustbox}

\usepackage{amssymb,latexsym}
\usepackage{amsmath,amsthm}
\usepackage{amsfonts,mathrsfs}
\usepackage{mathtools}

\usepackage{tikz}
\usetikzlibrary{shapes.geometric, arrows, positioning,decorations.pathreplacing,calligraphy}

\tikzstyle{bloc} = [rectangle, rounded corners, 
minimum width=3cm, 
minimum height=1cm,
text centered, 
draw=black, 
fill=airforceblue!30]

\tikzstyle{decision} = [diamond,
minimum width=3cm, 
minimum height=1cm, 
text centered, 
draw=black, 
fill=airforceblue!30]
\tikzstyle{arrow} = [thick,->,>=stealth]

\tikzstyle{io} = [trapezium, 
trapezium stretches=true, 
trapezium left angle=70, 
trapezium right angle=110, 
minimum width=3cm, 
minimum height=1cm, text centered, 
draw=black, fill=blue!30]

\tikzstyle{process} = [rectangle, 
minimum width=3cm, 
minimum height=1cm, 
text centered, 
text width=3cm, 
draw=black, 
fill=orange!30]

\usepackage{tikz-cd}
\usepackage{todonotes}

\usepackage[all]{xy}
\usepackage{graphicx}

\usepackage{hyperref}
\hypersetup{
    colorlinks,
    citecolor=blue,
    filecolor=blue,
    linkcolor=blue,
    urlcolor=blue
}

\usepackage{enumitem}

\usepackage{listings}
\usepackage{color}

\definecolor{dkgreen}{rgb}{0,0.6,0}
\definecolor{gray}{rgb}{0.5,0.5,0.5}
\definecolor{mauve}{rgb}{0.58,0,0.82}

\newtheorem{thm}{Theorem}[section]
\newtheorem{cor}[thm]{Corollary}

\newtheorem{lem}[thm]{Lemma}

\newtheorem{theorem}[thm]{Theorem}

\newtheorem{corollary}[thm]{Corollary}
\newtheorem{lemma}[thm]{Lemma}
\newtheorem{proposition}[thm]{Proposition}

\newtheorem*{theorem*}{Theorem}

\theoremstyle{definition}
\newtheorem{defn}[thm]{Definition}

\newtheorem{fact}[thm]{Fact}

\theoremstyle{definition}
\newtheorem{definition}[thm]{Definition}
\newtheorem{example}[thm]{Example}
\newtheorem{question}[thm]{Question}

\theoremstyle{remark}
\newtheorem{rem}[thm]{Remark}

\newtheorem{remark}[thm]{Remark}

\newcommand{\brac}[2]{\left[ {#1} , {#2} \right]}

\makeatletter
\let\c@equation\c@thm
\makeatother
\numberwithin{equation}{section}

\newcommand\N{\mathbb{N}}

\newcommand\F{\mathbb{F}}

\newcommand\G{\mathbb{G}}

\newcommand\Lbb{\mathbb{L}}

\def\Th{\operatorname{Th}}

\newcommand\CC{\mathscr{C}}
\newcommand\LL{\mathscr{L}}

\newcommand\bfG{\mathbf{G}}
\newcommand\bfL{\mathbf{L}}

\newcommand{\tp}{\mathrm{tp}}

\newcommand{\HS}{\mathrm{HS}}

\newcommand{\lev}{\mathrm{lev}}

\def\seq{\subseteq}

\def\lteq{\trianglelefteq}

\newcommand{\set}[1]{\left\{ {#1} \right\}}
\newcommand{\vect}[1]{\langle {#1} \rangle}
\newcommand{\abs}[1]{\lvert {#1} \rvert}

\DeclareMathOperator{\ad}{ad}
\DeclareMathOperator{\Span}{span}
\DeclareMathOperator{\rk}{rk}

\DeclareMathOperator{\Der}{Der}

\newcommand{\blue}[1]{\textcolor{blue}{#1}}
\definecolor{airforceblue}{rgb}{0.36, 0.54, 0.66}

\def\Ind{\setbox0=\hbox{$x$}\kern\wd0\hbox to 0pt{\hss$\mid$\hss}
\lower.9\ht0\hbox to 0pt{\hss$\smile$\hss}\kern\wd0}
\def\Notind{\setbox0=\hbox{$x$}\kern\wd0\hbox to 0pt{\mathchardef
\nn=12854\hss$\nn$\kern1.4\wd0\hss}\hbox to
0pt{\hss$\mid$\hss}\lower.9\ht0 \hbox to 0pt{\hss$\smile$\hss}\kern\wd0}
\def\ind{\mathop{\mathpalette\Ind{}}}
\def\nind{\mathop{\mathpalette\Notind{}}}

\def\indi#1{\mathop{\ \ \hbox to 0ex{\hss$\vert^{\hbox to 0ex{$\scriptstyle#1$\hss}}$\hss}
\lower1ex\hbox to 0ex{\hss$\smile$\hss}\ \ }}

\def\nindi#1{\mathop{\ \ \hbox to 0ex{\hss$\!\not{\vert}^{\hbox to 0ex{$\scriptstyle\,#1$\hss}}$\hss}
\lower1ex\hbox to 0ex{\hss$\smile$\hss}\ \ }}

\renewcommand{\models}{\vDash}

\title{Model-theoretic properties of nilpotent groups and Lie algebras}
\thanks{$^\dagger$ supported by the UKRI Horizon Europe Guarantee Scheme, grant no EP/Y027833/1. $^\ddagger$ supported by NSF grant DMS-2246992.  }

\author[C. d'Elb\'{e}e]{Christian d\textquoteright Elb\'ee$^\dagger$}
\address{School of Mathematics, University of Leeds\\
Office 10.17f LS2 9JT, Leeds}
\email{C.M.B.J.dElbee@leeds.ac.uk}
\urladdr{\href{http://choum.net/\textasciitilde chris/page\textunderscore perso/}{http://choum.net/\textasciitilde chris/page\textunderscore perso/}}

\author[I. M\"uller]{Isabel M\"uller}
\address{Department of Mathematics and Actuarial Science \\
The American University in Cairo \\ Egypt }
\email{isabel.muller@aucegypt.edu}
\urladdr{\href{https://www.aucegypt.edu/fac/isabel}{https://www.aucegypt.edu/fac/isabel}}

\author[N. Ramsey]{Nicholas Ramsey$^{\ddagger}$}
\address{Department of Mathematics \\
University of Notre Dame\\
 USA}
\email{sramsey5@nd.edu}
\urladdr{\href{https://math.nd.edu/people/faculty/nicholas-ramsey/}{https://math.nd.edu/people/faculty/nicholas-ramsey/}}

\author[D. Siniora]{Daoud Siniora}
\address{Department of Mathematics and Actuarial Science \\The American University in Cairo\\
 Egypt}
\email{daoud.siniora@aucegypt.edu}
\urladdr{\href{https://sites.google.com/view/daoudsiniora/}{https://sites.google.com/view/daoudsiniora/}}

\date{\today}

\begin{document}

\maketitle

\begin{abstract}
We give a systematic study of the model theory of generic nilpotent groups and Lie algebras.  We show that the Fra\"iss\'e limit of $2$-nilpotent groups of exponent $p$ studied by Baudisch is $2$-dependent and NSOP$_{1}$.  We prove that the class of $c$-nilpotent Lie algebras over an arbitrary field, in a language with predicates for a Lazard series, is closed under free amalgamation.  We show that for $2 < c$, the generic $c$-nilpotent Lie algebra over $\mathbb{F}_{p}$ is strictly NSOP$_{4}$ and $c$-dependent.  Via the Lazard correspondence, we obtain the same result for $c$-nilpotent groups of exponent $p$, for an odd prime $p > c$. 
\end{abstract}

\setcounter{tocdepth}{1}
\tableofcontents

\section{Introduction}

A classic theorem in model theory of groups states that every abelian group, viewed as a structure in the language of groups, is stable.  This follows from Szmielew's quantifier elimination for abelian groups down to pp-formulas \cite{szmielew}. This leads naturally to the question of whether analogous results might be proved for \emph{nilpotent} groups, which are in some sense the  least complicated class of groups properly containing the abelian groups.  Constructions of Mekler \cite{Mekler} and related ones by Ershov \cite{Ershov} show, however, that already groups of nilpotence class $2$ and exponent $p$, for an odd prime $p$, are totally wild.  These results give an \emph{ad hoc} construction of a nilpotent group that codes an arbitrary graph into the commutation relation on the group.  Given the nature of these constructions, these results leave open whether or not nilpotent groups might still \emph{generically} exhibit tame model-theoretic behavior. 

This paper studies the neostability-theoretic properties of nilpotent groups and Lie algebras at a generic scale and lays the foundations for a detailed study of definability in such structures.  Within model theory, the study of existentially closed nilpotent groups was initiated by Saracino in \cite{saracino1976existentially}, who showed that the theories of nilpotent groups of class $c$ and torsion-free nilpotent groups of class $c$ do not admit model companions. Saracino \cite{saracino1978existentially} and later Saracino and Wood \cite{saracino1979periodic} extended the theory, focusing on the nilpotence class 2 case. Maier gave an elaborate amalgamation construction for torsion-free groups of nilpotence class $c$, for $c$ possibly greater than $2$ \cite{maier1986amalgame}.  He extended this to nilpotent groups of class $c$ and exponent $p$, for primes $p > c$ \cite{Maierexpp}. Lie rings of Morley rank less than or equal to $4$ have recently be described by Deloro and Tindzogho Ntsiri \cite{deloro2023simple}.

Our starting point is the investigation of Fra\"iss\'e limits of nilpotent groups of class $2$ and exponent $p$, for an odd prime $p$, which we study in detail in Section \ref{nil-2}. In \cite{Baudisch2}, Baudisch considers the class of finite $2$-nilpotent groups $G$ of exponent $p$ and shows that it forms a Fra\"iss\'e class in the language of groups, together with a predicate $P$ for a subgroup such that $[G,G] \subseteq P(G) \subseteq Z(G)$.  These Fra\"iss\'e limits are natural generalizations of the extra-special $p$-groups, which impose the additional requirement that $Z(G)$ be cyclic.  The extra special $p$-groups were studied by Felgner in \cite{Felgner} and were later an important example of quasi-finite theories, studied by Cherlin and Hrushovski, all of which have simple theories. However, in \cite{Baudisch1}, Baudisch shows that his Fra\"iss\'e limits have TP$_{2}$, and therefore have independence property and are not simple.  We build on Baudisch's analysis, showing that the theories of these Fra\"iss\'e limits are NSOP$_{1}$ and $2$-dependent and thus are, roughly speaking, minimally complicated outside of the NIP, simple theories, and NTP$_{2}$ theories. The feature of these groups that makes them particularly tractable is that, in such a group, both $G/Z(G)$ and $Z(G)$ may be viewed as $\mathbb{F}_{p}$-vector spaces, with the commutator inducing an alternating bilinear map $[\cdot,\cdot] : G/Z(G) \times G/Z(G) \to Z(G)$.  Although the interpretation of this bilinear map is not a \emph{bi-interpretation}, we, following Baudisch, observe that many of the model-theoretic features of the Fra\"iss\'e limit are determined by those of the associated bilinear map, reducing the group-theoretic analysis to linear algebra. 

In subsequent sections, we turn our attention to groups of exponent $p$ and nilpotence class greater than $2$. A group $G$ is of nilpotence class $c$ if there is a subnormal series 
$$
G = H_{1} \unrhd H_{2} \unrhd H_{3} \unrhd \ldots 
$$
such that $[H_{i},H_{j}] \subseteq H_{i+j}$ and $H_{k} = 1$ for all $k > c$.  Such a series is called a \emph{Lazard series} of length $c$ for the group; the lower central series of a $c$-nilpotent group is a familiar example.  Lazard series may be analogously defined for Lie algebras, replacing subgroups with subalgebras and commutator with Lie bracket. Our main tool in the study of nilpotent groups is the \emph{Lazard correspondence}, which provides a different way of reducing the study of these groups to linear algebra.  This correspondence, an analogue of the more well-known Malcev correspondence, associates to each group of exponent $p$ and nilpotence class $c$ a Lie algebra of nilpotence class $c$ over $\mathbb{F}_{p}$, assuming the prime $p > c$.  In fact, this correspondence is a uniform bi-interpretation between the group and the Lie algebra which takes a Lazard series for the group to a Lazard series for the Lie algebra.  The details of the Lazard correspondence are outlined in Section \ref{Preliminaries}.  The remainder of the paper considers the model-theoretic properties of nilpotent Lie algebras in a language with predicates for a Lazard series.  The Lazard correspondence allows us, then, to infer properties of certain nilpotent Lie algebras over $\mathbb{F}_{p}$, but, moreover, this analysis turns out to be interesting in its own right and allows us to analyze nilpotent Lie algebras over arbitrary fields.  

The algebraic heart of the paper is contained in Section \ref{sec:freeamalgamation}, where we prove that, for any field $\mathbb{F}$, the class of $c$-nilpotent Lie algebras over $\mathbb{F}$, in a language with predicates for a Lazard series, has the amalgamation property.  In fact, we prove that such Lie algebras can be freely amalgamated, for a notion of free amalgamation introduced by Baudisch.  The existence of \emph{strong} amalgams for $c$-nilpotent Lie algebras, in the special case where the field of scalars is $\mathbb{F}_{p}$ for a prime $p > c$, can be deduced, via the Lazard correspondence, from the amalgamation results of Maier \cite{Maierexpp}. The existence of \emph{free} amalgams for nilpotent Lie algebras over an arbitrary field was sketched by Baudisch in \cite{Baudisch4}.  We are very influenced by Baudisch's proposed construction in \cite[Theorem 4.1]{Baudisch4}, but the details provided there do not suffice for our applications, so we opted to give our own account.  Our construction proceeds in several stages.  In the first stage, we amalgamate two extensions of codimension 1 over a common ideal, presenting the amalgam as a semi-direct product of the ideal and a free nilpotent Lie algebra generated by elements in the complements of the ideal in the respective structures. In the next stages, we use this construction itself as a step in an inductive construction of an amalgam, which draws heavily on ideas of Maier from the group case.  

In Section \ref{sec:neostabilityproperties}, we spell out the consequences of the existence of free amalgams for the model-theoretic properties of Fra\"iss\'e limits of nilpotent groups and Lie algebras. In this section, we naturally restrict attention to Lie algebras over finite fields, and especially over finite prime fields, which correspond to groups via the Lazard correspondence, since these Fra\"iss\'e limits have $\aleph_{0}$-categorical theories. We show that for $c > 2$, Fra\"iss\'e limits of $c$-nilpotent Lie algebras over finite fields have an SOP$_{3}$ and NSOP$_{4}$ theory.  To show that the theories have SOP$_{3}$, we provide a direct construction of a $3$-nilpotent Lie algebra which witnesses SOP$_{3}$ with respect to quantifier-free formulas.  On the other hand, to show that these theories are NSOP$_{4}$, we leverage the stationary independence relation coming from free amalgamation, following the strategy of Patel for bowtie-free graphs \cite{patel2006family}, later systematized by Conant in \cite{conant2017axiomatic} and generalized by Mutchnik \cite{mutchnik2022conant}.  As NSOP$_{4}$ has recently emerged as a class of theories for which there is some hope of a meaningful structure theory, we are optimistic that these Lie algebras (and the associated groups) can serve, alongside the curve-exluding fields of \cite{johnson2023curve}, as illuminating algebraic examples of strictly NSOP$_{4}$ theories, playing a similar role to that played by Lie geometries for simple theories and played by the two-sorted vector spaces equipped with bilinear forms for NSOP$_{1}$ theories. We summarize the analogies in the following table:

\small{
\begin{table}[htbp]
  \centering
  \label{tab:my-table}
  \begin{adjustbox}{center}
  \begin{tabularx}{\textwidth}{|X|X|X|X|}
  \hline
    Stable & Simple & NSOP$_{1}$ & NSOP$_{4}$ \\
    \hline
    ACF & Psf/ACFA & $\omega$-free PAC fields & Curve-excluding fields \\
    \hline
     Vector spaces & $\mathbb{F}_{p}$-vector spaces with a bilinear map & Vector spaces over ACF with a bilinear map & Nilpotent Lie algebras \\
    \hline 
    Equivalence relations & Random graph & Parameterized equivalence relations & Henson graphs  \\
    \hline 
  \end{tabularx}
  \end{adjustbox}
\end{table}
}
\newpage

We also analyze the place of these examples in the $n$-dependence hierarchy introduced by Shelah \cite{shelah2007definable,shelah2014strongly}.  In essence, a theory is called $n$-dependent if there is no interpretable $(n+1)$-ary $(n+1)$-partite hypergraph that contains the random $(n+1)$-ary $(n+1)$-partite hypergraph as an induced subhypergraph. 
 The placement of a theory in this hierarchy, then, gives a way of quantifying the arity of random relations in the theory.  The $n=1$ case corresponds to the much-studied class of NIP theories, and recent work has generalized some aspects of the theory to this broader setting and produced new examples \cite{chernikov2019n,chernikov2021n,hempel2016n}.  The only known examples of pure groups that are $(n+1)$-dependent but not $n$-dependent were constructed by Chernikov and Hempel \cite{chernikov2019mekler}.  These groups are produced by the aforementioned construction of Mekler, which takes a graph and  produces a group of nilpotence class 2 and exponent $p$ that codes the graph into the commutation relation of the group. Due to the somewhat artificial character of the groups produced by this method, it was left open whether these model-theoretic classes were inhabited by groups occurring `in nature.' We show that the Fra\"iss\'e limit of $c$-nilpotent groups of exponent $p$ (for $p > c$), in the language with predicates for a Lazard series, is $c$-dependent and $(c-1)$-independent.  As these predicates are definable (with quantifiers) in the pure group language, these furnish natural examples that exhibit the strictness of the $n$-dependence hierarchy in groups. Our proof makes use of the `Composition Lemma' of Chernikov-Hempel \cite{ChernikovHempel3}, drawing on a similar set of ideas as their work on multilinear forms, which has yet to appear. 

The above results give a fairly exhaustive analysis of the neostability-theoretic complexity of the theories of generic nilpotent groups and Lie algebras we considered. However, genericity for nilpotent groups can be understood within three distinct rubrics:
\begin{enumerate}
    \item \textbf{Model-theoretic}:  Understand existentially closed nilpotent groups, investigate the existence of Fra\"iss\'e limits and/or model companions and describe their definable sets.
    \item \textbf{Descriptive set-theoretic}:  Naturally view the collection of all nilpotent groups with underlying set $\mathbb{N}$ and interesting subclasses of such nilpotent groups as Polish spaces, describe which properties hold on a comeager set of groups.
    \item \textbf{Probabilistic}:  Consider various models of random groups (e.g. random groups in the sense of Gromov), specialized to the case of nilpotent groups. Calculate the probability that group-theoretic properties hold in a randomly sampled nilpotent group and determine the `probability 1' theory.  
\end{enumerate}

There has been a considerable amount of work in each of these directions. 
 In addition to the model-theoretic work mentioned above, the descriptive set-theoretic point of view on generic groups has been taken up in recent work by Elekes, Geh\'er, Kanalas, K\'atay, and Keleti \cite{elekes2021generic} and by Goldbring, Elayavalli, and Lodha \cite{goldbring2023generic}.  Both of these papers considered the space of all countable groups with domain $\mathbb{N}$ and the former additionally studied the generic properties of the subspace of abelian groups, leaving the class of nilpotent groups as an unexplored intermediary case.  The primary treatment of nilpotent groups from a probabilistic point of view was undertaken by Cordes, Duchin, Duong, Ho, and S\'anchez, who adaped the Gromov model for random groups to the nilpotent setting \cite{cordes2018random}.  In a different vein, Diaconis and Malliaris gave a quantitative study of randomness in Heisenberg groups over $\mathbb{F}_{p}$, particularly illustrative examples of $2$-nilpotent groups \cite{diaconis2021complexity}. We view the work done here as a first step in a broad project of determining how these rubrics for the study of generic nilpotent groups fit together.  In this paper, we focus exclusively on the model-theoretic picture but anticipate that the algebraic foundations laid here will be useful in subsequent explorations of these intersecting frameworks.

\section{Preliminaries} \label{Preliminaries}

Our model-theoretic notation is standard.  We do not distinguish between singletons and tuples except when explicitly mentioned, and we use the usual model-theoretic abuse of notation of denoting $A \cup B$ by $AB$.  We write $a \equiv_{A} b$ to indicate that $\mathrm{tp}(a/A) = \mathrm{tp}(b/A)$. We write $\mathrm{acl}(A)$ and $\mathrm{dcl}(A)$ to denote the algebraic and definable closures of $A$, respectively. We write $\omega$ for the set of natural numbers and, given $k \geq 1$, we define $[k] = \{1, \ldots, k\}$.  

\subsection{Fraïssé theory}\label{subsec:generalitiesonfraisseclass}

Let $\LL$ be a first-order language and let $\CC$ be a class of $\LL$-structures. We say that $\CC$ has the \textit{hereditary property} (HP) if it is closed under substructures, i.e., whenever $B\in\CC$ and $A\subseteq B$, then $A\in\CC$. We say that $\CC$ has the \textit{joint embedding property} (JEP) if whenever $A,B$ are in $\CC$, there exists a structure $S\in\CC$ which embeds both $A$ and $B$. Finally, we say that $\CC$ has the \textit{amalgamation property} if for all $A,B,C$ in $\CC$ and embeddings $f_0:C\to A$ and $g_0:C\to B$, there exists a structure $S\in\CC$ together with embeddings $f_1:A\to S$ and $g_1:B\to S$ such that $f_1\circ f_0=g_1\circ g_0$. In this case, we say that $S$ is an \textit{amalgam} of $A$ and $B$ over $C$. When $f_1(A)\cap g_1(B)=f_1(f_0(C))$ we say that $S$ is a \textit{strong amalgam}.

\begin{center}\begin{tikzcd}
& A \ar[dr,"f_1", dashed] 
&
&[1.5em] \\
C \ar[ur,"f_0"] \ar[dr,"g_0"]
&
& S  \\
& B \ar[ur, "g_1", dashed]
&
&
\end{tikzcd}

The amalgamation property\end{center}

\begin{definition}\label{fraisseclass}~
   \begin{itemize}
       \item  We say that a class $\CC$ of finitely generated $\LL$-structures is a \textit{Fraïssé class} if it is closed under isomorphisms, contains countably many isomorphism types, and has the hereditary property, joint embedding property, and amalgamation property.

       \item A countable $\LL$-structure $M$ is called \textit{homogeneous} if every isomorphism between finitely generated substructures of $M$ extends to an automorphism of $M$.
       \item An $\LL$-structure $M$ is called \textit{uniformly locally finite} if there exists a function $f:\N\to\N$ such that any substructure of $M$ which is generated by $n$ elements has cardinality at most $f(n)$.
   \end{itemize}
\end{definition}

\begin{fact}[Fraïssé's Theorem]
Suppose that $\CC$ is a Fraïssé class of finitely generated $\LL$-structures. Then there exists a unique (up to isomorphism) homogeneous structure $\mathbf{M}$ such that the class of all structures isomorphic to finitely generated substructures of $\mathbf{M}$ (known as the age of $\mathbf{M}$) is precisely $\CC$. We call $\mathbf{M}$ the Fraïssé limit of $\CC$. 
Conversely, the age of a homogeneous structure is a Fraïssé class.
\end{fact}

\begin{fact}\cite[Corollary 6.4.2]{hodgeshorter}
    Let $\LL$ be a finite first-order language and let $M$ be a countably infinite $\LL$-structure. Then the following are equivalent. 
    \begin{enumerate}[label=(\roman*)]
        \item $M$ is homogeneous and uniformly locally finite.
        \item The theory of $M$ is $\omega$-categorical and has quantifier elimination. 
    \end{enumerate}
\end{fact}

\subsection{Nilpotent groups and Lie algebras}

Let $G$ be a group and let $a,b\in G$. The \textit{commutator} of $a$ and $b$ is the group element $[a,b]=a^{-1}b^{-1}ab$. For two subsets $A,B\subseteq G$ we define $[A,B]$ to be the subgroup generated by all commutators of the form $[a,b]$ where $a\in A$ and $b\in B$. The \textit{derived subgroup} of $G$ is the normal subgroup $G'=[G,G]$. Given a normal subgroup $N$ of $G$, the quotient $G/N$ is abelian if and only if $G'\subseteq N$. We denote the center of the group $G$ by $Z(G)$.  The \textit{lower central series} $(G_n)_{n\geq 1}$ of $G$ is defined as follows: 
    \begin{itemize}
        \item $G_1=G$;
        \item $G_{n+1}=[G_n,G]$ for $n\geq 1$.
    \end{itemize}
The lower central series is a normal series, i.e. each $G_n$ is normal in $G$. Since $G'_n=[G_n,G_n]\subseteq [G_n, G]=G_{n+1}$, it follows that the successive quotients $G_n/G_{n+1}$ are abelian groups. We also have that $G_n/G_{n+1}$ is contained in $Z(G/G_{n+1})$. We have the following containments: 
\[G = G_1 \trianglerighteq G_2 \trianglerighteq G_3\trianglerighteq \ldots\]
\begin{defn}
    A group $G$ is \textit{nilpotent} if its lower central series terminates in the trivial subgroup in finitely many steps. The least integer $c$ such that $G_{c+1}=1$ is called the \textit{nilpotency class} of $G$. For this, we also say $G$ is a $c$\emph{-nilpotent group} or a \emph{nil-}$c$ group for short. 
\end{defn}
Nilpotent groups of class 1 are exactly the abelian groups. Nilpotent groups of class 2 are nonabelian groups where the derived subgroup is contained in the center. In general, if $G$ is a $c\,$-nilpotent group, then $G_c\subseteq Z(G)$.

\begin{defn}
    A \textit{Lie algebra} $L$ over a field $\F$ is a vector space $L$ over $\F$ equipped with a binary operation $[\cdot,\cdot]:L\times L\to L$, called a \textit{Lie bracket}, satisfying the following properties for every $a,b,c\in L$ and $\mu \in \F$:
    \begin{itemize}\itemsep3pt
        \item  $[a,a] = 0$; \hfill (Alternativity)
        \item  $\brac{a+b}{c} = \brac{a}{c}+\brac{b}{c}$,  \hfill (Bilinearity) \\ 
        $\brac{a}{b+c} = \brac{a}{b}+\brac{a}{c}$,\\
        $[\mu a,b]=\mu[a,b]=[a, \mu b]$;
        \item $\brac{a}{\brac{b}{c}}+\brac{b}{\brac{c}{a}}+\brac{c}{\brac{a}{b}}=0$. \hfill (Jacobi identity) 
    \end{itemize}
\end{defn}

It follows that the Lie bracket is antisymmetric, that is, $[a,b]=-[b,a]$ for all $a,b\in L$.  This follows by using alternativity and bilinearity to evaluate $[a+b,a+b]$. Indeed, if the characteristic of $\F$ is not $2$, then alternativity and antisymmetry are equivalent.

We use the following standard notation for iterated brackets: For $n \geq 3$ and elements $x_1, \dots, x_n$ in any Lie algebra $L$, we define $[x_{1}, \ldots, x_{n}]$ inductively by $[x_{1}, \ldots, x_{n}] = [[x_{1}, \ldots, x_{n-1}],x_{n}]$. 

A subspace $U\subseteq L$ is called a \textit{Lie subalgebra} of $L$ if $U$ is closed under the Lie bracket. If $U$ and $V$ are subspaces of $L$, we define $[U,V]$ as the subspace spanned by the elements $[u,v]$ for $u\in U$ and $v\in V$, that is, 
    \[\brac{U}{V} = \{r_1\brac{u_1}{v_1}+\ldots+r_k\brac{u_k}{v_k}\mid k\geq 1, u_i\in U,v_i\in V, r_i\in \F\}.\]
Note that $[U,V]=[V,U]$ by antisymmetry of the bracket. A priori, one does not know whether or not $[U,V]$ is a subalgebra. A subalgebra $I\subseteq L$ is called an \textit{ideal} of $L$ if $\brac{I}{L}\subseteq I$. The bracket in Lie algebras is analogous to the commutator in groups. Consequently, we present the following concepts. The \textit{center} of $L$ is $Z(L) = \{a\in L\mid \brac{a}{b} = 0 \textrm{ for all } b\in L\}$. Also, $L$ is \textit{abelian} if $[a,b]=0$ for all $a,b\in L$. 

The Jacobi identity and antisymmetry further imply the following useful identities. 
\begin{itemize}
    \item[$\circ$] $[a,[b,c]]=[[a,b],c]+[b,[a,c]]$.
    \item[$\circ$] $[[b,c],a]=[[b,a],c]+[b,[c,a]]$.
\end{itemize}

The first bullet point, for example, can be viewed as saying that for $a \in L$, the map $\ad_{a}: L \to L$ defined by $\ad_{a}(x) = [a,x]$ satisfies the Leibniz rule (with respect to `multiplication' given by the Lie bracket). In other words, $\ad_{a}$ is a \textit{derivation} on $L$ in the following sense.

\begin{definition}
    A \textit{derivation} $\delta$ on a Lie algebra $L$ over $\F$ is an $\F$-linear endomorphism $\delta: L\to L$ which satisfies Leibniz' rule: $\delta([a,b]) = [\delta(a),b]+[a,\delta(b)]$. We write $\Der(L)$ for the space of all derivations over $L$.
\end{definition}
 Note that $\Der(L)$ is a vector subspace of the space of $\F$-linear endomorphism of $L$. It is also a Lie algebra for the bracket
\[[\delta,\mu] := \delta\mu - \mu\delta.\]
The map $\ad:L\to \Der(L)$ is a homomorphism of Lie algebras with kernel $Z(L)$, usually called the \textit{adjoint representation}. A derivation of the form $\ad_a$ is called an \textit{inner derivation on $L$}, otherwise it is called an \textit{outer} derivation.

\begin{definition}[Semi-direct product]
    Given two Lie algebras $L_1,L_2$ and a homomorphism $g:L_2\to \Der(L_1)$ we define the semi-direct product $S = L_1\rtimes L_2$ to be the Lie algebra with underlying vector space $L_1\oplus L_2$ and bracket defined as:
     \[[x_1+x_2,y_1+y_2] = [x_1,y_1]_{L_1}+g(x_2)(y_1)-g(y_2)(x_1) + [x_2,y_2]_{L_2}\]
    for $x_1,y_1\in L_1$, $x_2,y_2\in L_2$. 
\end{definition}
 One easily checks that if $S = L_1\rtimes L_2$ then $L_2$ is a subalgebra of $S$ and $L_1$ is an ideal of $S$.

We define inductively the \textit{lower central series} of $L$ as follows:
    \begin{itemize}
        \item $L_1 = L;$
        \item $L_{n+1} = [L_n,L] \textrm{ for } n\geq 1.$
    \end{itemize}
Note that each $L_n$ is an ideal of $L$. A Lie algebra $L$ is \textit{nilpotent of class $c$} if $c$ is the least integer such that
    \[L = L_1\supseteq L_2\supseteq \ldots \supseteq L_c \supseteq L_{c+1} = 0.\]
If $L$ is nilpotent of class $c$, then $L_c \leq Z(L)$. 

\begin{defn}
Let $G$ be any group. A \emph{Lazard series} of length $c$ of $G$ is a sequence of subgroups $G=H_1\unrhd H_2 \unrhd \ldots \unrhd H_{c+1}=1$ such that $[H_i,H_j]\subseteq H_{i+j}$ for all $i,j$. By convention, we set $H_k=1$ for all $k>c$. Accordingly, we say that a sequence of subalgebras $(L_{i})_{1 \leq i \leq c+1}$ is a \emph{Lazard series} of a Lie algebra $L$ if 
$$
L = L_{1} \geq L_{2} \geq \ldots \geq L_{c+1} = 0
$$
and
$$
[L_{i},L_{j}] \leq L_{i+j}
$$
for all $i,j$ (where, as above, we stipulate $L_k = 0$ for all $k > c$).  
\end{defn}

Note that if $G$ is a group with a Lazard series $(H_{i})_{1 \leq i \leq c+1}$, then $G$ must be of nilpotence class at most $c$ (and analogously for Lie algebras).  The lower central series is an example of a Lazard series. 

\begin{defn}
    We define a \emph{Lazard group} $(G, \overline{H})$ to be a group $G$ together with a distinguished Lazard series $\overline{H} = (H_{i})_{1 \leq i \leq c+1}$. Similarly, a \emph{Lazard Lie algebra (LLA)} $(L,\overline{L})$ is a Lie algebra $L$ with a distinguished Lazard series $\overline{L} = (L_{i})_{1 \leq i \leq c+1}$.  We will not always explicitly display the Lazard series $\overline{L}$ when referring to an LLA $(L,\overline{L})$, referring to it instead simply as $L$.  
\end{defn}

\begin{defn}\label{def:levelofelement}
    Let $A$ be an LLA of nilpotency class $\leq c$ with distinguished Lazard series $(A_i)_{1\leq i\leq c+1}$. For any $a\in A$, we define the \textit{level of $a$}, denoted $\lev(a)$ to be the maximal $1\leq i\leq c+1$ such that $a\in A_i$. Equivalently, for $a \neq 0$, $\lev(a)$ is the (unique) $i$ such that $a\in A_i\setminus A_{i+1}$.
\end{defn}
The property $[A_i,A_j]\seq A_{i+j}$ of the Lazard series implies the property $\lev([a,b]) \geq \lev(a)+\lev(b)$. Note that $\lev(a)+\lev(b)$ takes the value $c+1$ as soon as $\lev(a)+\lev(b)\geq c+1$.

\begin{defn}
    Suppose $L$ is an LLA with Lazard series $(L_{i})_{1 \leq i \leq c+1}$. Define $\mathrm{Der}_{\mathrm{Laz}}(L)$ by 
    $$
    \mathrm{Der}_{\mathrm{Laz}}(L) = \{ \delta \in \mathrm{Der}(L) : \delta(L_{i}) \subseteq L_{i+1} \text{ for all } i\}.
    $$
\end{defn}

\begin{lem} \label{Lazard derivations}
    Suppose $L$ is an LLA with Lazard series $(L_{i})_{1 \leq i \leq c+1}$.
\begin{enumerate}
    \item $\mathrm{Der}_{\mathrm{Laz}}(L)$ is a $(c-1)$-nilpotent subalgebra of $\mathrm{Der}(L)$ with Lazard series $(D_{i})_{1 \leq i \leq c}$ defined by 
    $$
    D_{i} = \{\delta \in \mathrm{Der}_{\mathrm{Laz}}(L) : \delta(L_{j}) \subseteq L_{i+j} \text{ for all } j\}.
    $$
    \item If $I \subseteq L$ is an ideal of $L$ (with LLA structure induced from $L$, that is, with Lazard series $(I_{i})_{i}$ defined by $I_{i} = I \cap L_{i}$), then there is an LLA homomorphism $L \to \mathrm{Der}_{\mathrm{Laz}}(I)$ defined by $a \mapsto \mathrm{ad}(a)|_{I}$.
\end{enumerate}
\end{lem}

\begin{proof}
    (1) It is clear from the definitions that we have 
    $$
    \mathrm{Der}_{\mathrm{Laz}}(L) = D_{1} \supseteq D_{2} \supseteq \ldots \supseteq D_{c} = 0.
    $$
    Suppose that $\delta \in D_{i}$ and $\delta' \in D_{j}$.  Let $k$ be arbitrary. Then we have $\delta \delta'(L_{k}) \subseteq \delta(L_{j+k}) \subseteq L_{i+j+k}$ and likewise, $\delta'\delta(L_{k}) \subseteq \delta'(L_{i+k}) \subseteq L_{i+j+k}$, hence 
    $$
    [\delta,\delta'](L_{k}) = (\delta \delta' - \delta' \delta)(L_{k}) \subseteq L_{i+j+k}.
    $$
    This shows that $[D_{i},D_{j}] \subseteq D_{i+j}$. 

    (2) As $\mathrm{ad}: L \to \mathrm{Der}(I)$ is a Lie algebra homomorphism, we only need to show that, for each $i$, $a \in L_{i}$ implies $\mathrm{ad}(a)|_{I} \in D_{i}$.  Fix $i$ and pick $a \in L_{i}$.  Let $j$ and $b \in I_{j}$ be arbitrary.  As $I$ is an ideal, we have $\mathrm{ad}(a)(b) \in I$ and, since $b \in L_{j}$, we have
    $$
    \mathrm{ad}(a)(b) = [a,b] \in L_{i+j}.
    $$
    This shows $\mathrm{ad}(a) \in D_{i}$, as desired. 
\end{proof}

\subsection{Lazard correspondence}\label{subsec:lazardcorrespondence}

Our main technical tool for understanding nilpotent groups of exponent $p$ is the \emph{Lazard correspondence}.  Assuming $c < p$ for an odd prime $p$, this correspondence associates to each nil-$c$ group of exponent $p$ a nil-$c$ Lie algebra over the field $\mathbb{F}_{p}$.\footnote{In fact, the Lazard correspondence is considerably more general than this, associating to every $\mathbb{Q}_{\pi}$-powered nilpotent group a Lie ring over $\mathbb{Q}_{\pi}$ of the same nilpotence class, where $\pi$ is a set of primes $\leq c$.  As this is far more generality than we will need, we refer the interested reader to \cite[Chapter 10]{Khu98}.} From a model-theoretic point of view, this correspondence establishes the uniform bi-definability of nil-$c$ groups of exponent $p$ and of nil-$c$ Lie algebras over $\mathbb{F}_{p}$.  Indeed, this uniform bi-definability applies both to the pure languages of groups and Lie algebras and to their respective expansions to languages with predicates for Lazard series.  This will allow us to conclude that the model-theoretic study of nil-$c$ groups of exponent $p$ reduces entirely to studying Lie algebras, which in turn can be analyzed using the more transparent tools of linear algebra. 

Suppose $c < p$, for $p$ an odd prime.  If $G$ is a group of exponent $p$ and nilpotence class $\leq c$, we define $L_{G}$ to be a structure with same underlying set and operations $+_{L_{G}}$ and $[\cdot,\cdot]_{L_{G}}$ defined by 
$$
g +_{L_{G}} h = h_{1}(g,h) =  gh[g,h]^{-\frac{1}{2}}[g,g,h]^{-\frac{1}{12}}[h,g,h]^{\frac{1}{12}}\ldots 
$$
and 
$$
[g,h]_{L_{G}} = h_{2}(g,h) = [g,h][g,g,h]^{\frac{1}{2}}[h,g,h]^{\frac{1}{2}}\ldots 
$$
where the brackets on the right denote group commutators in $G$. We will usually omit the subscripts. Since the nilpotence class of $G$ is at most $c$, it turns out that both $h_{1}$ and $h_{2}$ are finite products of group commutators in $G$ raised to powers in $\mathbb{Z}_{(p)}$, where $\mathbb{Z}_{(p)}$ denotes the set of $q \in \mathbb{Q}$ such that if $q = \frac{l}{m}$ is in reduced form, then $\mathrm{gcd}(m,p) = 1$. Since the group $G$ is of exponent $p$, it makes sense to raise any element to powers in $\mathbb{Z}_{(p)}$ and these yield well-defined operations on $L_{G}$.   The coefficients of $h_{1}$ and $h_{2}$ are explicitly described in \cite{Cicaloetal}.  

Conversely, given a Lie algebra $L$ over $\mathbb{F}_{p}$ of nilpotence class at most $c$, one defines $G_{L}$ to be the structure with the same underlying set and with a binary operation $*_{G_{L}}$ defined by 
$$
a *_{G_{L}} b = H(a,b) =  a + b + \frac{1}{2} [a,b] + \frac{1}{12}[a,a,b] - \frac{1}{12} [b,a,b] + \ldots 
$$
This is the Baker-Campbell-Hausdorff formula, where the brackets on the right-hand side are the Lie bracket of $L$. This is usually an infinite sum but, since $L$ is of nilpotence class  at most $c$, the function $H$ can be written as a finite linear combination of Lie monomials with coefficients in $\mathbb{Z}_{(p)}$.  This can thus be viewed as an $\mathbb{F}_{p}$-linear combination of Lie monomials.  

The following fact summarizes the Lazard correspondence.

\begin{fact} \cite[Chapter 10]{Khu98}
Suppose $c < p$ for an odd prime $p$.  To every group $G$ of exponent $p$ and nilpotence class $\leq c$, the Lazard correspondence associates a Lie algebra $L_{G}$ over $\mathbb{F}_{p}$ with the same underlying set $L_{G} = G$ and with operations $a+b = h_{1}(a,b)$ and $[a,b] = h_{2}(a,b)$, and $ra = a^{r}$ for every $r \in \mathbb{F}_{p}$.  Conversely, for every Lie algebra $L$ over $\mathbb{F}_{p}$ of nilpotence class $\leq c$, there is a corresponding group $G_{L}$ of exponent $p$ with the same underlying set and group operation defined by 
$$
a * b = H(a,b)
$$
and $a^{r} = ra$ for $r \in \mathbb{F}_{p}$. These operations are inverses to each other:  as Lie algebras over $\mathbb{F}_{p}$, we have $L_{G_{L}} = L$ and additionally $G_{L_{G}} = G$ as groups. 
\end{fact}

The following summarizes the key facts that we need about the Lazard correspondence.

\begin{fact} \label{Lazard facts} \cite[Chapter 10]{Khu98}
    Suppose $c < p$ for an odd prime $p$. Suppose that $L$ is a Lie algebra over $\mathbb{F}_{p}$ of nilpotence class $\leq c$, that $G$ is a group of nilpotence class $\leq c$ of exponent $p$, and that $L$ and $G$ are in correspondence with one another, i.e. $L = L_{G}$ as Lie algebras and $G = G_{L}$ as groups.
    \begin{enumerate}
        \item For all $a,b \in L$, 
        $$
        [a,b]_{L} = [a,b]_{G} \prod_{j} \chi_{j}^{s_{j}}
        $$
        where $s_{j} \in \mathbb{Z}_{(p)}$ and $\chi_{j}$ are group commutators in $a$ and $b$ of degree $\geq 3$.
        \item For all $a,b \in G$, 
        $$
        [a,b]_{G} = [a,b]_{L} + \sum_{j} u_{j} \chi_{j}
        $$
        where $u_{j} \in \mathbb{F}_{p}$ and the $\chi_{j}$ are Lie monomials in $a$ and $b$ of degree $\geq 3$. 
        \item A subset $K \subseteq G$ is a subgroup of $G$ if and only if $K \subseteq L$ is a Lie subalgebra. 
        \item A subset $I$ is a normal subgroup of $G$ if and only if $I$ is an ideal of $L$.
        \item A function from the underlying set $G = L$ to itself is an endomorphism of the group $G$ if and only if it is an endomorphism of the Lie algebra $L$. In particular, the automorphism groups $\mathrm{Aut}(L)$ and $\mathrm{Aut}(G)$ coincide as permutation groups. 
    \end{enumerate}
\end{fact}

Note that Fact \ref{Lazard facts} implies if $L$ is a Lie algebra over $\mathbb{F}_{p}$ of nilpotence class $\leq c$, the group $G$ is of nilpotence class $\leq c$ of exponent $p$, and $L$ and $G$ are in correspondence with one another, then a sequence $(H_{i})_{1 \leq i \leq c+1}$ is a Lazard series for $G$ if and only if it is a Lazard series for $L$.

\section{The generic nilpotent group of class 2} \label{nil-2}

\subsection{The bilinear-map correspondence}

We recall the following definitions from Baudisch \cite{Baudisch2}.
\begin{defn}\label{def:baudischdefvarietiesgroupsetc}
Let $p$ be a prime number.
\begin{enumerate}
\item We write $\mathbb{G}_{2,p}$ for the class of nilpotent groups of class 2 and exponent $p$.  We let $\mathbb{K}_{2,p}$ denote the finite groups in $\mathbb{G}_{2,p}$.
\item $\LL_{P}$ is the language of groups together with a unary predicate $P$.  We write $\mathbb{G}^{P}_{2,p}$ for those $\LL_{P}$-structures $G$ whose reduct to the language of groups lies in $\mathbb{G}_{2,p}$ and in which additionally $P(G)$ is a normal subgroup satisfying $[G,G] \subseteq P(G) \subseteq Z(G)$.  Likewise, we write $\mathbb{K}^{P}_{2,p}$ for the finite structures in this class.  When $p$ is understood from context, we will simply write $\mathbb{G}^{P}$ and $\mathbb{K}^{P}$.  
\item We write $\mathbb{B}_{p}$ for the class of pairs of $\mathbb{F}_{p}$-vector spaces $(V,W)$, viewed as $\mathcal{L}_{B}$-structures, where $\mathcal{L}_{B}$ contains a sort for each vector space (and the abelian group structure on each) together with an alternating bilinear map $\beta : V \times V \to W$. As in (2), we just write $\mathbb{B}$ when $p$ is understood from context. 
\end{enumerate}
\end{defn}

Baudisch defines a functor $F : \mathbb{G}^{P} \to \mathbb{B}$, which is defined by
$$
F(G) = (G/P(G), P(G), [\cdot,\cdot]) 
$$
 for all $G \in \mathbb{G}^{P}$, where $[\cdot,\cdot]$ is the commutator in $G$.  Given an embedding $f : G \to H$ of structures in $\mathbb{G}^{P}$, we define $F(f) (= (F(f)_{V},F(f)_{W}))$ to be the pair of maps $(\overline{f}, f|_{P(G)})$, where $\overline{f} : G/P(G) \to H/P(H)$ is the induced embedding and $f|_{P(G)}$ is the restriction of $f$ to $P(G)$.  
%

Baudisch deduces that both $\mathbb{K}^{P}$ and the class of finite structures in $\mathbb{B}$ are Fra\"iss\'e classes with quantifier elimination \cite[Corollary 1.3]{Baudisch2}.  Let $\mathbf{G}$ and $\mathbf{B}$ be their respective Fra\"iss\'e limits.  It can be shown that in $\mathbf{G}$, we have $P(\mathbf{G}) = Z(\mathbf{G})$ \cite[Corollary 1.3]{Baudisch2}.  Let $T_{\mathbf{G}} = \mathrm{Th}(\mathbf{G})$ and $T_{\mathbf{B}} = \mathrm{Th}(\mathbf{B})$.  Let $\mathbb{M}_{\mathbf{G}} \models T_{\mathbf{G}}$ and $\mathbb{M}_{\mathbf{B}} = (\mathbb{V},  \mathbb{W}, \beta) \models T_{\mathbf{B}}$ be their respective monster models.  Note that we may view $\mathbb{M}_{\mathbf{B}}$ as a structure interpreted in $\mathbb{M}_{\mathbf{G}}$ with $\mathbb{V} = \mathbb{M}_{\mathbb{G}}/Z(\mathbb{M}_{\mathbf{G}})$, $\mathbb{W} = Z(\mathbb{M}_{\mathbf{G}})$, and $\beta(\cdot, \cdot) = [\cdot, \cdot]$.  By abuse of notation, if $A \subseteq \mathbb{M}_{\mathbf{G}}$ is a substructure, we will write $F(A)$ to denote the image of $A$ under this interpretation, i.e. identifying $F(A)$ with a substructure of $\mathbb{M}_{\mathbf{B}}$ is an obvious way. 

Given $B = (V,W,\beta) \in \mathbb{B}$, we fix a basis $\overline{b} = \{b_{i} :  i < \alpha\}$ for $V$ and define a group $G(\overline{b},B)$ whose underlying set consists of $V \times W$ with multiplication defined by 
$$
\left(\sum_{i < \alpha} r_{i}b_{i}\,,\,w\right) \cdot \left(\sum_{i < \alpha} s_{i}b_{i}\,,\, w'\right) = \left(\sum_{i < \alpha} (r_{i} + s_{i})b_{i}\,,\, w + w' + \sum_{i < j < \alpha} r_{i} s_{j} \beta(b_{j},b_{i}) \right),
$$
where, in the above expressions, all but finitely many $r_{i}$ and $s_{i}$ are zero. Note that if $B = F(H)$ for some $H \in \mathbb{G}_{p}$, we have $G(\overline{b},F(H)) \cong H$.  More explicitly, if $h \in H$, then, since $\overline{b}$ is a basis of $H/P(H)$, we can pick $c_{i} \in H$ such that $c_{i}P(H) = b_{i}$ for each $i < \alpha$.  Then we have that $hP(H) = \prod_{i < \alpha} c_{i}^{r_{i}}P(H)$ for some $0 \leq r_{i} < p$ for each $i < \alpha$. It follows that $h = \prod_{i < \alpha} c_{i}^{r_{i}}w$ for some $w \in P(H)$. The map $h \mapsto \left( \sum_{i < \alpha} r_{i}b_{i},w\right)$ is an isomorphism from $H$ to $G(\overline{b},B)$.  

\begin{lem} \label{acl description}
Given a set $A$ of parameters in $\mathbb{M}_{\mathbf{B}}$, 
$$
\mathrm{acl}(A) = \mathrm{dcl}(A) = \mathrm{span}_{V}(V(A)) \cup \mathrm{span}_{W}(W(A) \cup \beta(V(A)^{2})).  
$$
\end{lem}

\begin{proof}
Let $S =  \mathrm{span}_{V}(V(A)) \cup \mathrm{span}_{W}(W(A) \cup \beta(V(A)^{2}))$.  As the other containments are clear, it suffices to show $\mathrm{acl}(A) \subseteq S$.  Pick $u \in \mathbb{M}_{\mathbf{B}} \setminus S$.  We will show there are infinitely many pairwise distinct $u_{i}$ with $u_{i} \equiv_{S} u$.  Since, in particular, $u_{i} \equiv_{A} u$, it follows that $u \not\in \mathrm{acl}(A)$.  

\textbf{Case 1}:  Assume $u \in V$.  Introduce distinct new elements $(u_{i})_{i < \omega}$ and consider the vector space $V'$ spanned by $V(A)$ and $\{u_{i} : i < \omega\}$, with the $u_{i}$ linearly independent over $V(A)$.  We define $\beta'$ by setting $\beta'(u_{i},b) = \beta(u,b)$ for all $i < \omega$ and $b \in V(S)$, and $\beta'(u_{i},u_{j}) = 0$ for all $i, j < \omega$.  This determines a unique alternating bilinear map $\beta'$ on $V'$.  Embedding over $S$, we may assume $(V', W(S), \beta')$ is a substructure of $(\mathbb{V},  \mathbb{W}, \beta)$.  By construction, the function $u \mapsto u_{i}$ extends to an isomorphism
$$
\sigma_{i} : (\langle V(S) u \rangle, W(S), \beta|_{\langle V(S) u \rangle}) \to (\langle V(S)u_{i}), W(S), \beta'),
$$
which fixes $S$ pointwise.  By quantifier elimination, the $\sigma_{i}$ witness that $u_{i} \equiv_{S} u$ for all $i$.

\textbf{Case 2}:  If $u \in W$, we can just take infinitely many new elements $(u_{i})_{i < \omega}$ and define $W'$ to be a vector space spanned by $W(S) \cup \{u_{i} : i < \omega\}$ with the $u_{i}$ linearly independent over $W(S)$.  Then in an obvious way, we have 
$$
(V(S), W(S), \beta|_{V(S)}) \subseteq (V(S), W', \beta|_{V(S)}),
$$
so embedding $(V(S), W', \beta|_{V(S)})$ into $(\mathbb{V}, \mathbb{W}, \beta)$, we see that $u_{i} \equiv_{S} u$ for all $i < \omega$.  
\end{proof}

\subsection{NSOP$_1$}

In this subsection, we will establish that $T_{\mathbf{G}}$ is NSOP$_{1}$ by characterizing Kim-independence in this theory. See \cite{KR20} for the basis of the theory. 
 We will actually first establish that $T_{\mathbf{B}}$ is NSOP$_{1}$ by establishing the independence theorem for algebraic independence and applying the NSOP$_{1}$ Kim-Pillay theorem.  We will then use Baudisch's functor $F$ to deduce that $T_{\mathbf{G}}$ is  NSOP$_{1}$ as well. Recall that algebraic independence, denoted by $a \indi{a}_{C} b$, means $\mathrm{acl}(aC) \cap \mathrm{acl}(bC) = \mathrm{acl}(C)$.  Note that, by \cite[Proposition 2.1]{Baudisch1}, $T_{\mathbf{G}}$ is not simple so NSOP$_{1}$ is, in some sense, best possible. 

\begin{lem}[Independence Theorem] \label{independence lemma}
Suppose we are given small subsets $A, B, C_{0}, C_{1}, D \subseteq \mathbb{M}_{\mathbf{B}}$ such that $A \indi{a}_{D} B$, $C_{0} \indi{a}_{D} A$, $C_{1} \indi{a}_{D} B$, and $C_{0} \equiv_{D} C_{1}$.  Then there is $C_{*}$ such that $C_{*} \equiv_{AD} C_{0}$, $C_{*} \equiv_{BD} C_{1}$, and $C_{*} \indi{a}_{D} AB$.  
\end{lem}

\begin{proof}
We may assume $A = \mathrm{acl}(AD)$, $B = \mathrm{acl}(BD)$, $C_{0} = \mathrm{acl}(C_{0}D)$, $C_{1} = \mathrm{acl}(C_{1}D)$, and $D = \mathrm{acl}(D)$.  Moreover, applying extension for $\ind^{a}$, we may assume that $A$ and $B$ are models of $T_{\mathbf{B}}$. Let $E = \mathrm{acl}(AB)$.  So in particular, $D$ is an algebraically closed subset of all of the given sets.  Choose $X$ to be a set that is linearly independent over $W(A) \cup W(C_{0})$ such that
$$
W(\mathrm{acl}(AC_{0}))  = \mathrm{span}(W(A) \cup W(C_{0}) \cup X).  
$$
Likewise choose $Y$ to be linearly independent over $W(B) \cup W(C_{1})$ and such that 
$$
W(\mathrm{acl}(BC_{1})) = \mathrm{span}(W(B) \cup W(C_{1}) \cup Y).  
$$
Fix $\sigma \in \mathrm{Aut}(\mathbb{M}_{\mathbf{B}}/D)$ with $\sigma(C_{0}) = C_{1}$.  

Work briefly in the reduct $(\mathbb{V},\mathbb{W})$ consisting of a disjoint union of two infinite dimensional $\mathbb{F}_{p}$-vector spaces, which is clearly stable.  Choose $C_{*}$ which realizes the unique non-forking extension of $\text{tp}(C_{0}/A)$ to $E$.  By transitivity and stationarity, we have that $C_{*}$ realizes the unique non-forking extension of $\text{tp}(C_{1}/B)$ as well.  Note that in this reduct, we have $X$ independent from $C_{0}A$ over $D$ and $Y$ independent from $C_{1}B$ over $D$.  Pick $X_{*}$ such that $X_{*}C_{*}$ has the same type (in the reduct) as $XC_{0}$ over $A$ and $Y_{*}$ such that $Y_{*}C_{*}$ has the same type as $YC_{1}$ over $B$.  By invariance, $Y_{*}$ is independent over $D$ from $C_{*}B$ so, by extension, we may assume $Y_{*}$ is independent from $C_{*}X_{*}E$.  There are isomorphisms (in the reduct language) $\tau_{0} : (V(\mathrm{acl}(AC_{0})),W(\mathrm{acl}(AC_{0}))) \to (\mathrm{span}(V(A)V(C_{*})), \mathrm{span}(W(C_{*})X_{*}W(A)))$ over $A$ and also $\tau_{1} : (V(\mathrm{acl}(BC_{1})),W(\mathrm{acl}(BC_{1}))) \to (\mathrm{span}(V(B)V(C_{*})), \mathrm{span}(W(C_{*})Y_{*}W(B)))$ over $B$ such that $\tau_{1} \circ \sigma|_{C_{0}} = \tau_{0}|_{C_{0}}$.  

Let $\beta_{0}$ be the alternating bilinear map on $\mathrm{span}(V(A)V(C_{*}))$ defined by pushing forward $\beta$ along $\tau_{0}$.  In other words, we define 
$$
\beta_{0}(v,w) = \tau_{0}(\beta(\tau_{0}^{-1}(v), \tau^{-1}_{0}(w))), 
$$
for all $v,w \in \mathrm{span}(V(A)V(C_{*}))$.  Likewise, define $\beta_{1}$ on $\mathrm{span}(V(B)V(C_{*}))$ by pushing forward $\beta$ along $\tau_{1}$.  

We claim that there is a unique alternating bilinear map on $\mathrm{span}(V(C_{*})V(E))$ extending $\beta_{0}$, $\beta_{1}$, and $\beta|_{V(E)}$.  First, note that if $v,w \in \mathrm{span}(V(A)V(C_{*})) \cap V(E)$, then since $C_{*}$ is independent with $E$ over $A$, it follows that $v,w \in V(A)$.  Since $\tau_{0}$ fixes $A$ pointwise and $A = \mathrm{acl}(A)$ (and therefore is closed under $\beta$), we have 
\begin{eqnarray*}
\beta_{0}(v,w) &=& \tau_{0}(\beta(\tau_{0}^{-1}(v), \tau^{-1}_{0}(w))) \\
&=& \tau_{0}(\beta(v,w)) \\
&=& \beta(v,w),
\end{eqnarray*}
so $\beta_{0}$ and $\beta|_{V(E)}$ agree on their common domain.  A symmetric argument shows that $\beta_{1}$ and $\beta|_{V(E)}$ agree on their common domain, using now that $C_{*}$ is independent with $E$ over $B$.  Finally, suppose 
$$
v,w \in \mathrm{span}(V(A)V(C_{*})) \cap \mathrm{span}(V(B)V(C_{*})).
$$
We know that, in the reduct, $C_{*}$ is independent from $AB$ over $D$ and $A$ is independent from $B$ over $D$, so by base monotonicity and transitivity, it follows that $A$ and $B$ are independent over $C_{*}$.  It follows, then, that $v,w \in V(C_{*})$.  Since $\tau_{0}|_{C_{*}} = \tau_{1} \circ \sigma|_{C_{*}}$ and $\sigma$ preserves $\beta$, we have 
\begin{eqnarray*}
\beta_{0}(v,w) &=& \tau_{0}(\beta(\tau_{0}^{-1}(v), \tau_{0}^{-1}(w))) \\
&=& (\tau_{1} \circ \sigma)(\beta((\sigma^{-1} \circ \tau_{1}^{-1})(v), (\sigma^{-1} \circ \tau_{1}^{-1})(w))) \\
&=& \tau_{1} ( \beta(\tau_{1}^{-1}(v), \tau_{1}^{-1}(w))) \\
&=& \beta_{1}(v,w),
\end{eqnarray*}
so $\beta_{0}$ and $\beta_{1}$ agree on the intersection of their domains.  It follows that the union of $\beta_{0}$, $\beta_{1}$, and $\beta|_{V(E)}$ determines an alternating form on $\mathrm{span}(V(C_{*})V(E))$.  After embedding over $E$ into $\mathbb{M}_{\mathbf{G}}$, we may assume that the structure we have constructed is a substructure of $\mathbb{M}_\mathbf{B}$.  By quantifier elimination, the isomorphisms $\tau_{0}$ and $\tau_{1}$ witness that $C_{*} \equiv_{A} C_{0}$ and $C_{*} \equiv_{B} C_{1}$.

We are left with showing that $C_{*} \indi{a}_{D} AB$.  However, by construction, we know $V(C_{*})$ is independent from $V(E)$ over $V(D)$ and $W(C_{*})$ is independent from $W(E)$ over $W(D)$.  Therefore $C_{*} \cap E = D$ which entails $C_{*} \indi{a}_{D} E$ and therefore $C_{*} \indi{a}_{D} AB$.  
\end{proof}

\begin{defn}
Suppose $A,B,C \subseteq \mathbb{M}_{\mathbf{G}}$ and denote by $Z$ the center of $\mathbb{M}_{\mathbf{G}}$.  We write $A \indi{*}_{C} B$ to indicate that the following hold:
\begin{enumerate}
\item $\langle A C \rangle \cap \langle B C \rangle = \langle C \rangle$. 
\item $\langle A/Z \rangle \cap \langle B/Z \rangle = \langle C /Z \rangle$, where $\langle A/Z \rangle$ denotes the subgroup of $\mathbb{M}_{\mathbf{G}}/Z$ generated by the cosets represented by elements of $A$. 
\end{enumerate}
\end{defn}

\begin{lem} \label{independence characterization}
Suppose $A,B,C \subseteq \mathbb{M}_{\mathbf{G}}$.  Then 
$$
A \indi{*}_{C} B \iff F(\langle AC \rangle) \indi{a}_{F(\langle C \rangle)}  F(\langle BC \rangle).  
$$
\end{lem}

\begin{proof}
We prove that the right hand side implies $(1)$ of the definition of $\indi *$, the rest is immediate from Lemma \ref{acl description}.  Assume that $F(\vect{AC})\cap F(\vect{BC}) = F(\vect{C})$ and let $x\in \vect{AC}\cap \vect{BC}$. As $(\vect{AC}/Z)\cap (\vect{BC}/Z) = \vect{C}/Z$, there exists $c\in \vect{C}$ such that $x-c\in Z$. Now $x-c\in Z(\vect{AC})\cap Z(\vect{BC})$ and the latter equals $Z(\vect{C})$ since $F(\vect{AC})\cap F(\vect{BC}) = F(\vect{C})$. So $x-c\in Z(\vect{C})$ hence $x\in \vect{C}$. The other inclusion being trivial, we have $\vect{AC}\cap \vect{BC} = \vect{C}$.
\end{proof}

\begin{thm} \label{star independence theorem}
Suppose $A,B,C_{0},C_{1},D \subseteq \mathbb{M}_{\mathbf{G}}$ satisfy $A \indi{*}_{D} C_{0}$, $B \indi{*}_{D} C_{1}$, $A \indi{*}_{D} B$, and $C_{0} \equiv_{D} C_{1}$, then there is $C_{*}$ with $C_{*} \equiv_{AD} C_{0}$, $C_{*} \equiv_{BD} C_{1}$, and $C_{*} \indi{*} AB$. 
\end{thm}

\begin{proof}
We may assume $A = \langle AD \rangle$, $B = \langle BD \rangle$, $C_{i} = \langle C_{i}D \rangle$ for $i =0,1$, and $D = \langle D \rangle$.  Fix $\varphi : C_{0} \to C_{1}$, an isomorphism over $D$ that witnesses that $C_{0} \equiv_{D} C_{1}$.  By Lemma \ref{independence characterization}, it follows that $F(A) \indi{a}_{F(D)} F(C_{0})$, $F(B) \indi{a}_{F(D)} F(C_{1})$, and $F(A) \indi{a}_{F(D)} F(B)$.  Moreover, we have $F(C_{0}) \equiv_{F(D)} F(C_{1})$, witnessed by the $F(D)$-isomorphism $F(\varphi)$.  

Fix a basis $\overline{d}$ for $V(F(D))$, then extend this by $\overline{a}$ to a basis for $V(F(A))$, by $\overline{b}$ to a basis $\overline{b}\overline{d}$ of $V(F(B))$, and by $\overline{c}_{0}$ to a basis $\overline{c}_{0}\overline{d}$ of $V(F(C_{0}))$.  Then setting $\overline{c}_{1} = F(\varphi)(\overline{c}_{0})$, we see that $\overline{c}_{1}\overline{d}$ is a basis of $V(F(C_{1}))$.  By independence, $\overline{d}\overline{a}\overline{c}_{0}$ is linearly independent so can be extended by $\overline{e}_{0}$ to a basis $\overline{d}\overline{a}\overline{c}_{0}\overline{e}_{0}$ for $V(F(\langle A C_{0} \rangle))$ and, likewise, we can find $\overline{e}_{1}$ such that $\overline{d}\overline{b}\overline{c}_{1}\overline{e}_{1}$ is a basis for $V(F(\langle BC_{1} \rangle))$, and we can find $\overline{f}$ such that $\overline{d}\overline{a}\overline{b}\overline{f}$ is a basis for $V(F(\langle A B \rangle))$.  For each $X \in \{A,B,C_{0},C_{1}, D, \langle A C_{0} \rangle, \langle BC_{1} \rangle, \langle AB \rangle\}$, we may identify the group $X$ with $G(\overline{x},F(X))$, where $\overline{x}$ is the distinguished basis for $X$ described above.  Note that, with this identification, 
$$
\varphi\left( \left( \sum_{\alpha} r_{\alpha} c_{\alpha},w \right) \right) = \left( \sum_{\alpha} r_{\alpha} F(\varphi)_{V}(c_{\alpha}), F(\varphi)_{W}(w)\right),
$$
where $\alpha$ ranges over the indices of $\overline{c}$ and all but finitely many $r_{\alpha} \in \mathbb{F}_{p}$ are equal to zero.  

Now we apply Lemma \ref{independence lemma} to obtain some $F_{*}$ such that $F_{*} \equiv_{F(A)} F(C_{0})$, $F_{*} \equiv_{F(B)} F(C_{1})$, and $F_{*} \indi{a}_{F(D)} F(A)F(B)$.  Let $f_{0} : F(\langle AC_{0} \rangle) \to \langle F(A) F_{*} \rangle$ be an isomorphism over $F(A)$ and let $f_{1} : F(\langle BC_{1} \rangle) \to \langle F(B) F_{*} \rangle$ be an isomorphism over $F(B)$ such that 
$$
f_{0}|_{F(C_{0})} = (f_{1} \circ F(\varphi))|_{F(C_{0})}.  
$$  
Let $\overline{c}_{*} = f_{0}(\overline{c}_{0}) = f_{1}(\overline{c}_{1})$.  Define the group $C_{*}$ as $G(\overline{c}_{*},F_{*})$.  Note that if $f_{0} = (f_{0,V}, f_{0,W})$, then we have $f_{0,V}(\overline{a}\overline{c}_{0}\overline{d}\overline{e}_{0}) = \overline{a} \overline{c}_{*} \overline{d}\overline{e}'_{0}$ for some $\overline{e}'_{0}$.  Then, we can define a map $\varphi_{0}: G(\overline{a}\overline{c}_{0}\overline{d}\overline{e}_{0},F(\langle A C_{0} \rangle)) \to G(\overline{a} \overline{c}_{*} \overline{d}\overline{e}'_{0},\langle F(A) F_{*} \rangle)$ by 
$$
\left( \sum_{\alpha} r_{\alpha} x_{\alpha}, w\right) \mapsto \left(\sum_{\alpha} r_{\alpha}f_{V}(x_{\alpha}), f_{W}(w)\right),
$$
where $x_{\alpha}$ ranges over elements of the basis $\overline{a}\overline{c}_{0}\overline{d}\overline{e}_{0}$ and all but finitely many $r_{\alpha}$ are $0$.  Note that this defines an isomorphism and $F(\varphi_{0}) = f_{0}$.  Similarly, we may define an isomorphism $\varphi_{1} : G(\overline{b}\overline{c}_{1}\overline{d}\overline{e}_{1},F(\langle B C_{1} \rangle)) \to G(\overline{b} \overline{c}_{*} \overline{d}\overline{e}'_{1},\langle F(B) F_{*} \rangle)$.  Note that $\varphi_{0}|_{G(\overline{d},\overline{c}_{0}, C_{0})} = (\varphi_{1} \circ \varphi)|_{G(\overline{d}, \overline{c}_{0},C_{0})}$.  By quantifier elimination, we may assume that all these structures are embedded in $\mathbb{M}_{\mathbf{G}}$ over $\langle A B \rangle$.  Then we have $C_{*} \equiv_{A} C_{0}$, $C_{*} \equiv_{B} C_{1}$, and $C_{*} \indi{*}_{D} AB$, by Lemma \ref{acl description}.  
\end{proof}

By \cite[Proposition 5.3]{CR16}, there is an Kim-Pillay-style criterion for NSOP$_{1}$, allowing one to show that a theory is NSOP$_{1}$ by proving the existence of a well-behaved independence relation. The variant of this theorem in \cite[Theorem 6.11]{kaplan2021transitivity}, moreover, allows one to conclude that this independence relation must additionally correspond to Kim-independence.  For the definitions of the relevant properties of an independence relation, we encourage the reader to consult \cite{kaplan2021transitivity}.

\begin{cor} \label{nsop1 conclusion}
The theory $T_{\mathbf{G}}$ is NSOP$_{1}$ and $\indi{*} = \indi{K}$ over models.  
\end{cor}

\begin{proof}
It is easy to check that $\ind^{*}$ is invariant and satisfies strong finite character, symmetry, monotonicity, and existence over models. The independence theorem is established in Theorem \ref{star independence theorem}.  Finally, to show witnessing, suppose $M \models T_{\mathbf{G}}$ and $a \nind^{*}_{M} b$.  Take a coheir sequence $(b_{i})_{i < \omega}$ over $M$ with $b_{0} = b$.  Suppose we are given terms $t,s$ and tuples $m,m' \in M$. If $d = t(a,m) = s(b,m')$ and $d \not\in M$, then, as $(b_{i})_{i < \omega}$ is a coheir sequence over $M$ and $s(b,m') \not\in M$, we must have that the $s(b_{i},m')$ are pairwise distinct, so $\{t(x,m) = s(b_{i},m') : i < \omega\}$ is inconsistent. Next, consider the case that we have an equality $\overline{d} = t(a,m)Z = s(b,m')Z$ of cosets of $Z$ with $\overline{d} \not\in M/Z$.  As before, since $(b_{i})_{i < \omega}$ is a coheir sequence over $M$, we must have that the cosets $s(b_{i},m')Z$ are pairwise distinct:  if $s(b_{1},m')Z = s(b_{0},m')Z$, then the formula $(s(y,m'))^{-1}s(b_{0},m') \in Z$ must have a realization in $M$, hence $s(b_{0},m')Z = \overline{d} \in M/Z$, against our assumption. This shows that $\{(t(x,m))^{-1}s(b_{i},m') \in Z : i < \omega\}$ is inconsistent.  Together these show witnessing, so we conclude that $\ind^{*} = \ind^{K}$ over $M$.    
\end{proof}

\begin{rem}
    It is straightforward to modify the previous arguments to show that $T_{\mathbf{G}}$ satisfies the existence axiom for $\ind^{*}$ and that $\ind^{*} = \ind^{K}$ over all sets, using the variant of the NSOP$_{1}$ Kim-Pillay theorem in \cite[Theorem 6.1]{chernikov2020transitivity}. As we will not use this later, we leave this extension to the reader for brevity's sake. 
\end{rem}

\begin{cor}
    The theory $T_{\mathbf{B}}$ is NSOP$_{1}$ and $\ind^{K} = \ind^{a}$ over models. 
\end{cor}

\begin{proof}
    This follows from Lemma \ref{independence characterization} and Corollary \ref{nsop1 conclusion}, using the intepretation of $T_{\mathbf{B}}$ in $T_{\mathbf{G}}$.  
\end{proof}

\subsection{2-dependence}
In this subsection, we will give a rapid argument that $T_{\mathbf{B}}$ is $2$-dependent.  Although the $2$-dependence of $T_{\mathbf{G}}$ can be deduced from this, the fact that the interpretation of $T_{\mathbf{B}}$ in $T_{\mathbf{G}}$ is not a bi-interpretation makes this route to proving the $2$-dependence of $T_{\mathbf{G}}$ excessively cumbersome.  We will prove the $2$-dependence of $T_{\mathbf{G}}$ later, in Corollary \ref{groupc-dependence}, as a consequence of a more general result about the $c$-dependence of $c$-nilpotent Lie algebras, using the Lazard correspondence. 
\begin{defn} \label{IPk def}
 A formula $\varphi(x;y_{0}, \ldots, y_{k-1})$ is said to have the $k$\emph{-independence property} (or IP$_{k}$) if there are $(a_{0,i}, \ldots, a_{k-1,i})_{i < \omega}$ and $(b_{X})_{X \subseteq \omega^{k}}$ such that, for all $X \subseteq \omega^{k}$,  
 $$
 \models \varphi(b_{X};a_{0,i_{0}}, \ldots, a_{k-1,i_{k-1}}) \iff (i_{0}, \ldots, i_{k-1}) \in X.  
 $$
 We say a theory $T$ is $k$\emph{-dependent} (or NIP$_{k}$) if no formula has IP$_{k}$ modulo $T$.
\end{defn}

Note that IP$_{1}$ is exactly the usual independence property and the $1$-dependent theories are exactly the NIP theories. 

\begin{fact} \label{2dep composition lemma} \cite[Theorem 5.12]{chernikov2021n}
Let $M$ be an $\LL'$-structure such that its reduct to a language $\LL \subseteq \LL'$ is NIP.  Let $\varphi(x_{1}, \ldots, x_{d})$ be an $\LL$-formula. For each $i \in [d]$, fix some $s_{i} < t_{i} \in [3]$ and let $f_{i} : M_{y_{s_{i}}} \times M_{y_{t_{i}}} \to M_{x_{i}}$ be an arbitrary binary function.  Then the formula 
$$
\psi(y_{1};y_{2},y_{3}) = \varphi(f_{1}(y_{s_{1}},y_{t_{1}}), \ldots, f_{d}(y_{s_{d}},y_{t_{d}}))
$$
is $2$-dependent. 
\end{fact}

\begin{lem}\label{lm:TB2dep}
    The theory $T_{\mathbf{B}}$ is $2$-dependent.  
\end{lem}

\begin{proof}
    Let $T_{-}$ denote the reduct of $T_{\mathbf{B}}$ to the language $\LL_{-}$ consisting of the sorts $V$ and $W$ and the abelian group structure on each, but forgetting the bilinear map. This is the theory of two disjoint copies of an infinite dimensional $\mathbb{F}_{p}$-vector space which is interpretable in an $\mathbb{F}_{p}$-vector space and is therefore stable (even $\omega$-stable).  By quantifier elimination, every $\LL_{B}$-formula $\varphi(x_{1}, \ldots, x_{n}, y_{1}, \ldots, y_{m})$, where the $x_{i}$ are in the sort $V$ and the $y_{i}$ are in the sort $W$, can be written in the form 
$$
\psi(x_{1}, \ldots, x_{n}, y_{1}, \ldots, y_{m}, (\beta(x_{i},x_{j}) : i < j))
$$
where $\psi(\overline{x}, \overline{y}, \overline{z})$ is a (stable) $\LL_{-}$-formula.  By Fact \ref{2dep composition lemma}, we can conclude that $T_{\mathbf{B}}$ is $2$-dependent. 
\end{proof}

The following was pointed out to us by Gabe Conant: 

\begin{remark}
In \cite{terry2022irregular}, Terry and Wolf defined a new classification hierarchy called NFOP$_k$, which was subsequently developed by Abd Aldaim, Conant and Terry in \cite{abdaldaim2023higher}. NFOP$_k$ is a promising candidate to be a $k$-ary extension of stability, the same way NIP$_k$ is a $k$-ary extension of NIP.
Using \cite[Theorem 2.16]{abdaldaim2023higher} instead of Fact \ref{2dep composition lemma} in the proof of Lemma \ref{lm:TB2dep}, we immediately conclude that $T_{\mathbf{B}}$ is NFOP$_2$.  Our proof of the $2$-dependence of $T_{\mathbf{G}}$ also gives that $T_{\mathbf{G}}$ is NFOP$_{2}$. 
\end{remark}

\section{Free amalgamation of Lie algebras} \label{sec:freeamalgamation}

\subsection{Stages of the construction}

In this section, we fix a natural number $c$, a prime number $p>c$ and a field $\F$. We extend the notations from Definition \ref{def:baudischdefvarietiesgroupsetc}. Let $\LL_{c,\F}$ be the (one-sorted) language of $\F$-vector spaces $\set{+,-,0, (\lambda\cdot)_{\lambda\in \F}}$ expanded by a binary function symbol $[\cdot,\cdot]$ and predicates $(P_i)_{1\leq i\leq c+1}$. Let $\LL_{c}$ be the reduct of $\LL_{c,\F}$ omitting the functions $(\lambda\cdot)_{\lambda\in \F}$.

\begin{defn}
\begin{enumerate}
\item Let $\Lbb_{c,\F}$ be the class of finitely generated Lazard Lie algebras over $\F$ of nilpotency class $\leq c$, in the language $\LL_{c,\F}$.
\item Let $\Lbb_{c,p}$ be the class of finite Lazard Lie algebras over $\F_p$ of nilpotency class $\leq c$, in the language $\LL_c$.
\item We write $\G_{c,p}$ for the class of finite Lazard groups of exponent $p$ and of nilpotency class $\leq c$ in the language of groups expanded by predicates for the Lazard series.
\end{enumerate}
\end{defn}

When $p > c$, the classes $\Lbb_{c,p}$ and $\G_{c,p}$ are uniformly bi-definable via the Lazard correspondence, see Subsection \ref{subsec:lazardcorrespondence}. Further, $\Lbb_{c,p}$ is a particular case of $\Lbb_{c,\F}$. The goal of this section is to prove that $\Lbb_{c,\F}$ is a Fraïssé class (see Definition \ref{fraisseclass}), by proving an amalgamation result. We actually prove a stronger result: $\Lbb_{c,\F}$ is a \textit{free amalgamation class} in the sense of Baudisch (Definition \ref{def:freeamalgambaudisch} below).
%

This free amalgam can be compared to an \textit{amalgamated free product} of structures. Although the existence of the free amalgam of graded Lie algebras (a notion equivalent to that of an LLA) was claimed by Baudisch in \cite{Baudisch4}, we found it worthwhile to give our own account, which differs substantially from the argument of \cite{Baudisch4}. Section \ref{subsection:stageII} gives explicitly the induction mentioned (but not proved) by Baudisch in \cite{Baudisch4}, which turns out to be highly non-trivial.

\begin{definition}
    A \textit{basic} extension of an LLA $A$ is an extension $B\supseteq A$ such that $A$ is an ideal of $B$ and $B = \vect{Ab}$ for some singleton $b$.  Note, we do not require $B$ to be a \emph{proper} extension of $A$. 
\end{definition}

The construction of the free amalgam of nilpotent Lie algebras passes through three stages.

\begin{itemize}
    \item[(I)] (Subsection \ref{subsection:stageI}) We construct the free amalgam of $A = \langle Ca \rangle$ and $B = \langle Cb \rangle$ over $C$ where both $A$ and $B$ are basic extensions of $C$. The free amalgam is in this case explicitly constructed as a semi-direct product of $C$ and a free nilpotent Lie algebra generated by $X$ and $Y$, where $X$ and $Y$ act on $C$ by derivations in the same way as $a$ and $b$ respectively.  
    \item[(II)] (Subsection \ref{subsection:stageII}) We construct the free amalgam of $A$ and $B$ over $C$ when $A$ is a basic extension of $C$ and $B$ is arbitrary. This is achieved using Stage I and induction on a certain \textit{rank} of the extension $B$ of $C$. This rank computes a certain amount of complexity of the extension $B/C$. It is also witnessed via a particular linear basis of $B$ over $C$, which we called \textit{Malcev basis}\footnote{We called those bases after Malcev for their direct connection with the so-called \textit{Malcev coordinates} in nilpotent groups of exponent $p$.}. An induction scheme is used to construct the amalgam which is rooted in the work of Maier \cite{Maierexpp}\footnote{Maier himself was inspired by the work of Higman \cite{higman64}.}. The rank drops at each step of the induction and yields a first amalgam (denoted by the $\oslash$-amalgam) which is complicated to describe except as ``the amalgam resulting from the induction scheme". Revisiting the induction for each case, we prove that the $\oslash$-amalgam obtained satisfies $(1)$, $(2)$ and $(3)$ of Definition \ref{def:freeamalgambaudisch} hence is a free amalgam. 
    \item[(III)] (Subsection \ref{subsection:stageIII}) We construct the free amalgam of $A$ and $B$ over $C$ when $A$ and $B$ are arbitrary. This is achieved via a standard induction using Stage II.
\end{itemize}

The conclusion is given in Subsection \ref{subsec:conclusionfraisseclass}: the classes $\Lbb_{c,\F}$, $\Lbb_{c,p}$ are Fraïssé (and therefore $\G_{c,p}$ as well, when $p > c$). The last subsection of Section \ref{sec:freeamalgamation}\textemdash Subsection \ref{subsec:extraNSOP4}\textemdash is dedicated to a technical result on the free amalgam which will be used in Section \ref{sec:neostabilityproperties} to prove that the theories of the Fraïssé limits of $\Lbb_{c,p}$ and $\G_{c,p}$ are NSOP$_4$.

\subsection{Free Lazard Lie algebras and Hall sets}

\subsubsection{Free nilpotent Lie algebras} The goal of this subsubsection is to define the free LLA in a given set of generators. Recall that we fix a field $\F$. Every vector space and Lie algebra for the entirety of this section will be assumed to be over $\F$.

\begin{defn}
    Let $X = (X_1,\ldots,X_n)$ be a tuple of indeterminates. The \textit{free Lie algebra (over $\F$) generated by $X$} is the Lie algebra $F(X)$ containing $X$ which satisfies the following universal property: for any Lie algebra $L$ and any function $f:\set{X_1,\ldots, X_n}\to L$, there exists a unique Lie algebra homomorphism $g:F(X)\to L$ which extends $f$.
\end{defn}

For any choice of $X$, the free Lie algebra generated by $X$ exists and is unique up to isomorphism (see for instance \cite[Proposition 1]{bourbakiLiechapter23}). When $\abs{X}=1$, we get the free Lie algebra on one generator which is an abelian Lie algebra of dimension 1 as an $\F$-vector space. When $|X|\geq 2$, the free Lie algebra will be of infinite dimension as an $\F$-vector space. Elements of $F(X)$ will often be identified with \textit{Lie polynomials} i.e. formal expressions obtained by identifying terms of the language $\set{+,0,[\cdot,\cdot]}$ modulo the equations $[X_i,X_i] = 0$ and the Jacobi identity. For any $P(X)\in F(X)$ and any $a = (a_1,\ldots, a_n)$ from some Lie algebra $L$, the map $X_i\to a_i$ extends to a Lie algebra homomorphism $h:F(X)\to \vect{a_1,\ldots,a_n}$ and we will denote $h(P(X))$ by $P(a)$, the \textit{evaluation} of $P$ in $a$. \\

\noindent\textbf{Gradation.} Free Lie algebras carry a natural gradation: 
 elements of the free Lie algebra are Lie polynomials in the generators and the Lie algebra can be written as the direct sum of the Lie monomials of each degree. \emph{Lie monomials}, or just \emph{monomials} when it is understood from context, in $F(X)$ are inductively defined as follows: every $X_i\in X$ is a monomial and if $P$ and $Q$ are monomials, then so is $[P,Q]$. We denote by $M(X)$ the set of Lie monomials. We will also need to consider gradations on a free Lie algebra where the generators live in a specified summand, which might not be the first summand. 
 
 Let $\N^+ = \set{1,2,\ldots}$. Any map $f:\set{X_1,\ldots,X_n}\to \N^+$ extends\footnote{More precisely, as $(\N^+,+)$ is a magma, the map $f$ extends to a homomorphism of the free magma generated by $X$, and being a magma homomorphism one has $f([X_i,X_j]) = f(X_i)+f(X_j)$. This map can directly be extended to the free Lie algebra in the following way: the map extends to the free $\F$-algebra having $M(X)$ as monomials and then to the quotient by the ideal generated by the elements $[x,x]$ and the Jacobi identity elements $[x,y,z]+[y,z,x]+[z,x,y]$ to get the corresponding $f$.} to a map $f:M(X)\to \N^+$ such that $f$ defines a \textit{gradation of $F(X)$} in the following sense: define 
\[V_n^f := \Span (\set{Q\in M(X)\mid f(Q) = n})\] then we have:
\begin{itemize}
    \item $F(X) = \bigoplus\limits_{n\in \N} V_n^f$,
    \item $[V_n^f,V_m^f]\seq V_{n+m}^f$.
\end{itemize}
This appears for instance in \cite[Chapitre 2, §2, Section 6]{bourbakiLiechapter23}. By setting $$P_n^f := \bigoplus_{k\geq n} V_k^f,$$ we get that $(P_i^f)_{i<\omega}$ is a Lazard series (of infinite length) in the Lie algebra $F(X)$. The map $f$ is entirely determined by the tuple $\alpha = (\alpha_1,\ldots, \alpha_n)$ where $\alpha_i = f(X_i)$ hence we will for now consider that each tuple $\alpha\in (\N^+)^n$ determines a unique gradation $(V_i^\alpha)_{i<\omega}$ and a Lazard series $(P_i^\alpha)_{i<\omega}$. Then, for each element $Q\in F(X_1,\ldots,X_n)\setminus\set{0}$ there exists a maximal $m$ and a minimal $n$ such that $m\leq n$ and 
\[Q\in V_{m}^\alpha \oplus V_{m+1}^\alpha\oplus \ldots \oplus V_{n}^\alpha.\]
In this case we call $m$ the \textit{$\alpha$-level of $Q$} and $n$ the \textit{$\alpha$-degree of $Q$}, denoted respectively by $\lev_\alpha(Q)$ and $\deg_\alpha(Q)$. We extend those functions to $0$ by setting $\lev_\alpha(0) = \deg_\alpha(0) = \infty$. Then $P^\alpha_n = \set{Q\in F(X_1,\ldots,X_n)\mid \lev_\alpha(Q)\geq n}$. An element $Q\in V_n^\alpha$ for some $n$ is called \textit{$\alpha$-homogeneous} and such element satisfies $\deg_\alpha(Q) = \lev_\alpha(Q)$. 

When $\alpha$ is not mentioned, $\lev$, $\deg$, $P_n$ in $F(X)$ refers to the case where $\alpha = (1,\ldots,1)$, and the associated gradation is called the \textit{natural gradation}. In this case, elements of each $V_n$ are called \emph{homogeneous}.

\begin{example}[Natural gradation]
    Consider the gradation given by $\alpha = (1,\ldots,1)$ in $F(X,Y,\ldots)$. Then $X$ is of level and degree $1$, $[X,Y]$ is of level and degree $2$ and $X+[X,Y]$ has level $1$ and degree $2$. The degree behaves like the degree of polynomials and the level behaves like a valuation: $\lev(a+b)\geq \min\set{\lev(a),\lev(b)}$ with equality if $\lev(a) \neq \lev(b)$.
\end{example}

Observe that for all $\alpha \in \N^n$ and $c\in \N$ the space $P_{c+1}^\alpha$ is an ideal of $F(X)$, hence we may consider the quotient of $F(X)$ by $P_{c+1}^\alpha$. This quotient is a Lie algebra of nilpotency class at most $c$. In the case of the natural gradation $\alpha = (1,\ldots,1)$, the Lie algebra $F_c(X) := F(X)/P_{c+1}$ is of nilpotency class $c$ and is usually called the \textit{free $c$-nilpotent Lie algebra}. It is easy to deduce the universal property for the Lie algebra $F_c(X)$ in the category of $c$-nilpotent Lie algebras from the universal property for $F(X)$ in the category of Lie algebras.

\begin{definition}[Free Lazard Lie algebra]
    Given $X = (X_1,\ldots,X_n)$, $\alpha = (\alpha_1,\ldots,\alpha_n)\in \N^n$ and $c\in \N$, we denote by $F_c(X,\alpha)$ the quotient $F(X)/P_{c+1}^\alpha$. Then $F_c(X,\alpha)$ is nilpotent of nilpotency class at most $ c$. In $F_c(X,\alpha)$, the sequence $(S_{i}^{\alpha})_{1\leq i\leq c+1}$ of ideals defined by
    \[S_i^\alpha := P_i^\alpha/P_{c+1}^\alpha\]
    is a Lazard series in $F_c(X,\alpha)$. The Lie algebra $F_c(X,\alpha)$ equipped with the Lazard series $(S_i^\alpha )_{1\leq i\leq c+1}$ is called the \textit{free Lazard Lie algebra associated to $n,\alpha, c$}. 
\end{definition}

One easily sees that $\lev_{\alpha}(Q)\geq \lev_{(1,\ldots,1)}(Q)$ hence $P_n\subseteq P_n^{\alpha}$ in general. It follows that $F_c(X,\alpha)$ surjects onto $F_c(X)$ in general, though they need not be isomorphic, see Example \ref{example:freeLLA}.
 
Recall that in any LLA $A$ and $a\in A$, $\lev(a)$ is the maximal $i$ such that $a\in P_i(A)$. In $F_c(X,\alpha)$, the $\alpha$-level coincides with the level of the LLA $(F_c(X,\alpha), (S_i^\alpha)_i)$. The free LLA $F_c(X,\alpha)$ enjoys the following universal property:
\begin{center}
    For any LLA $A$ and $a = (a_1,\ldots,a_n)\in A^n$ such that $\lev(a_i)\geq \alpha_i$ for each $1\leq i\leq n$ then the map $X_i\to a_i$ extends to a LLA homomorphism $F_c(X,\alpha)\to A$.
\end{center}

\subsubsection{Hall sets}  A particular (ordered) linear basis for the free Lie algebra $F(X)$ is given by the set of \textit{Hall monomials}.

\begin{definition}[Hall sets]
    Let $X=(X_1,\ldots,X_n)$. We recursively define linearly ordered sets of monomials in $F(X)$, called \textit{Hall sets}. Start with $\HS_1=\{X_1,X_2,\ldots,X_n\}$ and declare $X_1<X_2<\cdots<X_n$. If  $\HS_1, \ldots, \HS_n$ have been defined, then $\HS_{n+1}$ is the set of monomials $[P,Q]$ such that
        \begin{enumerate}[label=(\alph*)]
            \item $P,Q\in \HS_1\cup \cdots \cup \HS_n$;
            \item $\deg(P)+\deg(Q)=n+1$;
            \item $P>Q$; and
            \item if $P=[R,S]$, then $S\leq Q$.
        \end{enumerate}
We then linearly order the monomials in $\HS_{n+1}$ and for $P\in \HS_1\cup\cdots \cup \HS_n$ and $Q\in \HS_{n+1}$ we declare $P<Q$. Put $\HS(X_1,\ldots,X_n)=\bigcup_{n\geq 1}\HS_n$. Members of $\HS$ are called \textit{Hall monomials}\footnote{These are known as \textit{basic commutators} in the literature, this terminology comes from the fact that they form a basis of the free Lie algebra.}.
\end{definition}

Every monomial of the Hall set $\HS_n$ is of degree $n$. By construction, for any Hall monomials $P$ and $Q$, if $\deg(P)<\deg(Q)$, then $P<Q$. Also, if $P<Q$, then $\deg(P)\leq \deg(Q)$. For any $n$, the set $\HS_n$ is a basis of the space of homogeneous polynomials of (natural) degree $n$ in $F(X)$, see \cite{hallbasisMhall}. In particular, the set $\HS$ forms a basis for the free Lie algebra $F(X)$. 

In the free $c$-nilpotent Lie algebra $F_c(X)$, Hall monomials of degree larger than $c$ vanish and the family $\HS_{\leq c} := \bigcup_{n\leq c} \HS_n$ is a basis of $F_c(X)$. 

There is also a weighted version of Hall sets which we give now. This appears in \cite[Section 2]{labute}, where it is called a \textit{weighted} Hall set. This is similar to \cite[Fact 4.2]{Baudisch4}.

\begin{definition}[Weighted Hall sets]
    Let $X=(X_1,\ldots,X_n)$ and $\alpha\in (\N^+)^n$. We recursively define linearly ordered sets of monomials in $F(X)$, called \textit{Hall sets}. Let $\HS_1^\alpha$ be those $X_i$ of $\alpha$-degree $1$ and order $\HS_1^\alpha$ arbitrarily, for instance $X_i< X_j$ if $i<j$. If  $\HS_1^\alpha, \ldots, \HS_n^\alpha$ have been defined, then $\HS_{n+1}^\alpha$ is the set of those $X_i$ of $\alpha$-degree $n+1$ together with monomials $[P,Q]$ such that
    \begin{enumerate}[label=(\alph*)]
        \item $P,Q\in \HS_1^\alpha\cup \ldots \cup \HS_n^\alpha$;
        \item $\deg_\alpha(P)+\deg_\alpha(Q)=n+1$;
        \item $P>Q$; and
        \item if $P=[R,S]$, then $S\leq Q$.
    \end{enumerate}
We then linearly order the monomials in $\HS_{n+1}^\alpha$ and for $P\in \HS_1^\alpha\cup\cdots \cup \HS_n^\alpha$ and $Q\in \HS_{n+1}^\alpha$ we declare $P<Q$. Put $\HS^\alpha(X_1,\ldots,X_n)=\bigcup_{n\geq 1}\HS_n^\alpha$. Members of $\HS^\alpha$ are called \textit{(weighted) Hall monomials}.
\end{definition}

We again get that $\HS_d^\alpha$ is a basis of the space of $\alpha$-homogeneous polynomials of $\alpha$-degree $d$.

\begin{remark}[Witt formula]\label{rk:wittformula}
The cardinality of $\HS_d$, i.e. the number of Hall monomials of degree $d$ in the free Lie algebra with $n$ generators is given by the Witt formula
\[\frac{1}{d} \sum_{k\mid d}\mu(k)n^{d/k},\]
where $\mu$ is M\"obius function. See \cite[Ch. II, \S3, Th\'eor\`eme 2]{bourbakiLiechapter23}.  It follows that the dimension of $F_c(X)$ is  
\[\sum_{d=1}^{c}\frac{1}{d} \sum_{k\mid d}\mu(k)n^{d/k}.\]
Note that the dimension of $F_c(X,\alpha)$ only depends on $c, n$ and $\alpha$, we denote it $g_{n,c}(\alpha)$ or just $g(\alpha)$. There should exist an explicit description of the function $g$. For any $\alpha\in (\N^+)^n$ we have $P_n\seq P_n^\alpha$, hence the dimension of $F_c(X,\alpha) = F(X)/P_{c+1}^\alpha$ is less than or equal to the dimension of $F_c(X) = F(X)/P_{c+1}$. Thus 
\[g(\alpha)\leq g(1,\ldots,1) = \sum_{d=1}^{c}\frac{1}{d} \sum_{k\mid d}\mu(k)n^{d/k}.\]
\end{remark}

\begin{remark}[A construction of weighted Hall sets]
We now describe an easy recipe for obtaining weighted Hall bases given the existence of unweighted ones. Let $X=(X_1,\ldots,X_n)$ and $\alpha = (\alpha_1,\ldots, \alpha_n)$ be given. For each $i=1,\ldots,n$ let $Y^i_1,\ldots, Y^i_{\alpha_i}$ be new variables and consider a Hall set $\HS$ in the free algebra $F((Y^1_k)_{k\leq \alpha_1},\ldots, (Y^n_k)_{k\leq \alpha_n})$ and pick $Z_i \in \vect{Y_1^i, \ldots, Y_{\alpha_i}^i}\cap \HS$ of degree $\alpha_i$ for each $i$. The Lie algebra $L = \vect{Z_1,\ldots, Z_n}$ is free by the Shirshov–Witt Theorem \cite{shirshov2009subalgebras, witt1956unterringe}, hence $L$ is isomorphic to $F(X_1,\ldots, X_n)$. Let $H$ be the set of those monomials in $\HS$ which only involve $Z_i$, then $H$ is a weighted Hall set of $L$ for the gradation given by $\alpha$.
\end{remark}

\begin{remark}[Weighted Hall sets in two variables]
In the case of two variables $(X,Y)$ and given a gradation $(\alpha,\beta)$, the situation is easier. The weighted Hall monomials $\HS^{\alpha,\beta}$ are constructed following the same algorithm as in the unweighted case with the only constraint that $X$ and $Y$ are ordered consistently with the order of $\alpha$ and $\beta$, i.e. $X<Y$ if and only if $\alpha<\beta$. Then $X,Y$ are elements of minimal $\alpha$-degree, the next monomial in the Hall basis is $[X,Y]$ (or $[Y,X]$) and the iterative construction of the Hall set follows the same procedure as in the unweighted case. This allows us to forget about the differences between weighted and unweighted Hall monomials in what concerns their iterative construction, as we will only consider them in two variables in the construction of the free amalgam of two basic extensions.
\end{remark}

We now illustrate the notions above with concrete examples.

\begin{example}\label{example:freeLLA}
    We consider $F_{3}(X,Y,\alpha,\beta)$ in the case of nilpotence class $3$, with two different initial conditions.\\
    
 \noindent\begin{minipage}{0.7\textwidth}
If $\alpha = \beta = 1$, then the Hall set is given by
    \[\HS_{\leq 3} = \set{X<Y<[Y,X]<[[Y,X],X]<[[Y,X],Y]}\]
    and $P_4 = \HS_{\geq 4}$, so $F(X,Y,1,1)$ is identified with the vector span of $\HS_{\leq 3} $. In particular, $\dim(F_{3}(X,Y,\alpha,\beta)) = 5$. 
\end{minipage}%
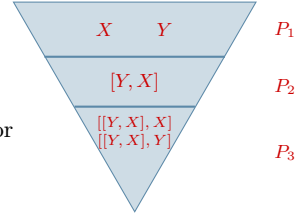
\begin{minipage}{0.3\textwidth}
\begin{tikzpicture}[scale=0.8, every node/.style={transform shape}]
\draw[airforceblue,thick] (0,0) -- (4,0) -- (2,-3.46) -- cycle; 
\fill[airforceblue!30] (0,0) -- (4,0) -- (2,-3.46) -- cycle; 
\draw[airforceblue, thick] (0.5,-0.9) -- (3.5,-0.9); 
\draw[airforceblue, thick] (1,-1.73) -- (3,-1.73); 
\node at (1.5,-0.45){\textcolor{red!80!black}{$X$}};
\node at (2.5,-0.45){\textcolor{red!80!black}{$Y$}};
\node at (2,-1.32){\textcolor{red!80!black}{$[Y,X]$}};
\node at (2,-2.0){\begin{scriptsize}\textcolor{red!80!black}{$[[Y,X],X]$}\end{scriptsize}};
\node at (2,-2.3){\begin{scriptsize}\textcolor{red!80!black}{$[[Y,X],Y]$}\end{scriptsize}};
\node at (4.5,-2.5){\textcolor{red!80!black}{$P_3$}};
\node at (4.5,-1.4){\textcolor{red!80!black}{$P_2$}};
\node at (4.5,-0.45){\textcolor{red!80!black}{$P_1$}};
\end{tikzpicture}
\end{minipage}

\noindent\begin{minipage}{0.7\textwidth}
If $\alpha = 1, \beta = 2$, one computes: 
        \begin{center}
            \begin{tabular}{ c|c } 
$Q$ & $\deg_{\alpha,\beta}(Q)$\\
 \hline
 $X$ & $1$ \\ 
 $Y$ & $2$ \\ 
 $[Y,X]$ & $3$ \\ 
 $[[Y,X],X]$ & $4$\\
 $[[Y,X],Y]$ & $5$
\end{tabular}
        \end{center}
        
The Hall sets are: $\HS_1^{\alpha,\beta} = \set{X}$, $\HS_2^{\alpha,\beta} = \set{Y}$, $\HS_3^{\alpha,\beta} = \set{ [Y,X]}$, $\HS_4^{\alpha,\beta} = \set{ [[Y,X],X]}$, $\HS_5^{\alpha,\beta} = \set{ [[Y,X],Y]}$. In particular, $P_4^{\alpha,\beta} = \HS_{\geq 4}^{\alpha,\beta}$ hence $\dim(F_{3}(X,Y,\alpha,\beta)) = 3$. 
\end{minipage}%
\begin{minipage}{0.3\textwidth}
\begin{tikzpicture}[scale=0.8, every node/.style={transform shape}]
\draw[airforceblue,thick] (0,0) -- (4,0) -- (2,-3.46) -- cycle; 
\fill[airforceblue!30] (0,0) -- (4,0) -- (2,-3.46) -- cycle; 
\draw[airforceblue, thick] (0.5,-0.9) -- (3.5,-0.9); 
\draw[airforceblue, thick] (1,-1.73) -- (3,-1.73); 
\node at (1.5,-0.45){\textcolor{red!80!black}{$X$}};
\node at (2.5,-1.32){\textcolor{red!80!black}{$Y$}};
\node at (2,-2.2){\textcolor{red!80!black}{$[Y,X]$}};
\node at (4.5,-2.5){\textcolor{red!80!black}{$P_3$}};
\node at (4.5,-1.4){\textcolor{red!80!black}{$P_2$}};
\node at (4.5,-0.45){\textcolor{red!80!black}{$P_1$}};
\end{tikzpicture} 
\end{minipage}
\end{example}

\subsection{Stage I - Basic extensions}\label{subsection:stageI}

Recall that all vector spaces and Lie algebras will be considered to be taken over $\mathbb{F}$.  Additionally, we fix a $c$ and will write `LLA' to mean `Lazard Lie algebra over $\mathbb{F}$ of nilpotence class at most $c$' for the rest of the section, unless otherwise specified.  Recall (Definition \ref{def:levelofelement}) that the level $\lev(a)$ is defined as the maximal $i \leq c+1$ such that $a\in P_i$.
\begin{definition}\label{def:levelofanextension}
    For LLAs $B\seq A$, the \textit{level of $A$ over $B$}, denoted $\lev(A/B)$, is defined to be the maximal $1\leq i\leq c+1$ such that $A = \Span(BP_i(A))$.
\end{definition}

The goal of this subsection is to prove the following theorem.

\begin{theorem}\label{thm:existence_strong_amalgam_monogenous}
    Let $A,B,C$ be LLAs of nilpotency class $\leq c$ with $A, B$ basic extensions of $C$. Assume that $ A= \vect{Ca}$ and $B = \vect{Cb}$ with $\lev(a),\lev(b)$ maximal with this property. Let $\alpha = \lev(a)$, $\beta = \lev(b)$. Then there exists a strong amalgam $S$ of $A$ and $B$ over $C$ such that the following conditions are satisfied:
    \begin{enumerate}
        \item in $S$, we have $\vect{a,b} \cong F_c(X,Y,\alpha,\beta)$ via the map $a\mapsto X$, $b\mapsto Y$;
        \item $C$ is an ideal of $S$;
        \item there exists an ideal $D$ of $S$ containing $B$ such that $S = \vect{Da}$ is a basic extension of $D$ and $\lev(D/B) = \lev(a)+\lev(b)$.
    \end{enumerate}
    In $(3)$, $D = \vect{Bh_3\ldots h_k}$, where $(h_i)_{3\leq i\leq k}$ is an enumeration of Hall monomials in $\HS^{\alpha,\beta}$ without $X$ and $Y$. 
\end{theorem}

\begin{proof} We define the amalgam $S$ as follows. Consider the free Lazard Lie algebra $F = F_c (X,Y,\alpha,\beta)$. As $C$ is an ideal of both $A$ and $B$, $\mathrm{ad}(a)|_{C}$ and $\mathrm{ad}(b)|_{C}$ define derivations on $C$.  Moreover,  $\mathrm{ad}(a)|_{C}$ and $\mathrm{ad}(b)|_{C}$ are in $D_{\alpha}$ and $D_{\beta}$ respectively in the associated Lazard series on $\mathrm{Der}_{\mathrm{Laz}}(C)$, by Lemma \ref{Lazard derivations}. By the universal property of $F_{c}(X,Y,\alpha,\beta)$, the function $X \mapsto \mathrm{ad}(a)|_{C}$ and $Y \mapsto \mathrm{ad}(b)|_{C}$ gives rise to a unique LLA homomorphism $g: F \to \mathrm{Der}_{\mathrm{Laz}}(C)$. We define $S = C \rtimes F$ to be the associated semi-direct product. We interpret the predicates by 
$$
P_{i}(S) = \mathrm{Span}(P_{i}(C) \cup P_{i}(F)).
$$
If $c\in P_i(C), d\in P_j(C)$ and $u\in P_i(F), v\in P_j(F)$ then 
\[[c+u,d+v] = [c,d]-g(v)(c)+g(u)(d)+[u,v]\]
It is clear that $[c,d]\in P_{i+j}(C)$ and $[u,v]\in P_{i+j}(F)$. Further, as $g$ is an LLA homomorphism, $g(v)\in D_j$ hence $g(v)(c)\in P_{i+j}(C)$. Similarly, $g(u)(d)\in P_{i+j}(C)$. We conclude that $[c+u,d+v]\in P_{i+j}(S)$ hence $(P_i(S))_{1\leq i\leq c+1}$ is a Lazard series on $S$.
Now we check that this $S$ is the desired amalgam.  Observe that properties (1) and (2) are satisfied by construction. 

Next, note that the map $\iota : A \to \langle C,X \rangle$, defined to be the identity on $C$ and mapping $a \mapsto X$, is an LLA isomorphism, where $\langle C,X \rangle$ is the subalgebra of $S$ generated by $C$ and $X$.  It is clearly an isomorphism of the underlying vector spaces, since $C$ is an ideal of both and thus, as vector spaces, we have $A = C \oplus \langle a \rangle$ and $\langle C,X \rangle = C \oplus \langle X \rangle$. It also respects the predicates for the Lazard series on $S$. Finally, if we are given arbitrary $c + \gamma a$, $c' + \gamma' a \in A$, where $c,c' \in C$ and $\gamma,\gamma' \in \mathbb{F}$, we have 
$$
\iota([c+ \gamma a, c'+\gamma'a]) = \iota([c,c'] + \gamma \mathrm{ad}(a)(c') - \gamma'\mathrm{ad}(a)(c)) = [c,c'] + \gamma [X,c'] - \gamma' [X,c].
$$
Similarly, we have 
$$
[\iota(c+ \gamma a), \iota(c'+\gamma'a)] = [c + \gamma X,c' + \gamma' X].
$$
By bilinearity, this easily implies that $\iota$ is a Lie algebra homomorphism and therefore $\iota$ is an LLA isomorphism. This shows that, via $\iota$, $A$ embeds into $S$ over $C$.  A parallel argument shows that $B$ embeds into $S$ over $C$, and thus $S$ is an amalgam.  Since $\langle C,X \rangle \cap \langle C,Y \rangle = C$, this amalgam is strong.  

Thus, we only are left with showing (3). If we set $D = \vect{Bh_3\ldots h_k}$, where $(h_i)_{3\leq i\leq k}$ is an enumeration of Hall monomials in $\HS^{\alpha,\beta}$ without $X$ and $Y$, then, since $[X,Y] \in P_{\alpha + \beta}(S)$ (or $[Y,X] \in P_{\alpha + \beta}(S)$) is the monomial of minimal degree in the Hall basis, excluding $X$ and $Y$, we have $\mathrm{lev}(D/B) = \alpha + \beta = \mathrm{lev}(a) + \mathrm{lev}(b)$. Since the Hall basis is a basis of $S$ over $C$ and $[X,h_{i}] \in \mathrm{Span}(h_{3}, \ldots, h_{k})$ for all $i > 1$, we know $D$ is an ideal of $S$ and $S = \langle D,X \rangle$ is a basic extension of $D$. 
\end{proof}


\subsubsection{What is the amalgam constructed by Baudisch?}

Recall Baudisch's definition of the free amalgam.
\begin{definition}[Baudisch] \label{def:freeamalgambaudisch}
    Let $A,B,C$ be LLAs with embeddings $f_0 :C\to A, g_0:C\to B$. We say that $S$ is a \textit{free amalgam of $A$ and $B$ over $C$} if $S$ is an amalgam of $A$ and $B$ over $C$, with embedding $f_1:A\to S$, $g_1:B\to S$ with $f_1\circ f_0 = g_1\circ g_0$ and such that the following three conditions hold, for $A' = f_1(A), B' = g_1(B), C' = (f_1\circ f_0)(C)$:
    \begin{enumerate}
        \item $S = \vect{A'B'}$;
        \item \textit{(Strongness)} $A'\cap B' = C'$;
        \item \textit{(Freeness)} for any LLA $D$ and any LLA \textit{homomorphisms} $f:A\to D$ and $g:B\to D$, there exists a (unique) $h : S\to D$ such that the following diagrams commute. 
        \begin{center}\begin{tikzcd}
& A \ar[dr,"f_1"] \ar[drr, "f", bend left=20]
&
&[1.5em] \\
C \ar[ur,"f_0"] \ar[dr,"g_0"]
&
& S  \ar[r, "h",dashed]
& D \\
& B \ar[ur, "g_1"]\ar[urr, "g"', bend right=20]
&
&
\end{tikzcd}\end{center}
    \end{enumerate}
    We denote the free amalgam $S$ by $A\otimes_C B$.
\end{definition}

\begin{remark}\label{rk:uniquenessofthemap}
Under condition $(1)$ if a map $h$ satisfying $(3)$ exists, then it is unique. Thus, we will allow ourselves to talk about \textit{the} free amalgam of $A$ and $B$ over $C$. 
\end{remark}

\begin{remark}[The free amalgam is unique up to isomorphism over $C$]
Let $S$ and $S'$ be two free amalgams of $A$ and $B$ over $C$. Then $S\cong S'$. By Remark \ref{rk:uniquenessofthemap}, we may assume that $A,B$ are finitely generated over $C$. By considering images, assume that $A,B,C\seq S$. Let $f:A\to S'$ and $g:B\to S'$ be embeddings agreeing on $C$ and let $A',B',C'$ be the images of $A,B,C$ in $S'$. By the freeness property, there is a homomorphism $h: S\to S'$ which commutes with the inclusions $C\seq A\cap B$ and the isomorphism $f:A\cong A'$ and $g:B\cong B'$. We have that $h$ is surjective as $S' = \vect{A'B'}$. The same argument yields that there is a surjective homomorphism from $S'\to S$. As $S$ and $S'$ are finitely generated over $C$, $h$ is an isomorphism. Note that the strongness condition is not used for the uniqueness, only properties $(1)$ and $(3)$. This definition is designed to make the strong amalgam unique, if it exists. 
\end{remark}

%
%
%

\begin{lemma}\label{lm:idealspan}
    Let $C,A$ be LLAs. If $C$ is an ideal of $\vect{AC}$, then $\vect{AC} = \Span(AC)$.
\end{lemma}
\begin{proof}
    We may assume that $A,C$ are finitely generated. Let $c = (c_1,\ldots, c_n)$ and $a = (a_1,\ldots,a_m)$ be generators of $C$ (resp. $A$) as LLAs. Let $\gamma = (\gamma_1,\ldots,\gamma_n)$ and $\alpha = (\alpha_1,\ldots,\alpha_m)$
    be the levels of $c$ and $a$. Let $X = (X_1,\ldots,X_n)$, $Y = (Y_1,\ldots,Y_m)$ and $F = F_{c}(X_1,\ldots,X_n,Y_1,\ldots,Y_n, \gamma,\alpha)$. Let $H$ be a Hall basis of $F$. The surjective endomorphism $F\to \vect{C,A}$ given by the universal property of $F$ implies that $H_0 = \set{P(c,a)\mid P(X,Y)\in H}$ is a generating subset of $\vect{AC}$ as a vector space. We prove that for each $P\in H$, either $P(c,a)\in C$ or $P(c,a)\in A$. We prove it by induction on the complexity. If $P(X,Y)$ is $X$ or $Y$ then $P(c,a)$ is in $C$ or $A$. By induction, assume that $P = [Q,R]$ for $P,Q,R\in H$. By induction hypothesis, $Q(c,a)$ and $R(c,a)$ are in $A$ or $C$. If either one is in $C$ then so is $P$ as $C$ is an ideal. Otherwise both $Q$ and $R$ are in $A$ hence $P$ is in $A$.
\end{proof}

Here is a cheap way to get free amalgams in the sense of Definition \ref{def:freeamalgambaudisch}.

\begin{lemma}\label{lm:characterisationidealfreeamalgamation}
    Let $C\seq A\cap B$ be LLAs and let $S$ be an amalgam of $A$ and $B$ over $C$ with $S = \vect{AB}$. Let $a = (a_1,\ldots,a_n)$ and $b=(b_1,\ldots,b_m)$ be bases of $A$ and $B$ over $C$ respectively. If
    \begin{enumerate}
        \item $C$ is an ideal of $S$
        \item $\vect{ab}\cap C = \set{0}$
        \item $\vect{ab} \cong F(X,Y,\lev(a),\lev(b))$
    \end{enumerate}
    Then $S$ satisfies (1),(2), and (3) of Definition \ref{def:freeamalgambaudisch}. In particular, the existence of such an $S$ implies $A\otimes_C B$ exists (here $X = (X_{1}, \ldots, X_{n})$ and $\mathrm{lev}(a)$ denotes the tuple $(\mathrm{lev}(a_{1}),\ldots, \mathrm{lev}(a_{n}))$ and likewise for $Y$ and $\mathrm{deg}(b)$).
\end{lemma}
\begin{proof}
    Condition $(3)$ implies that $a\cup b$ is linearly independent over $\emptyset$, hence by $(2)$ we have $A\cap B = C$. It remains to prove the freeness condition. By $(1)$, $C$ is an ideal of $\vect{C\vect{ab}}$ hence $S = \vect{AB} = \Span(C,\vect{ab})$ by Lemma \ref{lm:idealspan}. Then $S = C\oplus \vect{ab}$, as a vector space, by $(2)$. Let $f:A\to D$ and $g:B\to D$ be homomorphisms such that $f(c) = g(c)$ for all $c\in C$. Let $j_0 = f\upharpoonright C: C\to D$. We extend $j_0$ to $j:S\to D$. By $(3)$ there is a Hall basis $H$ of $\vect{ab}$ and we define $j_1:\vect{ab}\to D$ by the universal property of the free LLA. More precisely we define $j_1$ inductively: for $a_i, b_i$ we define $j_1(a_i) = f(a_i)$ and $j_1(b_i) = g(b_i)$. For $P\in H$ there exists a unique pair $(Q,R)\in H$ such that $P = [Q,R]$. Inductively, $j_1$ is defined on $Q,R$ and we define $j_1(P) := [j_1(Q),j_1(R)]$. This defines an LLA homomorphism $j_1:\vect{ab}\to D$. As $S = C\oplus \vect{ab}$, $j := j_0+j_1$ defines a linear homomorphism. We check that it preserves the Lie bracket. Note that as $C$ is an ideal, hence for $c,c'\in C, u,v\in \vect{ab}$ \[[c+u,c'+v] = \underbrace{[c,c']+[c,v]+[u,c']}_\text{$\in C$}+[u,v].\]
     It follows that $j([c+u,c'+v]) = j_0([c,c'])+j_0([c,v])+j_0([u,c'])+j_1([u,v]) = [j_0(c),j_0(c')]+j_0([c,v])+j_0([u,c'])+ [j_1(u),j_1(v)]$ hence by bilinearity it is enough to check that $j([c,P]) = [j(c),j(P)]$ for all $c\in C,P\in H$. We proceed by induction. If $P=a_i$ (or $P=b_i$) then it follows from the fact that $f$ (resp. $g$) is a Lie algebra homomorphism. If $P = [Q,R]$ for $Q,R\in H$, we have $[c,P] = [[c,Q],R]+[Q,[c,R]]$. As $C$ is an ideal, $[c,Q]\in C$ hence by the induction hypothesis on $R$, $j([[c,Q],R]) = [j([c,Q]),j(R)]$. By the induction hypothesis (on $Q$) we have $j([c,Q]) = [j(c),j(Q)]$ hence in turn $j([[c,Q],R]) = [[j(c),j(Q)],j(R)]$. Similarly $j([Q,[c,R]]) = [j(Q),[j(c),j(R)]]$, so, using bilinearity and the Jacobi identity backwards:
     \[j([c,P]) = j([[c,Q],R])+j([Q,[c,R]]) = [j(c),[j(Q),j(R)]]\]
     As $j$ extends $j_1$ we have $[j(Q),j(R)]=j([Q,R]) = j(P)$ hence we conclude $j([c,P]) = [j(c),j(P)]$.
\end{proof}

\begin{remark}
    Of course, the converse of Lemma \ref{lm:characterisationidealfreeamalgamation} does not hold in general: take $A= \vect{a}$ and $B= \vect{b_1,b_2}$ with $b_1,b_2$ not free (e.g. $[b_1,b_2] = 0$), then there is no such amalgam of $A$ and $B$.
\end{remark}

\begin{corollary}\label{cor:characterisationBaudischfreeamalgammonogeneous}
    The amalgam constructed in Theorem \ref{thm:existence_strong_amalgam_monogenous} is a free amalgam. In particular  for singletons $a,b,C$ with $C$ an ideal of $\vect{Ca}$ and $\vect{Cb}$ the free amalgam of $\vect{Ca}$ and $\vect{Cb}$ over $C$ exists. Finally, the following are equivalent:
    \begin{enumerate}
        \item $\vect{Cab}\cong \vect{Ca}\otimes_{C} \vect{Cb}$
        \item \begin{enumerate}
        \item $C$ is an ideal of $\vect{Cab}$
        \item $\vect{ab}\cap C = \set{0}$
        \item $\vect{ab} \cong F(X,Y,\deg(a),\deg(b))$
    \end{enumerate}
    \end{enumerate}
\end{corollary}
\begin{proof}
    By Theorem \ref{thm:existence_strong_amalgam_monogenous} and Lemma \ref{lm:characterisationidealfreeamalgamation}, for any such $C, a,b$ an amalgam satisfying $(a),(b),(c)$ exists and this amalgam is free. $(2)$ implies $(1)$ is Lemma \ref{lm:characterisationidealfreeamalgamation}. Now assume that $\vect{Cab}\cong \vect{Ca}\otimes_{C} \vect{Cb}$ and let $S$ be the amalgam of $\vect{Ca}$ and $\vect{Cb}$ over $C$ constructed from Theorem \ref{thm:existence_strong_amalgam_monogenous}. Then $S$ is also free hence by the uniqueness of the free amalgam we have $S\cong \vect{Ca}\otimes_{C} \vect{Cb}$ hence via the isomorphism, $\vect{Cab}$ satisfies $(a),(b),(c)$.
\end{proof}

\begin{definition}
    For all $A,B,C$ subsets of a common $c$-nilpotent LLA over $\mathbb{F}$, we define
    \[A\indi \otimes _C B\iff \vect{ABC}\cong \vect{AC}\otimes_C \vect{BC}\]
\end{definition}

\begin{proposition}\label{prop:basicpropertiesoffreeindependence}
    The relation $\indi \otimes$ satisfies symmetry, invariance, stationarity, and transitivity. Furthermore, if $\indi \otimes$ satisfies full existence (for all $A,B,C$ there exists $A'\equiv_C A $ such that $A'\indi \otimes _C B$) then it also satisfies monotonicity and base monotonicity, hence $\indi \otimes$ is a stationary independence relation in the sense of \cite{tentzieglerurysohn}.
\end{proposition}
\begin{proof}
    This follows from \cite[Theorem 3.4]{Baudisch4}, with the observation that the proofs of symmetry, invariance, stationarity, and transitivity do not use full existence\footnote{Note that what is here called full existence is what was called existence in \cite{Baudisch4} and the literature at that time. Existence is nowadays understood to refer to the property $A\ind_C C$ for all $A,C$. Full existence follows from existence and \textit{extension}: if $A\ind_C B$ and $C\seq B\seq D$ then there exists $A'\equiv_{BC} A$ such that $A'\ind_C D$. }. 
\end{proof}

\subsection{Stage II - Induction on the rank}\label{subsection:stageII}

\subsubsection{Malcev sets}

Observe the following consequence of Lemma \ref{lm:idealspan}:
\begin{corollary} \label{cor:spanprop}
    If $C\lteq \vect{Ca_1}\lteq \vect{Ca_1a_2}\lteq \ldots \lteq \vect{Ca_1\ldots a_n}$ then $\vect{Ca_1\ldots a_m} = \Span_\F(Ca_1\ldots a_m)$ for all $m\leq n$.
\end{corollary}
\begin{proof}
    Let $x$ be in $\vect{Ca_1\ldots a_m}$. As $\vect{a_m} = \Span(a_m)$, $\Span(a_m)$ is an LLA so by Lemma \ref{lm:idealspan} there exists $\lambda_m\in \F$ and $y\in \vect{Ca_1\ldots a_{m-1}}$ such that $x = y+\lambda_m a_m$. By induction hypothesis, there exists $\lambda_1,\ldots,\lambda_{m-1}\in \F$ such that $y = c+\sum_{i=1}^{m-1} \lambda_ia_i$ hence $x = c+\sum_{i=1}^m \lambda_i a_i\in \Span_\F(Ca_1,\ldots, a_m)$.
\end{proof}

Recall from Corollary \ref{cor:characterisationBaudischfreeamalgammonogeneous}: for singletons $a,b$ with $C$ an ideal of $\vect{Ca}$ and $\vect{Cb}$, we have $a\indi \otimes_C b$ if and only if 
\begin{enumerate}
    \item $C$ is an ideal of $\vect{Cab}$
    \item $\vect{ab}\cap C = \set{0}$
    \item $\vect{ab} \cong F(X,Y,\deg(a),\deg(b))$
\end{enumerate}

\begin{corollary}\label{cor:droppingtheEbaudischcase}
    Assume that $C$ is an ideal of $\vect{Ca}$ and $\vect{Cb}$ for singletons $a,b$. If $a\indi \otimes _C b$, then for all $E\seq C$ with $E\lteq \vect{Ea}$ and $E\lteq \vect{Eb}$ we have $a\indi \otimes _E b$.
\end{corollary}

\begin{proof}
If $E\lteq \vect{Ea}$ and $E\lteq \vect{Eb}$ then $a\indi \otimes _E b$ is equivalent to 
\begin{enumerate}
        \item $E$ is an ideal of $\vect{Eab}$
        \item $\vect{ab}\cap E = \set{0}$
        \item $\vect{ab} \cong F(X,Y,\deg(a),\deg(b))$
    \end{enumerate}
As $E\seq C$ only Condition $(1)$ needs to be checked. Let $h_1,\ldots,h_l$ be the Hall basis, evaluated in $(a,b)$, with $h_1 = a$ and $h_2 = b$. Then $\vect{Eab} = \Span_\F(E (h_i)_{1\leq i\leq l})$ and we prove by induction that $E$ is an ideal of $\vect{Eab}$. For $h_1=a$ and $h_2 = b$ we have, for all $e \in E$, $[e,h_1]\in E$ and $[e,h_2]\in E$ as $E$ is an ideal of $\vect{Ea}$ and $\vect{Eb}$. Then for each $k\geq 3$ there exists $i,j\leq k$ such that $[h_i,h_j] = h_k$. Then by induction we see that, for all $e \in E$, $[e,h_k] = [[e,h_i],h_j]+[h_i,[e,h_j]]\in E$ so we conclude.
\end{proof}

\begin{remark}
    For singletons $a,b$ we actually always have $E\lteq \vect{Ea}$ and $E\lteq \vect{Eb}$ iff $E\lteq \vect{Eab}$.
\end{remark}

\begin{definition}\label{def:Malcev}
    A tuple $a =(a_1,\ldots,a_n)$ is called a \textit{Malcev} tuple over an LLA $C$ (or simply \textit{Malcev} over $C$) if $a$ is linearly independent over $C$ and for all $i = 1,\ldots, c$ we have 
    \[\Span_\F(CP_i(\vect{Ca})) = \Span_\F(C P_i(a)).\]
    Here we write $P_{i}(a)$ for the subtuple of $a$ contained in $P_{i}$. 
 If $A = \vect{Ca}$ we call $a$ a \textit{Malcev basis} of $A$ over $C$.
\end{definition}

\begin{remark}\label{rk:malcevtuples} Note that we always have the inclusion $ \Span_\F(CP_i(\vect{Ca})) \supseteq  \Span_\F(C P_i(a))$. Below we list some easy facts.
\begin{enumerate}
    \item If $a = (a_1,\ldots,a_n)$ is Malcev over $C$ then there is a re-indexing of $a$ such that for some $n = k_1\geq \ldots \geq k_c\geq 1$ we have 
    \[\Span_\F(CP_i(\vect{Ca})) = \Span_\F(C a_1\ldots a_{k_i}).\]
    Namely, re-index $a$ so that $\lev(a_{i})\geq \lev(a_{i+1})$ and apply Corollary \ref{cor:spanprop}. Then we also have:
    \[C\lteq \vect{Ca_1}\lteq \ldots \lteq \vect{Ca_1\ldots a_n}\]
    and hence $\vect{Ca_1\ldots a_i} = \Span_\F(Ca_1\ldots a_i)$ for all $i=1,\ldots,n$. We will now call it an \textit{ordered} Malcev basis/tuple.
    \item It is easy to see that for $a = (a_1\ldots a_n)$ with $\lev(a_i)\geq \lev(a_{i+1})$, the tuple $a$ is Malcev over $C$ if and only if $(a_{k+1},\ldots ,a_n)$ is Malcev over $Ca_1\ldots a_k$ and $a_1\ldots a_k$ is Malcev over $C$, for all $1\leq k\leq n$.
    \begin{proof}
        For the forward direction, fix $1\leq k\leq n$ and $A = \vect{Ca_1,\ldots,a_k}$. We prove that $\Span(A P_i(\vect{Aa_{k+1},\ldots,a_n}) = \Span(AP_i(a_{k+1},\ldots,a_n))$. Fix $1\leq i\leq c+1$. From $(1)$ we have $C\lteq \vect{Ca_1}\lteq\ldots\lteq \vect{Ca_1,\ldots,a_n}$ hence $A=\Span(Ca_1,\ldots,a_k)$and $\vect{Aa_{k+1},\ldots,a_n} = \Span(Ca_1,\ldots,a_n)$. As $a$ is Malcev over $C$, it follows that \[P_i(\vect{Aa_{k+1},\ldots,a_n}) \seq \Span(Ca_1,\ldots,a_s)\]
        where $1\leq s\leq n$ is such that $P_i(a) = (a_1,\ldots,a_s)$. If $s\leq k$, then 
        \[\Span(A P_i(\vect{Aa_{k+1},\ldots,a_n}) = A\] and we conclude since $P_i(a_{k+1},\ldots,a_n) = \emptyset$. If $s>k$ then 
        \[\Span(A P_i(\vect{Aa_{k+1},\ldots,a_n}) = \Span(A a_{k+1},\ldots,a_s)\]
        and we conclude since $a_{k+1},\ldots,a_s = P_i(a_{k+1},\ldots,a_n)$. The same sort of argument yields that $a_1,\ldots,a_k$ is Malcev over $C$.
        The converse is a particular case of Lemma \ref{lm:propertiesofmalcevsets} (1) below.
    \end{proof}
    \item By $(1)$ if $a = (a_1,\ldots,a_n)$ is Malcev over $C$ then \[\vect{Ca_1\ldots a_n} = \Span_\F(Ca_1\ldots a_n)\] 
    This is regardless of the indexing of the $a_i$.
    \item For any LLA extension $B\seq A$ there exists an ordered Malcev basis $a_1,\ldots,a_n$ such that $A = \vect{B,a_1,\ldots,a_n} = \Span(B,a_1,\ldots,a_n)$. This is obtained by iteratively taking bases of the complement of $B$ in $P_c(A)$, in $P_{c-1}(A)$, etc. Another way of seeing this: observe that $B$ is an ideal of $\Span(BP_c(A))$ which is an ideal of $\Span(BP_{c-1}(A))$, etc. which is an ideal of $\Span(BP_{1}(A)) = A$ and a Malcev basis is given by taking iteratively bases of $\Span(BP_{i}(A))$ over $\Span(BP_{i+1}(A))$. For such basis, we have $\lev(a_i)\geq \lev(a_{i+1})$ so we see that $a_n$ is of minimal level among the $a_i's$. 
    \item If a tuple $a = (a_1,\ldots, a_n)$ is Malcev over $C$ then it is not necessarily the case that every subtuple of $a$ is Malcev over $C$. To see this, consider a Lie algebra with basis $a_1,a_2,a_3$ such that $[a_2,a_3] = a_1$ and every other bracket $[a_{i},a_{j}]$ with $i < j$ is trivial. Then for $\lev (a_1)\geq \lev (a_2)\geq \lev (a_3)$ we have that $(a_1,a_2,a_3)$ is an ordered Malcev basis of $A = \vect{a_1,a_2,a_3}$, over $\set{0}$ in particular it is Malcev over $\set{0}$ but $a_2,a_3$ is not Malcev over $\set{0}$.
    \item Consider the following example: let $b$ be in $P_1\setminus P_2$ and $B = \Span_\F(b)$. Let $a$ be in $P_2\setminus P_3$ and define the bracket to be trivial on $A = \Span(a,b)$. We have $c = a+b\in P_1\setminus P_2$ and $a+b\in\Span(B P_2(A))$ and $\lev(a+b) = 1$, hence $\lev(\Span(B P_2(A))/B) = 1$. We have $A = \vect{B,a} = \vect{B,c}$. Here both $a$ and $c$ satisfy $\vect{Ba} = \Span(Ba)$ and $\vect{Bc} = \Span(Bc)$ but only $a$ is Malcev over $B$. 
    \item By the previous point if for some $B$ and $a = (a_1,\ldots,a_n)$ we have $\vect{Ba_1\ldots a_n} = \Span(Ba_1\ldots a_n)$ then $a$ is not necessarily Malcev over $B$.
    \item If $\langle Ca \rangle \otimes _C \langle Cb \rangle = \vect{C(h_i)_i} = \Span_\F(C(h_i)_i)$ is as in Corollary \ref{cor:characterisationBaudischfreeamalgammonogeneous}, where $h_1,\ldots,h_k$ are (evaluated) Hall polynomials from $\HS^{\alpha,\beta}$ where $\alpha = \lev(a)$ and $\beta = \lev(b)$. Then $h_k,\ldots,h_1$ is a (ordered) Malcev basis of $\vect{Ca}\otimes_C \vect{Cb}$ over $C$, see Proposition \ref{prop:baudischamalgamhallbasisgivesmalcevbasisandmore}.
\end{enumerate}
\end{remark}

\begin{lemma}\label{lm:propertiesofmalcevsets}
Let $L$ be any LLA and let $a$, $b$ be tuples and $C$ a subalgebra of $L$.
\begin{enumerate}
    \item If $a$ is Malcev over $\vect{Cb}$ and $b$ is Malcev over $C$, then $ab$ is Malcev over $C$.
    \item If $b$ is Malcev over $C$ and $ab$ is Malcev over $C$, then $a$ is Malcev over $Cb$.
    \item If $a$ is Malcev over $\vect{Cb}$ and $ab$ is Malcev over $C$, then $b$ is Malcev over $C$.
\end{enumerate}
In short, any two of the three Malcev conditions between $a,b$ and $C$ above imply the third. 
\end{lemma}
\begin{proof}
For an ease of notation, we say $M(x/D)$ holds if $x$ is Malcev over $\vect{D}$. Recall
\begin{itemize}
    \item [(i)] $M(a/Cb)$ if and only if $a$ is independent over $\vect{Cb}$ and for any $k$ we have $\Span(\vect{Cb}P_k(\vect{Cab})=\Span(\vect{Cb}P_k(a))$;
     \item [(ii)] $M(b/C)$ if and only if $b$ is independent over $C$ and for any $k$ we have $\Span(CP_k(\vect{Cb})=\Span(CP_k(b))$ and
      \item [(iii)] $M(ab/C)$ if and only if $ab$ is independent over $C$ and for any $k$ we have $\Span(CP_k(\vect{Cab})=\Span(CP_k(ab))$.
\end{itemize}
\emph{Proof of $(1)$:} Assume $(i)$ and $(ii)$ from above hold. We want to establish $(iii)$. One easily checks that $ab$ is still linearly independent over $C$. For the equality of the spans, consider $x\in \Span(CP_k(\vect{Cab}) $ arbitrary. We can write 
$$ x= c_1 + y$$
 for some $c_1\in C$ and $y\in P_k(\vect{Cab})$. By $(i)$, we know $x\in \Span(\vect{Cb}P_k(\vect{Cab}))=\Span(\vect{Cb}P_k(a))$, whence 
$$ x = \beta + \sum_i \lambda_ia_i$$
for some $\lambda_i \in \F, a_i$ from $P_k(a)$ and $\beta\in \vect{Cb}$. As  $P_k(L)$ is a subalgebra, observe that also $y-\sum_i \lambda_ia_i=\beta - c_1$ is in $P_k(L)$. Further, as $\beta, c_1\in \vect{Cb}$, we get that actually
$\beta-c_1 \in P_k(\vect{Cb})$ and by  $(ii)$ we infer that $\beta-c_1 \in \Span(CP_k(b))$. Thus, we find $c_2\in C$ and $\mu_i\in \F$ such that $\beta-c_1 = c_2+\sum_i \mu_i b_i$ for $b_i\in P_k(b)$. This yields, 
$$x=\beta + \sum_i \lambda_ia_i = c_1 +c_2+\sum_i \mu_i b_i+\sum_i \lambda_ia_i \in \Span(CP_k(ab)),$$
as desired.\\

\emph{Proof of $(2)$:} Now assume $(ii)$ and $(iii)$ hold. We want to establish $(i)$. To see that $a$ is linearly independent over $\vect{Cb}$, recall that by $(ii)$ and Remark \ref{rk:malcevtuples} (3), we get that $\vect{Cb}=\Span(Cb)$ and use that, by $(iii)$, we know that $ab$ is linearly independent over $C$. 
Now we need to take care of the equality of spans. To this end, consider $x\in \Span(\vect{Cb}P_k(\vect{Cab}))$ arbitrary. Then there are $\beta\in \vect{Cb}$ and $y\in P_k(\vect{Cab})$ such that $x=\beta+y$. By $(iii)$, we get $y\in \Span(CP_k(ab))$, whence $y=c+\sum_i\lambda_ia_i+\sum_i\mu_ib_i$, for some $c\in C, a_i\in P_k(a), b_i\in P_k(b)$ and $\lambda_i,\mu_i\in \F$. Thus 
$$x=\left(\beta+c+\sum_i\mu_ib_i\right)+\sum_i\lambda_ia_i$$
and as $\beta+c+\sum_i\mu_ib_i\in \vect{Cb}$, we conclude $x\in \Span(\vect{Cb}P_k(a))$, as desired. \\

\emph{Proof of $(3)$:}
Finally, assume $(i)$ and $(iii)$ hold. We need to establish $(ii)$. Clearly, $b$ is independent over $C$ by $(iii)$. Now, as above, pick $x\in \Span(CP_k(\vect{Cb})$ arbitrary. By $(iii)$, we have $x\in \Span(CP_k(ab))$. But by choice we know that  $x\in \vect{Cb}$ and as $(i)$ yields that $a$ is linearly independent over $\vect{Cb}$, it is easy to conclude that indeed $x\in \Span(CP_k(b))$, as desired. 
\end{proof}

\begin{proposition}\label{prop:baudischamalgamhallbasisgivesmalcevbasisandmore}
    Let $a,b$ be singletons and let $A =  \vect{Ca}$ and $B= \vect{Cb}$ be basic extensions of $C$. Let $S = A\otimes_C B$ and $ (h_1,\ldots,h_k)$ be an enumeration of the evaluated Hall monomials in $a$ and $b$, with $h_1 = a$, $h_2 = b$, and $h_3 = [a,b]$. Then
    \begin{enumerate}
        \item $(h_1,\ldots,h_k)$ is a Malcev basis of $S$ over $C$,
        \item $(a,h_3,\ldots,h_k)$ is Malcev over $C$,
        \item $(b,h_3,\ldots,h_k)$ is Malcev over $C$,
        \item $(h_3,\ldots,h_k)$ is Malcev over $A$ and over $B$. 
    \end{enumerate}
\end{proposition}
\begin{proof}
\begin{enumerate}
    \item We may assume that $\lev(h_i)\geq \lev(h_{i+1})$ for all $i = 1,\ldots, k-1$. By Corollary \ref{cor:characterisationBaudischfreeamalgammonogeneous} and Lemma \ref{lm:idealspan}, $(h_1,\ldots,h_k)$ is a linear basis of $S$ over $C$. Let $i\in \set{1,\ldots, c}$ and let $n_i\leq k$ such that $P_i(h_1,\ldots,h_k) = (h_{n_i},\ldots, h_k)$. By Theorem \ref{thm:existence_strong_amalgam_monogenous}, $P_i(S)$ is defined as $P_i(C)\oplus P_i(\vect{a,b})$. As $\vect{a,b}\cong F(X,Y,\lev(a),\lev(b))$, we have $P_i(\vect{a,b})=\Span(h_{n_i},\ldots, h_k)$ hence $\Span(CP_i(S)) = \Span(Ch_{n_i}\ldots h_k)$, so $(h_1,\ldots,h_k)$ is a Malcev basis over $C$. 
    \item As in $(1)$, we may assume that $\lev(h_i)\geq \lev(h_{i+1})$ for all $i = 1,\ldots, k-1$ and that $1 = n_1\leq \ldots \leq n_{c} \leq k$ are such that $P_i(h_1,\ldots,h_k) = (h_{n_i},\ldots, h_k)$. It is clear that $h' = (a,h_3,\ldots,h_k)$ is linearly independent over $C$. In particular, we still have $P_i(\vect{Ch'}) = P_i(C)+P_i(\vect{h'})$, so it is enough to prove that $P_i(\vect{h'}) = \Span(P_i(h'))$. First, by Theorem \ref{thm:existence_strong_amalgam_monogenous}, we have $\vect{a,b}\cong F(X,Y,\lev(a),\lev(b))$. In particular, $b\notin \vect{h'}$ and from $\vect{h_1,\ldots,h_k} = \Span(h_1,\ldots,h_k)$ we obtain $\vect{h'} = \Span(h')$. Then 
    \begin{align*}
        P_i(\vect{h'}) &= \vect{h'}\cap P_i(\vect{h_1,\ldots,h_k})\\ 
        &= \Span(h')\cap \Span(h_{n_i},\ldots,h_n) \\
        &=\Span(P_i(h')).
    \end{align*}
    \item Same as $(2)$ by symmetry.
    \item As $A$ is a basic extension of $C$, $a$ is Malcev over $C$. Using $(2)$ and Lemma \ref{lm:propertiesofmalcevsets} (2) we get that $(h_3,\ldots,h_k)$ is Malcev over $A$. The argument for $B$ is symmetric.
\end{enumerate}    
\end{proof}

\begin{remark}[Malcev Calculus]\label{rk:malcevcalculus}
    Lemma \ref{lm:propertiesofmalcevsets} can be seen as a list of basic operations for obtaining new Malcev tuples from old ones. For any tuples $a,b$, we denote $M( a/ b C)$ to express ``$a$ is Malcev over $\vect{C b}$". We have the Malcev triangle:
    \[\begin{tikzcd}[every arrow/.append style={dash}]
        M( a/ C  b)\ar[rr] &  & M( b/C) \ar[ddl]\\
        & & \\
        & M(  a  b / C) \ar[uul]& 
    \end{tikzcd}
    \]
    where by Lemma \ref{lm:propertiesofmalcevsets} every two vertices implies the third. For instance, one easily deduces $M(a/Cbh_3,\ldots,h_k)$ and  $M(bh_3,\ldots,h_k/Ca)$ from Proposition \ref{prop:baudischamalgamhallbasisgivesmalcevbasisandmore}. This ``Malcev calculus" will be heavily used later, in particular in the proof of Theorem \ref{thm:weaktransitivitystageII}. 
\end{remark}

\subsubsection{Rank of LLA extensions} 


Recall (Definition \ref{def:levelofanextension}) that, given LLAs $B\seq A$, the level $\lev(A/B)$ is the maximal $1\leq i\leq c+1$ such that $A = \Span(BP_i(A))$. Recall that as far as levels of elements are concerned, the addition is ``truncated" in $\set{1,\ldots,c+1}$ in the sense that for $i,j\in \set{1,\ldots,c+1}$ we have that $i+j$ takes the value $c+1$ if the numerical value of $i+j$ is $\geq c+1$.
\begin{definition}
    Given LLAs $B\seq A$ with $A$ finite-dimensional over $B$, we define the \textit{rank of $A$ over $B$}, denoted $\rk(A/B)$ to be the pair $(\nu, n)$ where $\nu = \lev(A/B)$ and $n$ is the dimension of $\Span(BP_\nu(A))$ over $\Span(BP_{\nu+1}(A))$ if $\nu\neq c+1$ or $n = 0$ if $\nu = c+1$.
\end{definition}

We order those pairs in a counter-intuitive way: 
\[(\mu,m)\prec (\nu,n) \iff \begin{cases} \mu>\nu\ \text{ or }\\ \nu = \mu\text{ and }m<n\end{cases}.\]

\begin{remark}\label{rk:rankstuff}
Some easy facts.
\begin{enumerate}
    \item $\rk(A/B) = (c+1,0)$ if and only if $A = B$.
    \item Assume that $\rk(A/B) = (\nu,n)$ and $B\subsetneq A$. Let $a_1,\ldots,a_k$ be a Malcev basis of $A$ over $B$. Then $\nu = \lev(A/B)$ is the minimum of the levels of $a_i$ and $n$ is the number of elements among $a_1,\ldots,a_k$ which are of level $\nu$. 
    \item Let $C\seq B\seq A$ be LLA, then $\lev(A/B)\geq \lev(A/C)$. Indeed, for $\nu = \lev(A/C)$ we have $A = \Span(CP_\nu(A))$ hence also $A = \Span(BP_\nu(A))$ so we have 
    $$\nu\leq \max\set{i \leq c+1 \mid A = \Span(BP_{i}(A))} = \lev(A/B).$$ 
    \item Assume that $C\seq B$ are LLAs and for some singleton $a$ we have $B\lteq \vect{Ba}$ then $\lev(\vect{Ba}/C)\geq \lev(B/C)$ hence $\rk(B/C)\preceq \rk(\vect{Ba}/C)$. 
    \begin{proof}
        First, we prove that $\lev(B/C)\geq \lev(\vect{Ba}/C)$. Let $\lev(a) = \mu$ and $\lev(B/C) = \nu$. If $\mu\geq \nu$, then $a\in P_\nu(\vect{Ba})$ and $\vect{Ba} = \Span(CP_\nu(\vect{Ba})$ so $\lev(\vect{Ba}/C) =  \nu= \lev(B/C)$. If $\mu<\nu$ then $\vect{Ba} = \Span(CP_\mu(\vect{Ba})$ and $\lev(\vect{Ba}/C) = \mu<\nu = \lev(B/C)$.
    \end{proof}
    \item For $A = \vect{Ca}$ and $B = \vect{Cb}$ basic extensions of $C$ we have 
    \begin{enumerate}
        \item $\rk(A\otimes_C B/A) = (\lev(b),1)$
        \item $\rk(A\otimes_C B/B) = (\lev(a),1)$
        \item If $A\otimes_C B = \vect{Da} = \vect{D'b}$ for some $A\seq D'\lteq \vect{D'b}$, $B\seq D\lteq \vect{Da}$ then $\rk(D/B) = \rk(D'/A)= (\lev(a)+\lev(b),1)$ This follows from Proposition \ref{prop:baudischamalgamhallbasisgivesmalcevbasisandmore} (4).
    \end{enumerate} 
\end{enumerate}
\end{remark}

\subsubsection{Construction of a $\oslash$-amalgam}

In this subsubsection, we describe the induction scheme along which the free amalgam of a basic extension and an arbitrary extension will be constructed. We are not certain whether the induction suggested by Baudisch in \cite{Baudisch4} (at the bottom of p. 944), corresponds to the one we describe below\footnote{The naive way one would inductively amalgamate a basic extension $A = \vect{Ca}$ of $C$ and an arbitrary extension $B$ over $C$ would be by writing $B = \vect{Cb_1,\ldots,b_n}$ where $(b_1,\ldots b_n)$ is a Malcev basis of $B$ over $C$ and do an induction on $n$. However, if $D$ is the amalgam of $A$ and $\vect{Cb_1\ldots b_{n-1}}$ over $C$ then there is no control of the dimension of $D$ over $\vect{Cb_1,\ldots,b_{n-1}}$, which could be greater than $n$. The notion of rank is there to circumvent this problem.}.



To define the $\oslash$-amalgam $A\oslash_C B$ for a basic extension $A$ of $C$, we proceed by induction on $\rk(B/C)$. More precisely, we prove the following by induction on $\rk(B/C)$:

\begin{center}
$(*)$ $\begin{cases}\fbox{
\begin{minipage}{0.75\textwidth}
For all LLAs $A,B,C$ such that $A = \vect{Ca}$ is a basic extension of $C$ and $C\seq B$, there exists an amalgam $S$ of $A$ and $B$ over $C$ such that: 
\begin{enumerate}[label=(\alph*)]
    \item  there exists $H\lteq S$ containing $B$ such that $S = \vect{Ha}$ is a basic extension of $H$ and $\lev(H/B) = \lev(a)+\lev(B/C)$
    \item $S = \vect{AB}$
\end{enumerate}
We call $S $ a $\oslash$-amalgam of $A$ and $B$ over $C$ if it satisfies those conditions, denoted $S = A\oslash_C B$.
\end{minipage}
}
\end{cases}$
\end{center}

Note that in the above (and below) we identify $A$ and $B$ with their image in the amalgam. For instance, condition $(a)$ and $(b)$ above hold in the amalgam $S$.

We prove $(*)$ by induction on $\rk(B/C)$ for the order $\prec$.

The base case starts with any $A =\vect{Ca}$ and $B = \vect{Cb}$ with $\rk(B/C)$ minimal such that $C\seq B$, i.e. $\rk(B/C) = (c+1,0)$. This means that $B = C$ and the amalgam $S := A$ satisfies $(*)$ by considering $H = C$. Recall that as far as addition of levels is concerned, $c+k = c+1$ for all $k\geq 1$, in particular $c+1 = \lev(C/C) = \lev(a)+\lev(C/C) = \lev(a)+c+1$.


Assume now that for some $(\nu,n)$ we have that $(*)$ holds for any basic extension $A = \vect{Ca}$ of $C$ and $B$ extending $C$ with $\rk(B/C) \prec (\nu,n)$. Fix a basic extension $A = \vect{Ca}$ of $C$ and an extension $B$ of $C$ such that $\rk(B/C) = (\nu,n)$. There exists an ordered Malcev basis $b_1,\ldots,b_s$ of $B$ over $C$ so that $B = \vect{Cb_1,\ldots,b_s} = \Span(Cb_1,\ldots,b_s)$. Then $\nu = \lev(b_s)$. Let $\mu = \lev(a)$. We have $1\leq \mu,\nu\leq c+1$.

Let $A_0 = A$, $C_0 = C$ and let $D_0 = \vect{C,b_1,\ldots,b_{s-1}}$ so that $B = \vect{D_0 b_{s}}$. As $\nu = \lev(b_s)$, we have that $\rk(D_0/C_0)$ is either $(\nu, n-1)$ or $(\nu',k)$ for some $\nu'>\nu$. It follows that $\rk(D_0/C_0)\prec (\nu,n)$ so by the induction hypothesis (with $A = A_0, B = D_0, C = C_0$), there exists a $\oslash$-amalgam $A_1$ of $A = \vect{Ca}$ and $D_0$ over $C_0$, and there exists $C_1$ containing $D_0$ such that $A_1 = \vect{C_1a}$ and $\lev(C_1/D_0)= \lev(a)+ \lev(D_0/C_0)$. As $D_0 = \vect{C,b_1,\ldots,b_{s-1}}$ we have that $\lev(D_0/C_0)\geq \nu$ (it is equal if $n>1$). It follows that $\lev(C_1/D_0)\geq \mu+\nu>\nu$.

Starting with $A_0 = A$, $B_0 = B$, $C_0 = C$ and $D_0 = \vect{C,b_1,\ldots,b_{s-1}}$, we recursively construct sequences $(A_i,B_i,C_i,D_i)_{i\leq t}$ for some $t\leq c+1$ for which $C_t = D_t$ or $D_t = C_{t+1}$ and such that the following holds:
\begin{itemize}
    \item $A_{i+1}$ is a $\oslash$-amalgam of $A_i$ and $D_i$ over $C_i$
    \item $B_{i+1}$ is a $\oslash$-amalgam of $B_i$ and $C_{i+1}$ over $D_i$
    \item $A_i = \vect{C_ia}$ is a basic extension of $C_i$, $B_i = \vect{D_ib}$ is a basic extension of $D_i$
    \item $C_i\seq D_i\seq C_{i+1}$ and $\lev(D_i/C_i)\geq i(\mu+\nu)+\nu$ and $\lev(C_{i+1}/D_i) \geq  (i+1)(\mu+\nu)$
\end{itemize}

We already constructed $C_0,C_1,D_0,A_0,A_1$. Let $B_0 = B$. We refer to Figure \ref{fig:meierscheme} for an overall picture of what is happening.
\begin{enumerate}
    \item \textit{Construction of $B_{i+1},D_{i+1}$ from $D_i, B_i, C_{i+1},C_i$.} By the recursive construction, we have that $\lev(C_{i+1}/D_i) \geq (i+1)(\mu+\nu)>\nu$. It follows that $\rk(C_{i+1}/D_i)\prec (\nu,n)$. We have that $B_i = \vect{D_ib_s}$, so we apply the induction hypothesis $(*)$ interchanging the roles of $a$ and $b_s$ (i.e. with $A = \vect{D_ib_s}$, $B = C_{i+1}$ and $C = D_i$) to get a $\oslash$-amalgam $B_{i+1}$ of $B_i = \vect{D_i b_s}$ and $C_{i+1}$ over $D_i$, and $D_{i+1}$ extending $C_{i+1}$ such that $B_{i+1} = \vect{D_{i+1}b_s}$ is a basic extension of $D_{i+1}$ with $\lev(D_{i+1}/C_{i+1}) = \lev(b_s)+\lev(C_{i+1}/D_i)$. By recursion, $\lev(C_{i+1}/D_i)\geq (i+1)(\mu+\nu)$ hence $\lev(D_{i+1}/C_{i+1})\geq (i+1)(\mu+\nu) + \nu$.
    \item \textit{Construction of $A_{i+1},C_{i+1}$ from $A_i,D_i,C_i$.} By recursion, $\lev(D_i/C_i)\geq i(\mu+\nu)+\nu>\nu$ in particular $\rk(D_i/C_i)\prec (\nu,n)$. As $A_i = \vect{C_ia}$, by the induction hypothesis $(*)$, there exists a $\oslash$-amalgam $A_{i+1}$ of $A_i$ and $D_i$ over $C_i$ and $C_{i+1}$ containing $D_i$ such that $A_{i+1} = \vect{C_{i+1}a}$ is a basic extension of $C_{i+1}$ such that $\lev(C_{i+1}/D_i) =  \lev(a)+\lev(D_i/C_i)$. By recursion $\lev(D_i/C_i) \geq i(\mu+\nu)+\nu$ hence $\lev(C_{i+1}/D_i)\geq (i+1)(\mu+\nu)$.
\end{enumerate}

\begin{figure}
    \centering
    \begin{tikzcd}
    & A_t \otimes_{C_t} B_t & \\
 A_t = A_{t-1}\oslash_{C_{t-1}} D_{t-1} = \vect{C_t a} \ar[ur]  &    & B_t = B_{t-1}\oslash_{D_{t-1}} C_t = \vect{C_t b_s} \ar[ul]\\
    & C_t=D_t \ar[ul] \ar[ur] & B_{t-1} \ar[u]\\
 A_{t-1} \ar[uu] & D_{t-1} \ar[uul] \ar[u] \ar[ur] &  \\
  \vdots \ar[u] & C_{t-1} \ar[ul] \ar[u] & \vdots \ar[uu] \\
A_3 = A_2\oslash_{C_2} D_2 = \vect{C_3 a}  
 \ar[u] &      \vdots \ar[u]   &   \\
      &     C_3 \ar[ul] \ar[u]    &     B_2 = B_1\oslash_{D_1} C_2 = \vect{D_2 b_s} \ar[uu]  \\
A_2 =A_1\oslash_{C_1} D_1 = \vect{C_2a} \ar[uu]  &     D_2   \ar[uul] \ar[u] \ar[ur] &        \\
     &     C_2 \ar[ul] \ar[u] \ar[uur]   &     B_1 = B_0\oslash_{D_0} C_1 = \vect{D_1b_s}\ar[uu] \\
A_1 = A_0\oslash_{C_0} D_0 = \vect{C_1 a} \ar[uu]  &     D_1 \ar[uul] \ar[u] \ar[ur]   &       \\     
     &     C_1 \ar[ul] \ar[u] \ar[uur]  &    B_0 = B = \vect{D_0 b_s} \ar[uu] \\
A_0 = A = \vect{C_0a} \ar[uu] &     D_0 \ar[uul] \ar[u] \ar[ur]   &     \\
     &     C_0 = C \ar[ul] \ar[u] \ar[uur]    &   
\end{tikzcd}
    \caption{Stage II induction scheme}
    \label{fig:meierscheme}
\end{figure}
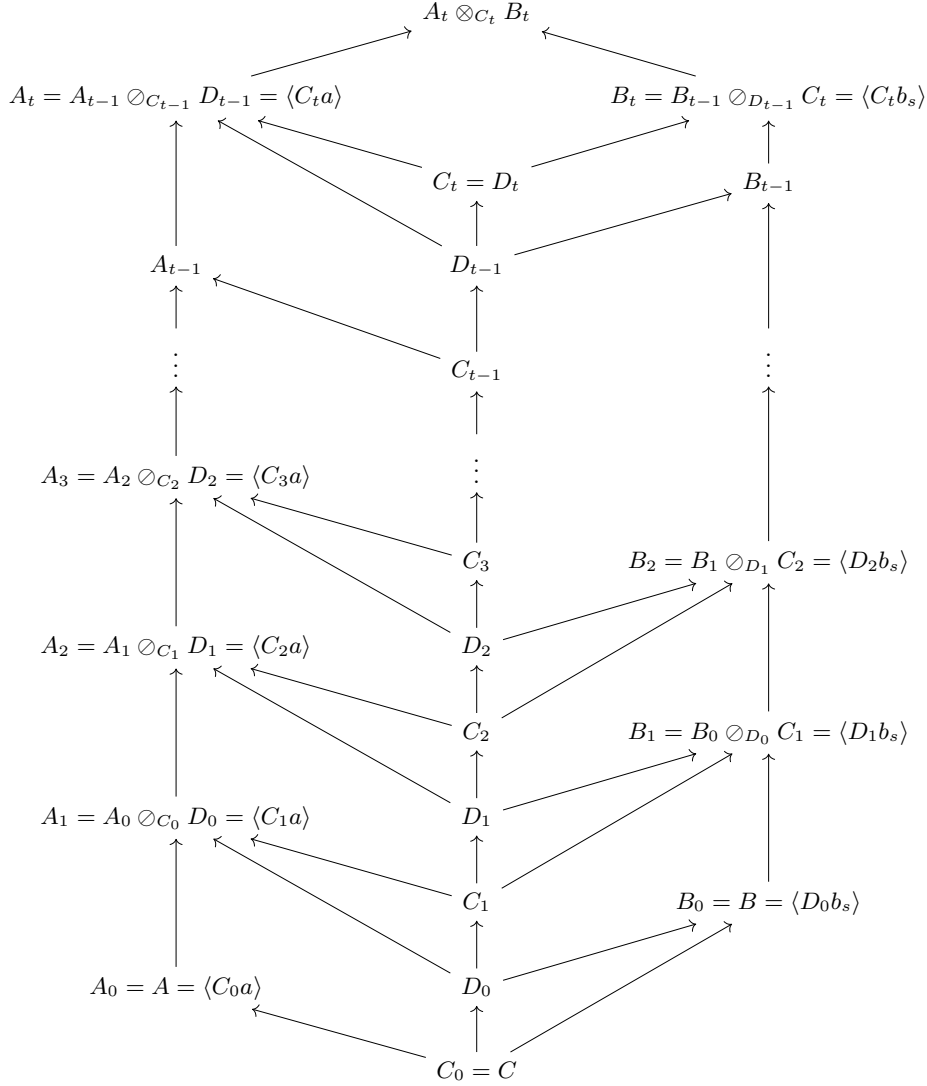

By nilpotence, there exists a smallest $t\in \N$ such that either $C_{t} = D_{t}$ or $D_t = C_{t+1}$. Assume $C_t = D_t$, the other case is treated similarly. We have that $C_{t}$ is an ideal of its basic extensions $A_t = \vect{C_{t}a}$ and $B_t = \vect{C_{t}b_s}$. In that case, there is a free amalgam $S$ of $A_t$ and $B_t$ over $C_{t}$ by Theorem \ref{thm:existence_strong_amalgam_monogenous} and Corollary \ref{cor:characterisationBaudischfreeamalgammonogeneous}. As $C = C_0\seq C_t$, the structure $S$ is an amalgam of $A$ and $B$ over $C$.

It remains to check that $S$ satisfies $(*)$ relative to $A = \vect{Ca}$ and $B$ over $C = C_0$, i.e. that $S = A\oslash_C B$. For $(*)(a)$, we prove that there exists $H\seq S$ containing $B$ such that $S = \vect{Ha}$ and $\lev(H/B)= \lev(a)+\lev(B/C)$. Let $h_1,\ldots,h_k$ be Hall monomials in $a=h_1$ and $b_s=h_2$ with $h_3 = [a,b_s]$ so that $S = \Span(C_{t},b_s,h_k,\ldots,h_3,a)$ and let $H = \vect{C_{t},b,h_k,\ldots,h_3}$. By Proposition \ref{prop:baudischamalgamhallbasisgivesmalcevbasisandmore} (3), $b_s,h_k,\ldots,h_3$ is a Malcev basis of $H$ over $C_t$ and $H = \Span(C_t,b_s,h_k,\ldots,h_3)$. By Proposition \ref{prop:baudischamalgamhallbasisgivesmalcevbasisandmore} (4), $h_k,\ldots,h_3$ is an ordered Malcev basis of $H$ over $\Span(C_{t},b_s)$. We have that $\lev(H/\vect{C_{t}b_s}) = \lev(h_3) = \lev(a)+\lev(b_s) = \mu+\nu$ and $\lev(b_s) = \lev(B/C)$. Note that $H = \Span(\vect{C_tb_s} P_{\mu+\nu}(H))$. As $B\seq \vect{C_{t}b_s}$, if $H = \Span(BP_{i}(H))$ then $H = \Span(\vect{C_tb_s} P_{i}(H))$ so $i\leq \lev(H/\vect{C_tb_s})$. It follows that $\lev(H/B)\leq \lev(H/\vect{C_tb_s}) = \mu+\nu$. Note that $B = \vect{D_0b_s}$. To get $\lev(H/B) \geq \lev(H/\vect{C_{t},b_s})$, it is enough to show that $H\seq \Span_\F(B,P_{\mu+\nu}(H))$. Recall that $H\seq \Span_\F(C_{t},b_s,h_k,\ldots,h_3)$. Now, for each $i\geq 1$ we have $\lev(C_{i}/D_{i-1})\geq \mu+\nu$ and $\lev(D_i/C_i)\geq \mu+\nu$. This implies that there is a basis of $C_{t}$ over $D_0$ in $P_{\mu+\nu}(H)$, hence $C_{t}\seq \Span(D_0 P_{\nu+\mu}(H))$. As $h_k,\ldots,h_3$ are all of degree $\geq \mu+\nu$ we also have $\Span(C_{t}h_k,\ldots,h_3)\seq \Span(D_0 P_{\nu+\mu}(H))$. Finally $\Span(C_{t}h_k,\ldots,h_3,b_s)\seq \Span(D_0b P_{\nu+\mu}(H)) = \Span(B P_{\nu+\mu}(H))$. This proves that $\lev(H/B) = \lev(a)+\lev(B/C)$. 

For $(*)(b)$, first note that $D_0,a,b_s\seq S$ hence $\vect{AB}\seq S$. Conversely, note that $C_0\seq A\seq \vect{AB}$ and $D_0\seq B\seq \vect{AB}$. By construction, $C_{i+1}\seq A_{i+1}$ and by $(\ast)(b)$, $A_{i+1} = \vect{A_i,D_i}$ hence $C_{i+1}\seq \vect{A_i,D_i}$. Similarly, $D_{i+1}\seq \vect{C_{i+1},B_i}$. By induction, $A_i,B_i\seq \vect{AB}$. We conclude $S = \vect{A_tB_t} =  \vect{AB}$, hence we proved $(*)(b)$.

In turn, we have proved that there exists an $\oslash$-amalgam of $A$ and $B$ over $C$.

\subsubsection{A $\oslash$-amalgam is strong}

\begin{theorem}\label{thm:strongamalgam}
Let $A,B,C$ be LLAs such that $A$ is a basic extension of $C$ and $B$ contains $C$. Then a $\oslash$-amalgam of $A$ and $B$ over $C$ is a strong amalgam, i.e. $A\cap B = C$ (in $A\oslash_C B$).
\end{theorem}

\begin{proof}
We proceed using the inductive construction of $A\oslash_C B$ via $\rk(B/C)$, by adding an extra condition in $(*)$, namely the following:

\begin{center}
$(*)$ $\begin{cases}\fbox{
\begin{minipage}{0.75\textwidth}
For all LLAs $A,B,C$ such that $A = \vect{Ca}$ is a basic extension of $C$ and $C\seq B$, there exists an amalgam $S$ of $A$ and $B$ over $C$ such that: 
\begin{enumerate}[label=(\alph*)]
    \item  there exists $H\lteq S$ containing $B$ such that $S = \vect{Ha}$ is a basic extension of $H$ and $\lev(H/B) = \lev(a)+\lev(B/C)$
    \item $S = \vect{AB}$
    \item $S$ is a strong amalgam of $A$ and $B$ over $C$
\end{enumerate}
\end{minipage}
}
\end{cases}$
\end{center}

The base case starts with any $A =\vect{Ca}$ and $B = C$, the amalgam is $S = A$ which clearly satisfies $(*)(c)$.

Now applying $(*)$ in the inductive construction of the $\oslash$-amalgam, we get: 
\begin{itemize}
    \item $A_{i+1}$ is a strong $\oslash$-amalgam of $A_i$ and $D_i$ over $C_i$
    \item $B_{i+1}$ is a strong $\oslash$-amalgam of $B_i$ and $C_{i+1}$ over $D_i$
    \item $A_i = \vect{C_ia}$ is a basic extension of $C_i$, $B_i = \vect{D_ib}$ is a basic extension of $D_i$
    \item $C_i\seq D_i\seq C_{i+1}$ and $\lev(D_i/C_i)\geq i(\mu+\nu)+\nu$ and $\lev(C_{i+1}/D_i) \geq  (i+1)(\mu+\nu)$
\end{itemize}
Again, we refer to the Figure \ref{fig:meierscheme} for an overall picture of what is happening.

Let $t\in \N$ be such that either $C_{t} = D_{t}$ or $D_t = C_{t+1}$. Assume $C_t = D_t$, the other case is treated similarly. We have that $C_{t}$ is an ideal of its basic extensions $A_t = \vect{C_{t}a}$ and $B_t = \vect{C_{t}b_s}$ and $S = \vect{C_t a}\otimes_{C_t} \vect{C_t b_s}$, in particular $S$ is a strong amalgam of $A_t$ and $B_t$ over $C_t$.

We prove that $S$ is a strong amalgam of $A$ and $B$ over $C$. We identify LLAs with their image, so every arrow in Figure \ref{fig:meierscheme} is an inclusion. In particular, $A_0\seq A_1\seq \ldots \seq A_t$ and $B_0\seq B_1\seq \ldots \seq B_t$. By $(*)(c)$, $A_{i}\cap D_i = C_i$ and $B_i\cap C_{i+1} = D_i$, as those are strongly amalgamated at each step. Let $x\in A\cap B$, so $x\in \bigcap_{i=0}^t A_i\cap B_i$. Then $x\in A_t\cap B_t = C_t$. In particular $x\in C_t\cap B_{t-1} = D_{t-1}$, hence $x\in D_{t-1}\cap A_{t-1} = C_{t-1}$. A straightforward iteration gives that $x\in C$. Hence $A\cap B = C$ which gives $(*)(c)$.
\end{proof}

\subsubsection{A $\oslash$-amalgam is free}

\begin{theorem}\label{thm:freeamalgam}
Let $A,B,C$ be LLAs such that $A$ is a basic extension of $C$ and such that $B$ contains $C$. Then the $\oslash$-amalgam of $A$ and $B$ over $C$ is a free amalgam, i.e. $A\oslash_C B \cong A\otimes_C B $. In particular the free amalgam exists.
\end{theorem}

\begin{proof}
We proceed using the inductive construction of $A\oslash_C B$ via $\rk(B/C)$, by adding an extra condition in $(*)$, namely the following:

\begin{center}
$(*)$ $\begin{cases}\fbox{
\begin{minipage}{0.75\textwidth}
For all LLAs $A,B,C$ such that $A = \vect{Ca}$ is a basic extension of $C$ and $C\seq B$, there exists an amalgam $S$ of $A$ and $B$ over $C$ such that: 
\begin{enumerate}[label=(\alph*)]
    \item  there exists $H\lteq S$ containing $B$ such that $S = \vect{Ha}$ is a basic extension of $H$ and $\lev(H/B) = \lev(a)+\lev(B/C)$
    \item $S = \vect{AB}$
    \item $S$ is a strong amalgam of $A$ and $B$ over $C$
    \item for all LLAs $L$ and for all homomorphisms $f:A\to L$ and $g:B\to L$ such that $f\upharpoonright C = g \upharpoonright C$, there exists a homomorphism $h: S\to L$ extending $f$ and $g$;
\end{enumerate}
\end{minipage}
}
\end{cases}$
\end{center}

The base case starts with any $A =\vect{Ca}$ and $B = C$, the amalgam is $S = A$ which clearly satisfies $(*)$.

We assume that $f:A\to L$ and $g:B\to L$ are LLA homomorphisms to an LLA $L$ such that $f\upharpoonright C = g\upharpoonright C$. We denote $f_0 = f$ and $g_0 = g$.

At the first stage of the induction, we have $\rk(D_0/C_0)\prec (\nu,n)$ by Remark \ref{rk:rankstuff} and $A_1 = A_0\oslash_{C_0} D_0$ hence applying the induction hypothesis$(*)(d)$ with $f_0$ and $g_0\upharpoonright D_0$ we get an LLA homomorphism $f_1:A_1\to L$ which extends $f_0$ and $g_0\upharpoonright D_0$ with $f_1\upharpoonright C_0 = g_0\upharpoonright C_0 = f\upharpoonright C_0$. 

At the second stage of the induction, we have $\rk(C_1/D_0)\prec (\nu,n)$ and $B_1 = B_0\oslash_{D_0} C_1$. As $f_1$ extends $g_0\upharpoonright D_0$, we have $f_1\upharpoonright D_0 = g_0\upharpoonright D_0$ hence applying the induction hypothesis$(*)(d)$ with $g_0$ and $f_1\upharpoonright C_1$ we get an LLA homomorphism $g_1:B_1\to L$ which extends $g_0$ and $f_1\upharpoonright C_1$ with $g_1\upharpoonright D_0 = g_0\upharpoonright D_0 = f\upharpoonright C_0$. 

Now applying $(*)$ in the inductive construction of the $\oslash$-amalgam, we get: 
\begin{itemize}
    \item $A_{i+1}$ is a strong $\oslash$-amalgam of $A_i$ and $D_i$ over $C_i$
    \item $B_{i+1}$ is a strong $\oslash$-amalgam of $B_i$ and $C_{i+1}$ over $D_i$
    \item $A_i = \vect{C_ia}$ is a basic extension of $C_i$, $B_i = \vect{D_ib}$ is a basic extension of $D_i$
    \item $C_i\seq D_i\seq C_{i+1}$ and $\lev(D_i/C_i)\geq i(\mu+\nu)+\nu$ and $\lev(C_{i+1}/D_i) \geq  (i+1)(\mu+\nu)$
    \item $f_{i+1}: A_{i+1}\to L$ such that $f_{i+1}$ extends $f_i:A_i\to L$ and $g_i\upharpoonright D_i$ with $f_{i+1} \upharpoonright C_i = f_i\upharpoonright C_i = g_i\upharpoonright C_i$
    \item $g_{i+1}: B_{i+1}\to L$ extends $g_i:B_i\to L$ and $f_{i+1}\upharpoonright C_{i+1}$ with $g_{i+1}\upharpoonright D_i = g_i\upharpoonright D_i = f_{i+1}\upharpoonright D_i$
\end{itemize}

We refer to the Figure \ref{fig:meierscheme} for an overall picture of what is happening.

Let $t\in \N$ such that either $C_{t} = D_{t}$ or $D_t = C_{t+1}$. Assume $C_t = D_t$, the other case is treated similarly. We have that $C_{t}$ is an ideal of its basic extensions $A_t = \vect{C_{t}a}$ and $B_t = \vect{C_{t}b_s}$ and $S = A_t\otimes_{C_t} B_t$. As $f_t: A_t\to L$ and $g_t: B_t\to L$ are such that $f_t\upharpoonright C_t = g_t\upharpoonright C_t$, by Corollary \ref{cor:characterisationBaudischfreeamalgammonogeneous} there exists $h: S\to L$ which extends $f_t$ and $g_t$ and such that $h\upharpoonright C_t = f_t\upharpoonright C_t = g_t \upharpoonright C_t$. As $f_{i+1}\upharpoonright D_i = g_i\upharpoonright D_i$ and $g_{i+1}\upharpoonright C_{i+1} = f_{i+1}\upharpoonright C_{i+1}$ we easily deduce from $C = C_0\seq D_0\seq C_1\seq D_1\seq \ldots $ that $f_0\upharpoonright C = g_0\upharpoonright C = h\upharpoonright C$. Also as $f_{i+1}$ extends $f_i$ for all $i$ and $f_0 = f$, we have that $h$ extends $f:A\to L$. Similarly, $h$ extends $g_0=g:B\to L$, so $S$ satisfies condition $(*)(d)$.
\end{proof}

\subsection{Stage III - From one to many}\label{subsection:stageIII}
We now describe the last stage of the construction of the amalgam.

\begin{theorem}\label{thm:freeamalgamofarbitraryLLA}
    Let $A,B,C$ with $C\seq A$ and $C\seq B$. Then there exists a free amalgam $A\otimes_C B$.
\end{theorem}
\begin{proof}
    By Theorem \ref{thm:freeamalgam}, we know that a free amalgam exists if one of the extensions is basic. Assume now that $A,B,C$ are arbitrary LLA, we inductively construct an amalgam $S$ of $A$ and $B$ over $C$. Let $a_1,\ldots,a_n$ be an ordered Malcev basis of $A$ over $C$. In particular, for $A_i = \vect{Ca_1\ldots a_i}$, $i=1\ldots n$, we have 
    \[C\lteq A_1\lteq \ldots \lteq A_n = A\]
    and $A_{i+1}$ is a basic extension of $A_i$. We define a chain $S_1\seq \ldots\seq S_n$ such that $S_n$ is the required amalgam. Start by taking $S_1 = A_1\otimes_C B$. Then, as $A_2$ is a basic extension of $A_1$, take $S_2 = A_2\otimes_{A_1}S_1$ and recursively if $S_i$ is constructed as $S_i = A_i\otimes_{A_{i-1}} S_{i-1}$, we have $A_{i+1}$ is a basic extension of $A_i$ hence define $S_{i+1} = A_{i+1}\otimes_{A_i} S_i$, until $S := S_n = A\otimes_{A_{n-1}} S_{n-1}$. We refer to Figure \ref{fig:stageIII} for this construction. We check that $S$ is a free amalgam of $A$ and $B$ over $C$. First, we check that $S = \vect{AB}$. As $S = S_n  = A\otimes _{A_{n-1}} S_{n-1}$, we have $S_n = \vect{A,S_{n-1}}$ and more generally, $S_i = \vect{A_i S_{i-1}}$, we immediately get $S = \vect{A S_{i-1}}$ and iteratively $S = \vect{AB}$. We now prove that $S$ is a strong amalgam of $A$ and $B$ over $C$. As $S = A\otimes_{A_{n-1}} S_{n-1}$ we have $A\cap S_{n-1}  = A_{n-1}$. Also, $B\seq S_1\seq \ldots \seq S_n$, so we have $A\cap B\seq A\cap S_{n-1} \seq A_{n-1}\cap S_{n-1}$ and iteratively, $A\cap B\seq A_i \cap S_i$ until $A\cap B\seq A_1\cap B = C$. It remains to check that $S$ satisfies the freeness property. Let $f:A\to L$ and $g:B\to L$ be a homomorphism such that $f\upharpoonright C = g\upharpoonright C$. Consider $f_i = f\upharpoonright A_i$. As $f_1:A_1\to L$ and $g:B\to L$ agree on $C\seq A_1\cap B$ we use $S_1 = A_1\otimes_C B$ to get a map $j_1: S_1\to L$ that extends both $f_1 $ and $g$. Then $f_2$ and $j_1$ agree on $A_1$ hence as $S_2 = A_2\otimes_{A_1} S_1$, there exists a homomorphism $j_2:S_2\to L$ extending $f_2$ and $j_1$. A straightforward iteration yields $j_1,\ldots , j_n$ such that $j_i:S_i\to L$ extends both $f_i$ and $j_{i-1}$ and agrees on $A_{i-1}$. In the end $j_n$ extends both $f$ and $g$.
\begin{figure}
    \centering
    \begin{tikzcd}
            &       S_n = A \otimes _{A_{n-1}} S_{n-1}\\
A = A_n \ar[ur]  &     S_n = A_{n-1}\otimes_{A_{n-2}} S_{n-1} \ar[u] \\
\vdots \ar[u] & \vdots \ar[u] \\
      \vdots      &       S_2 = A_2\otimes_{A_1} S_1\ar[u] \\
 A_2 = \vect{Ca_1a_2}\ar[ur] \ar[u]         &      S_1 = A_1\otimes_C B \ar[u] \\
  A_1 = \vect{Ca_1}  \ar[ur]\ar[u]        &      B\ar[u] \\
C\ar[ur]\ar[u]    & 
\end{tikzcd}
\caption{Stage III induction scheme}
    \label{fig:stageIII}
\end{figure}
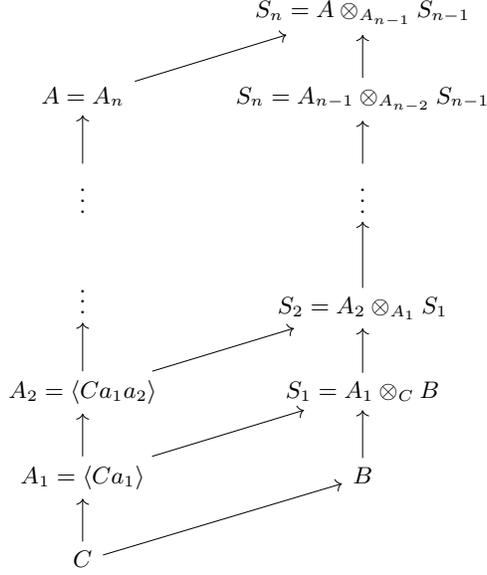
\end{proof}


\subsection{Conclusion}\label{subsec:conclusionfraisseclass}

We refer to Subsection \ref{subsec:generalitiesonfraisseclass} for generalities on Fraïssé theory. Recall the following definition.

\begin{defn}
Let $p$ be a prime number. 
\begin{enumerate}
\item Let $\Lbb_{c,\F}$ be the class of finitely generated Lazard Lie algebras over $\F$ of nilpotency class $\leq c$, in the language $\LL_{c,\F}$.
\item Let $\Lbb_{c,p}$ be the class of finite Lazard Lie algebras over $\F_p$ of nilpotency class $\leq c$, in the language $\LL_c$.
\item We write $\G_{c,p}$ for the class of finite Lazard groups of exponent $p$ and of nilpotency class $\leq c$ in the language of groups expanded by predicates for the Lazard series.
\end{enumerate}
\end{defn}

By Theorem \ref{thm:freeamalgamofarbitraryLLA}, the class $\Lbb_{c,\F}$ has the AP. Note that JEP follows from AP, since $0$ is a common substructure of all structures in the class. As HP is immediate we conclude the following.
\begin{theorem} \label{Lcf has a fraisse limit}
    For any $c\in \N^+$ and field $\F$, the class $\Lbb_{c,\F}$ is a Fraïssé class, with Fra\"iss\'e limit denoted $\mathbf{L}_{c,\mathbb{F}}$.  
\end{theorem}

\begin{remark} \label{rk:uniformlocallyfinite}
 Assume that $A$ is a finite LLA, generated by say $a_1,\ldots,a_n$. Let $\alpha_i = \lev(a_i)$. By the universal property, there exists a surjection $F(X_1,\ldots,X_n,\alpha_1,\ldots,\alpha_n)\to A$. Using Remark \ref{rk:wittformula}, the dimension of $A$ is bounded by 
 \[\sum_{k = 1}^c \frac{1}{k} \sum_{d\mid k} \mu(d)n^{k/d}.\]
 It follows that when $\mathbb{F}$ is a finite field, $\mathbf{L}_{c,\mathbb{F}}$ is uniformly locally finite. 
\end{remark}

As $\Lbb_{c,p}$ is a particular case of $\Lbb_{c,\F}$ and using Lazard correspondence \blue{for $c<p$} (see Subsection \ref{subsec:lazardcorrespondence}) we get the following:

\begin{corollary} \label{cor: limit existence}
    When $\mathbb{F}$ is a finite field, the theory $\Th(\bfL_{c,\mathbb{F}})$ is $\omega$-categorical with quantifier elimination.  Likewise, for $c<p$, the class $\mathbb{G}_{c,p}$ is a Fra\"iss\'e class and, letting $\mathbf{G}_{c,p}$ denote the Fra\"iss\'e limit, the theory $\mathrm{Th}(\mathbf{G}_{c,p})$ is $\omega$-categorical with quantifier-elimination. 
\end{corollary}

We will write $T_{c,\mathbb{F}}$ for the theory $\mathrm{Th}(\mathbf{L}_{c,\mathbb{F}})$ when $\mathbb{F}$ is a finite field or $T_{c,p}$ in the special case that $\mathbb{F} = \mathbb{F}_{p}$.

\begin{corollary}\label{cor:freerelationisanSIR}
    In a monster model of $T_{c,p}$, the relation $\indi \otimes$ is a stationary independence relation in the sense of Tent-Ziegler \cite{tentzieglerurysohn}.
\end{corollary}
\begin{proof}
    Apply Proposition \ref{prop:basicpropertiesoffreeindependence}.
\end{proof}

As the Fraïssé limit of $\Lbb_{c,p}$ it is standard that $\bfL_{c,p}$ is \textit{$\Lbb_{c,p}$-saturated}: for any finite $A\seq \bfL_{c,p}$ (substructure as an LLA) and for any finite LLA $B$ extending $A$ there is a copy of $B$ in $L$ and extending $A$. 

\begin{proposition}\label{prop:algebraicpropertiesL}
    In $\bfL = \bfL_{c,p}$ the following holds.
    \begin{enumerate}
        \item $a\in P_{i+j}$ if and only if there exists $b\in P_i$ and $c\in P_j$ such that $a = [b,c]$.
        \item For $n\leq c-1$, if $a\in P_n\setminus P_{n+1}$ then there exists $b\in \bfL$ such that $[a,b]\in P_{n+1}\setminus P_{n+2}$.
        \item $(P_n)_{1\leq n\leq c+1}$ is the lower central series of $\mathbf{L}$, i.e. $P_{1} = \mathbf{L}$ and $P_{n+1} = [P_{n},\mathbf{L}]$ for all $1 \leq n \leq c$. 
        \item The lower central series and the upper central series of $\mathbf{L}$ coincide.
    \end{enumerate}
\end{proposition}  
    \begin{proof}
    \textit{(1)} Let $a\in P_{i+j}$. Let $A = \vect{a}$. Then as $[a,a] = 0$ $A$ is the abelian LLA on $\Span_{\F_p}(a)$ with $P_1 = P_2 =\ldots =P_{i+j} = A$. Note that anything can happen for $P_{i+j+1},P_{i+j+2},\ldots$ etc, it could be that $P_{i+j+1} = A$ or that $P_{i+j+1} = \set{0}$, depending on the type of $a$. Let $b,c$ be linearly independent over $a$ and define a Lie algebra structure on $B := \Span(a,b,c)$ by the following: $[a,b] = [a,c] = 0$ and $[b,c] = a$. We check the Jacobi identity:
    \begin{align*}
        0 = [a,a] &= [a,[b,c]]\\
        &= [[a,b],c]+[b,[a,c]]\\
        &=0+0
    \end{align*}
    The bracket thus defined is a Lie algebra which is often referred to as the Heisenberg algebra.  To define the predicates, assume $i<j$. Define $P_i = P_{i-1}=\ldots =P_1 = B$, $P_{i+1} = \Span(c,a)$ so that $b\in P_i\setminus P_{i+1}$, and similarly $P_{j+1} = \Span(a)=\ldots=P_{i+j}$, so that $b$ is of degree $i$ and $c$ is of degree $j$. Define $P_{i+j+1},P_{i+j+2}, \ldots$ as in $A$. The Lie bracket defined above is compatible with the Lazard predicates because $b\in P_i$, $c\in P_j$ and $a\in P_{i+j}$. By $\Lbb_{c,p}$-saturation, there exists a copy of $B$ in $L$ hence there is some $b',c'\in \bfL$ such that $a = [b',c']$ and $b'\in P_i(\bfL)$, $c'\in P_j(\bfL)$.

    \textit{(2)} Consider $A = \vect{a}$ then $P_1(A) =\ldots = P_n(A) = A$ and $P_{n+1}(A) = \ldots = P_c(A) = \set{0}$. Let $b,c$ be independent elements and consider $B = \Span(a,b,c)$. Define the bracket: $[a,b] = c$, $[a,c] = [b,c] = 0$. One Jacobi identity to check is enough: 
    \begin{align*}
        0 = [a,0] &= [a,[b,c]]\\
        &= [[a,b],c]+[b,[a,c]]\\
        &=[c,c]+0=0
    \end{align*}
    It remains to define the predicates: the only nontrivial relation is $[a,b] = c$ hence we may put $b\in P_1\setminus P_2$ and $c\in P_{n+1}\setminus P_{n+2}$ since $a\in P_n\setminus P_{n+1}$ (which means: $P_1 = B, P_2 = \Span(a,c) = P_3 = \ldots P_n$ then $P_{n+1}=\Span(c)$, $P_{n+2} = \set{0} = \ldots =P_c$). This defines an LLA structure on $B$ which extends the LLA $A$, hence we conclude by $\Lbb_{c,p}$-saturation.

\textit{(3)} Immediate from \textit{(1)}.

\textit{(4)} Recall that the upper central series is defined by $Z_1 = Z(\bfL)$ and
$$
Z_{i+1} = \set{a\in \bfL\mid [a,b]\in Z_i \text{ for all $b\in \bfL$}}.
$$
We show that $Z_n = P_{c-n+1}$. As $Z_1 = P_c$, by induction, we assume that $P_{c-n+2} = Z_{n-1}$. First if $a\in P_{c-n+1}$, then for all $b$, $[a,b]\in P_{c-n+2} = Z_{n-1}$ hence by definition $a\in Z_n$. Conversely assume that $a\in Z_n$ and $a\notin P_{c-n+1}$. Using \textit{(2)} above, there exists $b\in L$ such that $[a,b]\notin P_{c-n+2} = Z_{n-1}$, which contradicts $a\in Z_n$.
\end{proof}

Let $(L,+,0,[\cdot,\cdot],(P_{i})_{1\leq i\leq c+1})$ be any LLA which is a model of $T_{c,p}$. Then using the Lazard correspondence, there exist $0$-definable functions $\cdot, ^{-1}$ such that $(L,\cdot,^{-1},1,(P_i)_{1\leq i\leq c+1})$ is a $c$-nilpotent group with a Lazard series $(P_i(L))_{1\leq i\leq c+1}$ and $(L,+,0,[\cdot,\cdot])$ and $(L,\cdot,^{-1})$ are interdefinable. Thus, when considering models of $T_{c,p}$, we implicitly consider it to be equipped with both a Lie algebra and a group structure, $L = (L,+,0,[\cdot,\cdot], \cdot,^{-1},1,(P_{i})_{1\leq i\leq c+1})$.

\begin{corollary}
    Let $T_{c,p}^{\mathrm{grp}}$ be the theory of $\bfG_{c,p}$ in the language of groups $\set{\cdot,^{-1},1}$. Then $T_{c,p}^{\mathrm{grp}}$ is the model-companion of the theory of $c$-nilpotent groups of exponent $p$. Similarly if $T_{c,p}^\mathrm{Lie}$ is the reduct of $\bfL_{c,p}$ to the language of Lie rings $\set{+,-,0,[\cdot,\cdot]}$, then $T_{c,p}^\mathrm{Lie}$ is the model-companion of the theory of $c$-nilpotent Lie algebras over $\F_p$.
\end{corollary}

\begin{proof}
    We first check that $T_{c,p}^{\mathrm{Lie}}$ is model-complete. It is enough to prove that the predicates $P_i$ are both existentially and universally definable in $\set{+,-,0,[\cdot,\cdot]}$. Using Proposition \ref{prop:algebraicpropertiesL} $(3),(4)$, the following are equivalent:
    \begin{enumerate}
        \item $x\in P_n$;
        \item $\exists y_1,\ldots,y_n [y_1,\ldots, y_n] = x$;
        \item $\forall y_1,\ldots,y_{c-n+1} [y_1,\ldots,y_{c-n+1},x] = 0$.
    \end{enumerate}
    This gives that $T_{c,p}^{\mathrm{Lie}}$ is model-complete. It remains to check that $T_{c,p}^{\mathrm{Lie}}$ is a companion of the theory of $c$-nilpotent Lie algebras over $\F_p$. To see this, observe that any $c$-nilpotent Lie algebra may be equipped with predicates for a Lazard series (for instance, by interpreting the predicates to coincide with the lower central series) to get an LLA, which may then be embedded in a model of $T_{c,p}$ by standard arguments. Then by forgetting the predicates, we get the result. Using the Lazard correspondence, the same transfers to $T_{c,p}^{\mathrm{grp}}$.
\end{proof}

\begin{remark}[An explicit axiomatization of $T_{2,p}^\mathrm{grp}$]
    In \cite{saracino1979periodic}, Saracino and Woods give explicit axioms for the model companion $T^m$ of the theory of $2$-nilpotent groups of exponent $m\in \N$. By uniqueness of the model companion we conclude that $T^p = T_{2,p}^\mathrm{grp}$. In particular, their axiomatization of $T^{m}$, then, entails that the theory $T_{2,p}^\mathrm{grp}$ may be axiomatized as follows:
    \begin{itemize}
        \item the center $Z$ is infinite and equals the set of commutators;
        \item for each $a_1,\ldots,a_n$ $\F_p$-linearly independent over $Z$, for each $c_1,\ldots,c_n\in Z$ there exists $b$ such that $[a_1,b] = c_1,\ldots,[a_n,b] = c_n$.
    \end{itemize}
\end{remark}

\begin{corollary}
    For $c \geq 4$ and $c<p$, $\bfL_{2,p}$ is definable in $\bfL_{c,p}$.
\end{corollary}

\begin{proof}
Assume $c\geq 4$ and let $k = \lfloor \frac{c}{2} \rfloor$ be the floor of $c/2$. As $c\geq 4$ we have $k\geq 2$. Then $2k$ is either $c$ or $c-1$ so $P_{2k}\neq 0$. Also, $3k = 2k+k\geq c-1+k\geq c+1$ so that $[P_{k},P_{2k}] = [P_{2k},P_{2k}] = 0$. Now for $H = P_k$ and $K = P_{2k}$, we conclude that $H$ is a Lie algebra which is nilpotent of class $2$ with center contained in $K$. We equip $H$ with the Lazard series $Q_1 = H, Q_2 = K, Q_3 = 0$. We prove that $(H,(Q_i)_{1\leq i\leq 3})$ is existentially closed in the class of 2-nilpotent LLAs. Let $(A,(Q_i)_{1\leq i\leq 3})$ be a 2-nilpotent LLA extending $(H,(Q_i)_{1\leq i\leq 3})$. We define a Lazard series of length $c$ on $A$: set $P_1(A) = \ldots = P_k(A) = Q_1(A) = A$ and $P_{k+1}(A) = \ldots = P_c(A) = Q_2(A)$ and $P_{c+1}(A) = 0$. Then $(A,(P_i)_{1\leq i\leq c+1})$ is an LLA extension of $\bfL_{c,p}$. It is clear that any existential formula $\exists x\phi(x)$ in $(A,(Q_i)_{1\leq i\leq 3})$ with parameters in $H$ can be translated into an equivalent statement $(\exists x\in P_k )\Tilde{\phi}(x)$ in $\set{+,-,0,[\cdot,\cdot],(P_i)_{1\leq i\leq c+1}}$ which will be true in $\bfL_{c,p}$ and be translated back in $\set{+,-,0,[\cdot,\cdot],(Q_i)_{1\leq i\leq 3}}$ so that $H\models \exists x\phi(x)$.  As the $2$-nilpotent LLA that we defined is clearly $\aleph_{0}$-categorical and every $2$-nilpotent LLA embeds into it, we must have that it is $\mathbf{L}_{2,p}$. 
\end{proof}

\begin{remark}
    The same method should yield that $\bfL_{n,p}$ is definable in $\bfL_{c,p}$ for $c$ sufficiently larger than $n$.
\end{remark}

\begin{question}
    Is $\bfL_{2,p}$ definable/interpretable in $\bfL_{3,p}$?
\end{question}





\subsection{An extra result on Malcev sets and free amalgamation}\label{subsec:extraNSOP4}

In the following proof, we use that $\indi \otimes$ satisfies transitivity and monotonicity, hence it uses Corollary \ref{cor:freerelationisanSIR} (for monotonicity). Note that we believe that a direct proof exists without using monotonicity, via modifying $(**)(e)$ and by an induction showing that constructing the free amalgam of $a$ and $b$ over $C$ automatically yields a free amalgam of $a$ and $b$ over $E$ (under the assumptions below).
\begin{theorem}\label{thm:weaktransitivitystageII}
Let $A = \vect{Ca}$ and $B = \vect{Cb_1,\ldots, b_s}$. Let $E\seq C$. Assume that 
\begin{itemize}
    \item $a$ is Malcev over $C$ and $E$
    \item $(b_1\ldots b_s)$ is Malcev over $C$ and $E$.
\end{itemize}
If $a\indi \otimes _C b_1\ldots b_s$ then $a\indi\otimes _E b_1\ldots b_s$. 
\end{theorem}

\begin{proof}
We prove the theorem by re-writing the induction in stage II with extra assumptions at each step. We refer to Figure \ref{fig:meierscheme} for a picture of the induction scheme. By Theorem \ref{thm:freeamalgam} we know that $\oslash = \otimes$. We assume that at each stage of the recursion, the $\otimes$-amalgam satisfies the following property:
\begin{center}
$(**)$ $\begin{cases}\fbox{
\begin{minipage}{0.75\textwidth}
For all LLAs $A,B,C$ such that $A = \vect{Ca}$ is a basic extension of $C$ and $C \subseteq B$, there exists an amalgam $S = A\otimes _C B$ of $A$ and $B$ over $C$ such that: 
\begin{enumerate}[label=(\alph*)]
    \item[$ $] for all $E\seq C$, if $B = \vect{Cb_1,\ldots,b_s}$ where $b = (b_1\ldots b_s)$ is ordered Malcev over both $C$ and $E$, and $A = \vect{Ca}$ where $a$ is Malcev over $C$ and $E$ (i.e. $C\lteq \vect{Ca}$, $E\lteq \vect{Ea}$), there exists a tuple $v$ such that 
    \begin{enumerate}
        \item $(a, b,  v)$ is a Malcev basis of $S$ over $C$ and $\lev(v) = \lev(a)+\lev(b)$. 
        \item $(a, b,  v)$ is Malcev over $E$. 
        \item $v$ is Malcev over $B$ and $\vect{Eb}$
        \item $a$ is Malcev over $\vect{Bv}$ and $\vect{Ebv}$
        \item $a\indi \otimes _E b$
    \end{enumerate}
\end{enumerate}
\end{minipage}
}
\end{cases}$
\end{center}

Observe first that, assuming $(**)(a)$, we may consider $H = \vect{Bv}$ which is an ideal of $S$ satisfying $S = \vect{Ha}$ and $\lev(H/B) = \lev(v) = \lev(a)+\lev(B/C)$. This gives $(*)(a)$ from the previous inductive constructions and hence we will use the induction scheme as above.

We assume that $b_1,\ldots, b_s$ is an ordered Malcev basis of $B$ over $C$ which is Malcev over $E\seq C$. In particular, by Remark \ref{rk:malcevtuples} (2) $b_1,\ldots, b_{s-1}$ is Malcev over $C$ and $E$ and $b_s$ is Malcev over $Cb_1\ldots b_{s-1}$ and $Eb_1\ldots b_{s-1}$. Let $S = A\otimes _C B$ be the $\oslash$-amalgam constructed above. We apply the induction hypothesis at each stage of the construction. Let $E_0 = E, F_0 = \vect{Eb_1\ldots b_{s-1}}$.

At the first stage of the construction, we have that $a$ is Malcev over $C = C_0$ and over $E = E_0$, and $b_1,\ldots,b_{s-1}$ is a Malcev basis of $D_0$ over $C_0$ and Malcev over $E_0$.  As $\rk(D_0/C)\prec (\nu,n)$, by the induction hypothesis, there exists a tuple $u_1$ such that 
\begin{itemize}
        \item $(a,b_1,\ldots,b_{s-1},u_1)$ is a Malcev basis of $A_1 = A_0\otimes_{C_0} D_0$ over $C_0$
        \item $u_1$ is Malcev over $D_0$ and over $F_0 = \vect{E_0b_1\ldots b_{s-1}}$
        \item $a$ is Malcev over $C_1 = \vect{D_0 u_1}$ and over $E_1 := \vect{E_0b_1\ldots b_{s-1} u_1} = \vect{F_0 u_1}$
        \item $a\indi \otimes _E b_1\ldots b_{s-1}$, i.e. $a\indi \otimes_{E_0} F_0$
\end{itemize}
Set $C_1 = \vect{D_0 u_1}$. Then we have $A_1 = \vect{C_1a}$ and $C_1\lteq A_1$.

At the second stage of the construction, we have $C_1 = \Span(D_0 u_1)$. We consider $F_0 \seq D_0$. We have that $b_s$ is Malcev over $D_0$ and over $F_0$ and $u_1$ Malcev over $D_0$ and $F_0$. As $\rk(C_1/D_0) \prec (\nu,n)$, by induction hypothesis, there exists a tuple $v_1$ such that 
\begin{itemize}
        \item $(b_s,u_1,v_1)$ is a Malcev basis of $B_1 = B_0\otimes_{D_0} C_1$ over $D_0$
        \item $v_1$ is Malcev over $C_1$ and over $E_1 = \vect{F_0u_1}$
        \item $b_s$ is Malcev over $D_1 = \vect{C_1 v_1}$ and $F_1 := \vect{E_1 v_1}$
        \item $b_s\indi \otimes _{F_0} E_1$
\end{itemize}
Set $D_1 = \vect{C_1 v_1}$ so that $B_1 = \vect{D_1 b_s}$ and $D_1\lteq B_1$.

\textit{Construction of $E_{i+1}, u_{i+1}$ from $F_i,v_i,E_i$.} 

Suppose that we have already constructed $E_i,F_i,v_i$ such that $v_i$ is a Malcev basis of $D_i$ over $C_i$ and $v_i$ is Malcev over $E_i$. We also have that $a$ is Malcev over $C_i$ and over $E_i$. By induction, there exists a tuple $u_{i+1}$ such that 
\begin{itemize}
        \item $(a,v_{i},u_{i+1})$ is a Malcev basis of $A_{i+1} = A_i\otimes_{C_i} D_i$ over $C_i$
        \item $u_{i+1}$ is Malcev $D_i$ and over $F_i$
        \item $a$ is Malcev over $C_{i+1} = \vect{D_i u_{i+1}}$ and over $E_{i+1} := \vect{F_i u_{i+1}}$
        \item $a\indi \otimes _{E_{i}} F_{i}$
\end{itemize}
Set $C_{i+1} = \vect{D_i u_{i+1}}$ so that $A_{i+1} = \vect{C_{i+1}a}$ and $C_{i+1}\lteq A_{i+1}$.

\textit{Construction of $F_{i+1},v_{i+1}$ from $E_{i+1},u_{i+1}, F_i$.}

Suppose that we have already constructed $E_{i+1},u_{i+1},F_i$ such that $u_{i+1}$ is a Malcev basis of $C_{i+1}$ over $D_i$ and $u_{i+1}$ is Malcev over $F_i$. We also have that $b_s$ is Malcev over $D_i$ and over $F_i$. By induction, there exists a tuple $v_{i+1}$ such that 
\begin{itemize}
        \item $(b_s,u_{i+1},v_{i+1})$ is a Malcev basis of $B_{i+1} = B_i\otimes_{D_i} C_{i+1}$ over $D_i$ 
        \item $v_{i+1}$ is Malcev over $C_{i+1}$ and over $E_{i+1}$
        \item $b_s$ is Malcev over $D_{i+1} = \vect{C_i v_{i+1}}$ and over $F_{i+1} := \vect{E_i v_{i+1}}$
        \item $b_s\indi \otimes _{F_{i}} E_{i+1}$
\end{itemize}
Set $D_{i+1} = \vect{C_{i+1} v_{i+1}}$ so that $B_{i+1} = \vect{D_{i+1} b_s}$ and $D_{i+1}\lteq B_{i+1}$.

As the rank drops at each stage there exists $t$ such that $C_t = D_t$ or $D_t = C_{t+1}$. Let us assume the former, the other case is treated similarly.

If $C_t = D_t$, we have that $v_t = \emptyset$, whence $E_t = F_t$ and $C_t$ is an ideal of both $\vect{C_t a}$ and $\vect{C_tb_s}$ and $S$ is given by $\langle  C_{t}a \rangle \otimes _{C_t} \langle C_{t} b_s \rangle$. Now let $u_{t+1}$ be such that $S = \vect{C_t a b_s u_{t+1}}$ as in Proposition \ref{prop:baudischamalgamhallbasisgivesmalcevbasisandmore}, i.e. such that $(a,b_s,u_{t+1}) = (h_1,\ldots ,h_k)$. Let $v = u_1v_1\ldots u_{t-1}v_{t-1}u_t u_{t+1}$. We need to check that $S$ satisfies $(**)$ with the tuple $v$. Recall that $E = E_0$ and $C = C_0$. 

We prove $(a)$. Let $v_0 = (b_1,\ldots,b_{s-1})$. By construction we have $C_i = \vect{Cv_0 u_1v_1\ldots v_{i-1} u_i}$ and $D_i = \vect{Cv_0 u_1v_1\ldots u_{i} v_i}$. We also have that $u_{i+1}$ is Malcev over $D_i$ and $v_{i+1}$ is Malcev over $C_{i+1}$. We conclude that $u_{i+1}$ is Malcev over $Cv_0 u_1v_1\ldots u_{i} v_i$ and $v_{i+1}$ is Malcev over $Cv_0 u_1v_1\ldots v_{i} u_{i+1}$. As $v_0$ is Malcev over $C$ and $u_1$ is Malcev over $Cv_0$ we have by Lemma \ref{lm:propertiesofmalcevsets} (1) that $v_0u_1$ is Malcev over $C$. As $v_1$ is Malcev over $Cv_0u_1$ we have by Lemma \ref{lm:propertiesofmalcevsets}(1) that $v_0u_1v_1$ is Malcev over $C$. An easy iteration using Lemma \ref{lm:propertiesofmalcevsets}(1) yields that $v_0u_1v_1\ldots v_{t-1} u_t$ is Malcev over $C$. Then the last amalgam yields $ab_s u_{t+1}$ is Malcev over $Cv_0\ldots u_{t-1} v_{t-1} u_t$ hence a last use of Lemma \ref{lm:propertiesofmalcevsets}(1) yields that $ab_s v_0\ldots v_{t-1}u_t u_{t+1} = a b_1\ldots b_s v$ is Malcev over $C$. The same argument yields that $(b)$ also holds. Concerning the equality in $(a)$, first, for $i\leq t$ we have $\lev(u_i)\geq \lev(a)+\lev(b)$ and $\lev(v_i)\geq \lev(a)+\lev(b)$. Also, $[a,b_s]\in u_{t+1}$ and is the element of lowest level in $u_{t+1}$, whence $\lev(v) = \lev(a)+\lev(b)$. 

We check $(c)$. First, $a$ is Malcev over $\vect{C_t b_s u_{t+1}} = \vect{C_0 b_s v_0 u_1v_1\ldots u_t u_{t+1}}$. As $b = v_0b_s$ and $v = u_1v_1\ldots u_t u_{t+1}$, we get that $a$ is Malcev over $Cbv$. As $(a,b,v)$ is Malcev over $C$, by Lemma \ref{lm:propertiesofmalcevsets}(3), we have $bv$ is Malcev over $C$. As $b v$ is Malcev over $C$ and $b$ is Malcev over $C$ we conclude by Lemma \ref{lm:propertiesofmalcevsets}(2) that $v$ is Malcev over $\vect{Cb}=B$. A similar argument yields that $v$ is also Malcev over $\vect{Eb}$.

We check $(d)$. By Proposition \ref{prop:baudischamalgamhallbasisgivesmalcevbasisandmore} (1) and (3), we have that $ab_su_{t+1}$ is Malcev over $C_t$ and $b_s u_{t+1}$ is Malcev over $C_t$, hence by Lemma \ref{lm:propertiesofmalcevsets} (2) we get that $a$ is Malcev over $\vect{C_t b_s u_{t+1}}$. As $C_t = \vect{Cv_0 u_1v_1\ldots v_{t-1} u_t}$, we have $\vect{C_t b_s u_{t+1}} = \vect{C b v}$ hence $a$ is Malcev over $Cb v$.

As for $(e)$, working in $S$, we assume that $a\indi \otimes _{C_t} b_s$. Then, $a$ is Malcev over $C_t$ and over $E_t$ and $b_s$ is Malcev over $D_t = C_t$ and over $F_t = E_t$. In other words, $E_t$ is an ideal of $\vect{E_ta}$ and $\vect{E_t b_s}$, so, by Lemma \ref{cor:droppingtheEbaudischcase}, we have $a\indi \otimes _{E_t} b$. We also have by construction that $b_s \indi \otimes _{F_i} E_{i+1}$ and $a\indi \otimes _{E_i} F_i$ for all $i = 0,\ldots, t-1$. From $a\indi \otimes _{E_t} b_s$ and $b_s\indi \otimes _{F_{t-1}} E_t$ we get $a\indi \otimes _{F_{t-1}} b_s$ by transitivity, symmetry, and monotonicity (Proposition \ref{prop:basicpropertiesoffreeindependence}). As $a\indi \otimes _{E_{t-1}} F_{t-1}$, we get $a\indi \otimes _{E_{t-1}} b_s$ by transitivity and monotonicity. A direct induction yields $a\indi \otimes _{F_0} b_s$. As $F_0 = \vect{Eb_1\ldots b_{s-1}}$ and $a\indi \otimes _E Eb_1\ldots b_{s-1}$, we use transitivity to conclude $a\indi \otimes _E b_1\ldots b_s$.
\end{proof}

The following will be crucial in order to prove that the theory of the generic LLA is NSOP$_4$.

\begin{theorem}\label{thm:weaktransitivitygeneralcase}
    Assume that $a = (a_1,\ldots,a_n)$, $b = (b_1,\ldots,b_m)$ are both Malcev over $C$ and over some sub-LLA $E$ of $C$. If $a\indi \otimes_C b$, then $a\indi \otimes_E b$.
\end{theorem}
\begin{proof}
We assume that $a_1\ldots a_n$ is an ordered Malcev basis of $B$ over $C$, so by Remark \ref{rk:malcevtuples} $(2)$, $a_{k+1},\ldots,a_n$ is Malcev over $Ca_1\ldots a_k$ and $Ea_1\ldots a_k$, for all $k = 1\ldots n$. Using Theorem \ref{thm:weaktransitivitystageII} at the first stage of Figure \ref{fig:stageIII}, we have that $a_1\indi \otimes_E b$. At the second stage, we use again Theorem \ref{thm:weaktransitivitystageII} where $\vect{Ea_1}$  plays the role of $E$ to get $a_2\indi \otimes _{\vect{Ea_1}} b$. Iterating, we get $a_k \indi \otimes_{\vect{Ea_1\ldots a_{k-1}}} b$ for all $k = 1\ldots n$. Using transitivity of $\indi \otimes $ (Proposition \ref{prop:basicpropertiesoffreeindependence}), we conclude $a_1a_2\ldots a_k \indi \otimes _E b$ for all $k = 1\ldots n$ hence $a\indi \otimes _E b$.
\end{proof}

\section{Neostability properties of $c$-nilpotent groups and Lie algebras}\label{sec:neostabilityproperties}

In this section, we analyze the model-theoretic properties of the theories of Fra\"iss\'e limits of LLAs over a finite field and, via the Lazard correspondence, deduce parallel classification results for the associated groups. For the entirety of the section, we will fix a \emph{finite} field $\mathbb{F}$, a natural number $c$, and a prime number $p$. We will denote the monster model of $T_{c,\mathbb{F}}$ by $\mathbb{M}_{c,\mathbb{F}}$ and we will write $\mathbb{M}_{c,p}$ for $\mathbb{M}_{c,\mathbb{F}}$ when $\mathbb{F} = \mathbb{F}_{p}$.  We begin with some preliminary observations on definability in these theories.  

\subsection{Definability}

\subsubsection{Flags} We define $\LL_{\mathrm{flag},c, \mathbb{F}}$ by 
$$
\LL_{\mathrm{flag},c,\mathbb{F}} = \{0,+,-, (P_{i})_{1 \leq i \leq c+1}, (\lambda \cdot )_{\lambda \in \mathbb{F}}\}.
$$
When $\mathbb{F} = \mathbb{F}_{p}$, we may omit the scalar multiplication functions from the language.  

Let $\mathbb{K}_{\mathrm{flag},c,\mathbb{F}}$ denote all finite dimensional $\mathbb{F}$-vector spaces $V$ viewed as $\LL_{\mathrm{flag},c,\mathbb{F}}$-structures in which the $\lambda \cdot$ are interpreted as the map $x \mapsto \lambda \cdot x$ and the $P_{i}(V)$s are subspaces of $V$ satisfying 
$$
V = P_{1}(V) \supseteq P_{2}(V) \supseteq \ldots \supseteq P_{c+1}(V) = 0. 
$$
The following observation is easy.

\begin{lem} \label{flags are stable}
Fix $c \geq 1$.
    \begin{enumerate}
        \item The class $\mathbb{K}_{\mathrm{flag},c,\mathbb{F}}$ is a Fra\"iss\'e class.
        \item For finite $\mathbb{F}$, the theory $T_{\mathrm{flag},c,\mathbb{F}}$ of the Fra\"iss\'e limit is $\aleph_{0}$-categorical, $\omega$-stable, and has elimination of quantifiers. 
    \end{enumerate}
\end{lem}

\begin{proof}
(1) To check that $\mathbb{K}_{\mathrm{flag},c,\mathbb{F}}$ is a Fra\"iss\'e class, we note that the hereditary property is easy and joint embedding follows from the amalgamation property since the trivial subspace is a subspace of all structures in $\mathbb{K}_{\mathrm{flag},c,\mathbb{F}}$. If $A,B,C$ are structures in $\mathbb{K}_{\mathrm{flag},c,\mathbb{F}}$ with $A \subseteq B,C$ and $B \cap C = A$, then we can define an amalgam $D$ by taking $D$ to be the $\mathbb{F}$-vector space which is the direct sum of $B$ and $C$ over $A$ interpreting $P_{i}(D)$ to be the span of $P_{i}(B) \cup P_{i}(C)$ for all $1 \leq i \leq c+1$. This proves that $\mathbb{K}_{\mathrm{flag},c,\mathbb{F}}$ is a Fra\"iss\'e class.

(2) Quantifier elimination and $\aleph_{0}$-categoricity follow from (1) and the fact that $\mathbb{K}_{\mathrm{flag},c,\mathbb{F}}$ is uniformly locally finite:  a structure in $\mathbb{K}_{\mathrm{flag},c,\mathbb{F}}$ generated by $n$ elements has cardinality at most $|\mathbb{F}|^{n}$.  Then $\omega$-stability follows easily by quantifier elimination.  Over a countable model $V$, there are at most $ \aleph_{0}$ non-algebraic $1$-types $p(x)$ which assert $x \in P_{i} \setminus P_{i+1}$, determined by specifying the coset $x + P_{j}(V) = v + P_{j}(V)$ or asserting that $x + P_{j}(V) \neq v + P_{j}(V)$ for all $v \in V$ and $i < j \leq c+1$.  Thus there are at most $\aleph_{0}$ many types over any countable set.   
\end{proof}

Recall that $\mathbf{L}_{c,\mathbb{F}}$, defined in Theorem \ref{Lcf has a fraisse limit}, is the Fra\"iss\'e limit of $\mathbb{L}_{c,\mathbb{F}}$, the class of $c$-nilpotent LLAs over $\mathbb{F}$ in the language $\LL_{c,\mathbb{F}}$, which properly contains $\LL_{\mathrm{flag},c, \mathbb{F}}$. 

\begin{lem}
The reduct of $\mathbf{L}_{c,\mathbb{F}}$ to the language $\LL_{\mathrm{flag},c, \mathbb{F}}$ is the Fra\"iss\'e limit of $\mathbb{K}_{\mathrm{flag}, c,\mathbb{F}}$. 
\end{lem}

\begin{proof}
    By \cite[Lemma 2.12]{RamseyCardInv}, we must show that if $A,B \in \mathbb{K}_{\mathrm{flag},c,\mathbb{F}}$ and $\pi: A \to B$ is an $\LL_{\mathrm{flag},c,\mathbb{F}}$-embedding, then, if $C = \langle A \rangle$ is a $c$-nilpotent Lie algebra over $\mathbb{F}$ generated by $A$, then there is some  $D = \langle B \rangle$ generated by $B$ and an LLA embedding $\tilde{\pi} : C \to D$ extending $\pi$.  Write the vector space $B$ as the direct sum of $\pi(A) \oplus E$.  Given $C$, define a vector space $D = F \oplus E$ extending $\pi(A) \oplus E$ so that $F \supseteq \pi(A)$ has the same dimension as $C$ and define $\tilde{\pi} : C \to D$ to be the vector space embedding taking $C$ to $F$ extending $\pi$. Define  a bracket on $F$ by pushing forward the structure from $C$, i.e. defining $[\pi(c),\pi(c')] = \tilde{\pi}([c,c'])$ for all $c,c' \in C$.  Then define $[d,e] = 0$ when $d \in D$ and $e \in E$, i.e. we give $E$ the structure of an abelian Lie algebra and then $D$ is the abelian direct sum of the Lie algebras $F$ and $E$.  It is clear that $D$ is an LLA and $\tilde{\pi} : C \to D$, then, is an embedding of LLAs. 
\end{proof}

\subsubsection{Algebraic closure}

\begin{lem}
    In $T_{c,\mathbb{F}}$ we have $\mathrm{acl}(A) = \mathrm{dcl}(A) = \langle A \rangle$  for all sets $A\subseteq \mathbb{M}_{c,\mathbb{F}}$. 
\end{lem}

\begin{proof}
    Pick $c \in \mathbb{M}_{c,\mathbb{F}} \setminus \langle A \rangle$. We let $B_{0} = \langle c,A \rangle$ and inductively pick $B_{i}$ such that $B_{i} \cong_{A} B_0$ and $\langle B_{\leq i} \rangle \cong B_{i} \otimes_{A} \langle B_{<i} \rangle$ be the amalgam over $A$. For each $i < \omega$, we can pick some isomorphism $\sigma_{i} : B_0 \to B_{i}$ and define $c_{i} = \sigma_{i}(c)$. By quantifier-elimination, we have $c_{i} \equiv_{A} c$ for all $i$, and, since each $B_{i}$ is freely amalgamated with $\langle B_{<i} \rangle$ over $A$, we get $B_{i} \cap B_{j} = \langle A \rangle$ for all $i \neq j$.  Therefore $c_{i} \neq c_{j}$ for all $i \neq j$ which shows $c \not\in \mathrm{acl}(A)$. Since $\langle A \rangle \subseteq \mathrm{dcl}(A) \subseteq \mathrm{acl}(A)$ always holds, the conclusion follows. 
\end{proof}

\subsection{SOP$_{3}$}  In this section, we will prove that the theories $T_{c,\mathbb{F}}$ have SOP$_{3}$, when $c \geq 3$.  Note that this is a marked jump in complexity from the $2$-nilpotent case analyzed in Section \ref{nil-2}, though we will show in the next subsection that $T_{c,\mathbb{F}}$ is NSOP$_{4}$ for all $c$. 
  
\begin{defn}
 Suppose $n \geq 3$.  We say a theory $T$ has SOP$_{n}$ ($n$-strong order property), if there is some type $p(x,y)$ and an indiscernible sequence $(a_{i})_{i < \omega}$ satisfying the following:
 \begin{itemize}
     \item $(a_{i},a_{j}) \models p \iff i < j$. 
     \item $p(x_{0},x_{1}) \cup p(x_{1},x_{2}) \cup \ldots \cup p(x_{n-2},x_{n-1}) \cup p(x_{n-1},x_{0})$ is inconsistent. 
 \end{itemize}
We say that $T$ is NSOP$_n$ if it does not have SOP$_n$.
\end{defn}

\begin{remark}\label{rk:NSOPnNSOPn+1}
    If $T$ is NSOP$_n$ then $T$ is NSOP$_{n+1}$. Indeed if $T$ is NSOP$_n$, then for any indiscernible sequence $(a_i)_{i<\omega}$ there are $c_1,\ldots, c_n$ such that $c_i c_{i+1}\equiv a_0a_1$ for $i< n$ and $c_n c_1\equiv a_0a_1$. Then $c_1c_2\equiv a_0a_1\equiv a_0a_2$ hence there exists $c_{\frac{3}{2}}$ such that $c_1c_{\frac{3}{2}} c_2 \equiv a_0a_1a_2$ and $c_1,c_{\frac{3}{2}}, c_2,\ldots,c_n$ witness NSOP$_{n+1}$.
\end{remark}

The following is \cite[Fact 1.3]{shelah2008more}. Note that the original third condition there of $\{\varphi(x;y), \psi(x;y)\}$ being  contradictory, is superfluous, whence we avoid it below.  

\begin{fact} \cite[Fact 1.3]{shelah2008more}
The theory $T$ has SOP$_{3}$ if and only if there are formulas $\varphi(x;y)$ and $\psi(x;y)$, as well as an indiscernible sequence $(a_{i})_{i < \omega}$ satisfying the following:
\begin{itemize}
    \item For all $k < \omega$ the set $\{\varphi(x;a_{i}) : i \leq k\} \cup \{\psi(x;a_{i}) : i > k\}$ is consistent. 
    \item For all $i < j$ the set $\{\psi(x;a_{i}), \varphi(x;a_{j})\}$ is inconsistent. 
\end{itemize}
\end{fact}

\begin{lem} \label{2-nil construction}
Let $\mathbb{F}$ be a field. Suppose $V$ and $W$ are $\mathbb{F}$-vector spaces and $[\cdot,\cdot]_{0}:V^{2} \to W$ is an alternating bilinear map.  Define a map $[\cdot,\cdot]:(V \oplus W)^{2} \to (V \oplus W)$ by 
$$
[(v,w), (v',w')] = (0,[v,v']_{0})
$$
for all $v,v' \in V$, $w,w' \in W$. Then $[\cdot,\cdot]$ gives $V \oplus W$ the structure of a $2$-nilpotent Lie algebra over $\mathbb{F}$.
\end{lem}

\begin{proof}
Clearly $[\cdot,\cdot]$ is an alternating bilinear map since $[\cdot,\cdot]_{0}$ is.  Moreover, 
$$
[x,[y,z]] = 0
$$
for all $x,y,z \in V \oplus W$ so the Jacobi identity is trivially satisfied and the resulting Lie algebra is $2$-nilpotent.
\end{proof}

\begin{thm}
Assume $c \geq 3$. The theory $T_{c,\mathbb{F}}$ has SOP$_{3}$.   
\end{thm}

\begin{proof}
It suffices to show that $T_{3,\mathbb{F}}$ has SOP$_{3}$, since the $3$-nilpotent LLA $V$ we construct below may also be regarded as a $c$-nilpotent LLA for any $c \geq 3$ with the interpretation $P_{i}(V) = 0$ for all $c > 3$, and hence may be embedded into $\mathbb{M}_{c,\mathbb{F}}$ as well. 

Let $V$ be an $\mathbb{F}$-vector space with basis $X = \{a_{i}, a'_{i}, b_{i}, b'_{i}, d_{i,j} : i < j < \omega\}$. Define an alternating bilinear map $[\cdot,\cdot]:V^{2} \to V$ by $[a'_{i},b_{j}] =  d_{i,j}$ for $i < j$ and $[x,y] = 0$ for all $x,y \in X$ such that $\{x,y\} \neq \{a'_{i},b_{j}\}$ for all $i < j$. If we define $V_{0} = \langle a_{i},a'_{i},b_{i},b'_{i} : i < \omega \rangle$ and $V_{1} = \langle d_{i,j} : i < j < \omega\rangle$, we have $V = V_{0} \oplus V_{1}$ and we may view the map $[\cdot,\cdot] : V^{2} \to V$ as a map induced from the alternating bilinear map $[\cdot,\cdot]_{0}:V_{0}^{2} \to V_{1}$ by $[a'_{i},b_{j}]_{0} =  d_{i,j}$ for $i < j$ and $[x,y]_{0} = 0$ for $\{x,y\} \neq \{a'_{i},b_{j}\}$ for all $i < j$, by Lemma \ref{2-nil construction}.  Therefore $(V,[\cdot,\cdot])$ is a $2$-nilpotent Lie algebra.  We give $V$ a flag structure by interpreting $P_{1},\ldots, P_{4}$ by 
$$
P_{1}(V) = V \supseteq P_{2}(V) = \langle a'_{i},b'_{i},d_{i,j} : i < j < \omega \rangle \supseteq P_{3}(V) = \langle d_{i,j} : i < j < \omega \rangle,
$$
and $P_{4}(V) = 0$. It is easy to check $[P_{i}(V),P_{j}(V)] \subseteq P_{i+j}(V)$ for all $i,j$. Thus, we may regard $V$ as a substructure of $\mathbb{M}_{3,\mathbb{F}}$.  Let $c_{i} = (a_{i},a'_{i},b_{i},b'_{i})$ for all $i < \omega$. It is immediate from the quantifier elimination that $I = (c_{i})_{i < \omega}$ is an indiscernible sequence.


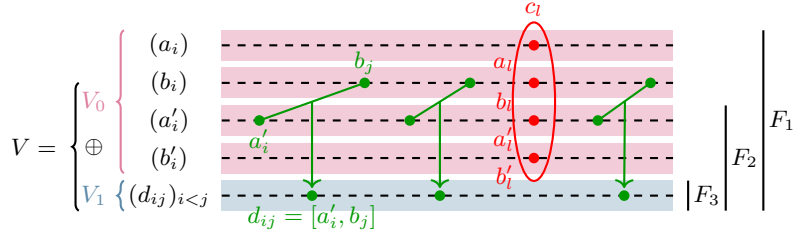
\begin{tikzpicture}
    \fill[airforceblue!30] (3,0.2) rectangle (9,-0.2);
    \draw[dashed, thick] (3,0) -- (9,0); 
    \node at (2.3,0) {$(d_{ij})_{i<j}$};
    \fill[purple!20] (3,0.3) rectangle (9,0.7); 
    \draw[dashed, thick] (3,0.5) -- (9,0.5);
    \node at (2.3,0.5) {$(b'_i)$};
    \coordinate[label=left:\textcolor{red}{$b'_l$}] (B) at (7,0.25);
    \fill[red] (7.15,0.5) circle (2pt);
    \fill[purple!20] (3,0.8) rectangle (9,1.2);
    \draw[dashed, thick] (3,1) -- (9,1);
    \node at (2.3,1) {$(a'_i)$};
    \coordinate[label=left:\textcolor{red}{$a'_l$}] (B) at (7,0.75);
    \fill[red] (7.15,1) circle (2pt);
    \fill[purple!20] (3,1.3) rectangle (9,1.7);
    \draw[dashed, thick] (3,1.5) -- (9,1.5);
    \node at (2.3,1.5) {$(b_i)$};
    \coordinate[label=left:\textcolor{red}{$b_l$}] (B) at (7,1.25);
    \fill[red] (7.15,1.5) circle (2pt);
    \fill[purple!20] (3,1.8) rectangle (9,2.2);
    \draw[dashed, thick] (3,2) -- (9,2);
    \node at (2.3,2) {$(a_i)$};
    \coordinate[label=left:\textcolor{red}{$a_l$}] (B) at (7,1.75);
    \fill[red] (7.15,2) circle (2pt);
    \draw[red,thick] (7.15,1.25) ellipse (8pt and 30pt);
    \node at (7.15,2.5) {\textcolor{red}{$\bar{c_l}$}}; 
    \draw [decorate, decoration = {brace}, thick, purple!50] (1.7,0.3) -- (1.7,2.2);
    \draw [decorate, decoration = {brace}, thick, airforceblue] (1.7,-0.2) -- (1.7,0.2);
    \node at (1.3, 1.25) {\textcolor{purple!50}{$V_0$}};
    \node at (1.3, 0.65) {$\oplus$};
    \node at (1.3, 0) {\textcolor{airforceblue}{$V_1$}};
    \draw [decorate, decoration = {brace}, thick] (1.1,-0.2) -- (1.1,1.5);
    \node at (0.5, 0.65) {$V = $};
    \draw[decorate, thick] (9.2, -0.2) -- (9.2, 0.2);
    \node at (9.45, 0) {$F_3$};
    \draw[decorate, thick] (9.7, -0.2) -- (9.7, 1.2);
    \node at (9.95, 0.5) {$F_2$};
    \draw[decorate, thick] 
    (10.2, -0.2) -- (10.2, 2.2);
    \node at (10.45, 1) {$F_1$};
    \coordinate (a) at (5.5,1);
    \fill[green!60!black] (5.5,1) circle (2pt);
    \coordinate (a) at (6.3,1.5);
    \fill[green!60!black] (6.3,1.5) circle (2pt);
    \coordinate[label=below:\textcolor{green!60!black}{$a'_i$}] (a) at (3.5,1);
    \fill[green!60!black] (3.5,1) circle (2pt);
    \coordinate[label=above:\textcolor{green!60!black}{$b_j$}] (a) at (4.9,1.5);
    \fill[green!60!black] (4.9,1.5) circle (2pt);
    \coordinate (a) at (8,1);
    \fill[green!60!black] (8,1) circle (2pt);
    \coordinate (a) at (8.7,1.5);
    \fill[green!60!black] (8.7,1.5) circle (2pt);
    \draw[green!60!black, thick] (3.5,1) -- (4.9,1.5);
    \draw[green!60!black, thick] (5.5,1) -- (6.3,1.5);
    \draw[green!60!black, thick] (8,1) -- (8.7,1.5);
    \draw[->, green!60!black, thick] (4.2, 1.25) -- (4.2, 0.1);
    \coordinate[label=below:\textcolor{green!60!black}{$d_{ij}=[a'_i,b_j]$}] (a) at (4.2,0);
    \fill[green!60!black] (4.2,0) circle (2pt);
    \draw[->, green!60!black, thick] (5.9, 1.25) -- (5.9, 0.1);
    \fill[green!60!black] (5.9,0) circle (2pt);
    \draw[->, green!60!black, thick] (8.36, 1.25) -- (8.35, 0.1);
    \fill[green!60!black] (8.35,0) circle (2pt);
    \end{tikzpicture}

We will define two formulas $\varphi(x;y,y',z,z') = [x,z] = z'$ and $\psi(x,y,y',z,z') = [x,y] = y'$. We will show that $\varphi$, $\psi$, and $I$ witness SOP$_{3}$. 

\textbf{Claim 1}:  For all $k < \omega$, 
$$
\{\varphi(x;c_{i}) : i \leq k\} \cup \{\psi(x;c_{j}) : j > k\}
$$
is consistent. 

\emph{Proof of Claim}:  Let $W = (W,[\cdot,\cdot],W_{1},\ldots, W_{4})$ be the substructure generated by
$$
\{b_{i},b'_{i} : i \leq k\} \cup \{a_{j},a'_{j} : j > k\}.
$$
Note that, by the construction of $V$, $W$ is just the span of these vectors together with a trivial Lie bracket (i.e. $[x,y] = 0$ for all $x,y \in W$). 

Let $W_{*}$ be an $\mathbb{F}$-vector space spanned by $\{b_{i} : i \leq k\} \cup \{a_{j} : j > k\} \cup \{c_{*}\}$ where $c_{*}$ is a new basis element. Let $W_{**}$ be the $\mathbb{F}$-vector space spanned by $\{b_{i}': i \leq k\} \cup \{a'_{j} : j > k\}$. Define an alternating bilinear map $[\cdot,\cdot]_{*} : W_{*} \to W_{**}$ by 
$$
\begin{matrix}
[c_{*},b_{i}]_{*} &=& b'_{i} & \text{ for }i \leq k \\
[c_{*},a_{i}]_{*} &=& a'_{i} & \text{ for } i > k \\
[x,y]_{*} &=& 0 & \text{ for } x,y \in \{b_{i} : i \leq k\} \cup \{a_{i} : i > k\}.
\end{matrix}
$$
Form $\tilde{W} = W_{*} \oplus W_{**}$ and let $\tilde{[\cdot,\cdot]}$ be the Lie algebra induced by $[\cdot,\cdot]_{*}$, via application of Lemma \ref{2-nil construction}.  Note that $\tilde{W}$ may be naturally viewed as an extension of $W$, with the flag defined by $P_{1}(\tilde{W}) = \tilde{W}$ and $P_{i}(\tilde{W}) = P_{i}(W)$ for $i = 2,3,4$. By quantifier elimination, we may embed $\tilde{W}$ into $\mathbb{M}_{3,\mathbb{F}}$, so we may likewise assume $\tilde{W}$ is a substructure of $\mathbb{M}_{3,\mathbb{F}}$. \\

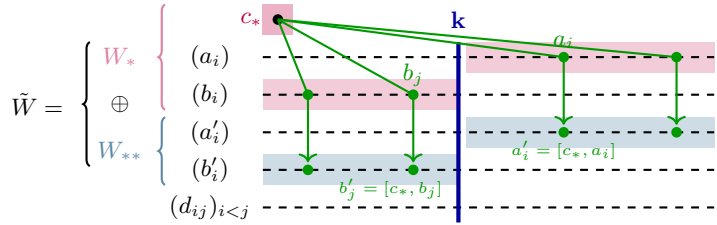
\begin{tikzpicture}
    \draw[dashed, thick] (3,0) -- (9,0); 
    \node at (2.3,0) {$(d_{ij})_{i<j}$};
    \fill[airforceblue!30] (3,0.3) rectangle (5.6,0.7); 
    \draw[dashed, thick] (3,0.5) -- (9,0.5);
   \node at (2.3,0.5) {$(b'_i)$};
    \fill[airforceblue!30] (5.7,0.8) rectangle (9,1.2);
    \draw[dashed, thick] (3,1) -- (9,1);
    \node at (2.3,1) {$(a'_i)$};
    \fill[purple!20] (3,1.3) rectangle (5.6,1.7);
    \draw[dashed, thick] (3,1.5) -- (9,1.5);
    \node at (2.3,1.5) {$(b_i)$};
    \fill[purple!20] (5.7,1.8) rectangle (9,2.2);
    \draw[dashed, thick] (3,2) -- (9,2);
    \node at (2.3,2) {$(a_i)$};
    \fill[purple!30] (3,2.7) rectangle (3.4,2.3);
    \coordinate[label=left:\textcolor{purple}{$c_*\;$}] (B) at (3.2,2.5);
    \fill[black] (3.2,2.5) circle (2pt);
    \draw[ultra thick, blue!60!black] (5.6, 2.2) -- (5.6, -0.2);
    \node at (5.6, 2.5) {\textcolor{blue!60!black}{$\mathbf{k}$}};
    \draw [decorate, decoration = {brace}, thick, purple!50] (1.7,1.3) -- (1.7,2.7);
    \draw [decorate, decoration = {brace}, thick, airforceblue] (1.7,0.3) -- (1.7,1.2);
    \node at (1.1, 2) {\textcolor{purple!50}{$W_*$}};
    \node at (1.1, 1.375) {$\oplus$};
    \node at (1.1, 0.75) {\textcolor{airforceblue}{$W_{**}$}};
    \draw [decorate, decoration = {brace}, thick] (0.7,0.55) -- (0.7,2.2);
    \node at (0, 1.375) {$\Tilde{W} = $};
    \coordinate (a) at (3.6,1.5);
    \fill[green!60!black] (3.6,1.5) circle (2pt);
    \coordinate[label=above:\textcolor{green!60!black}{$b_j$}] (a) at (5,1.5);
    \fill[green!60!black] (5,1.5) circle (2pt);
    \coordinate[label=above:\textcolor{green!60!black}{$a_i$}] (a) at (7,2);
    \fill[green!60!black] (7,2) circle (2pt);
    \coordinate (a) at (8.5,2);
    \fill[green!60!black] (8.5,2) circle (2pt);
    \draw[green!60!black, thick] (3.2,2.5) -- (3.6,1.5);
    \draw[green!60!black, thick] (3.2,2.5) -- (5,1.5);
    \draw[green!60!black, thick] (3.2,2.5) -- (7,2);
    \draw[green!60!black, thick] (3.2,2.5) -- (8.5,2);
    \draw[->, green!60!black, thick] (5, 1.5) -- (5, 0.6);
    \coordinate[label=below:\begin{tiny}\textcolor{green!60!black}{$b'_j=[c_*,b_j]$}\end{tiny}] (a) at (4.7,0.5);
    \fill[green!60!black] (5,0.5) circle (2pt);
    \draw[->, green!60!black, thick] (3.6, 1.5) -- (3.6, 0.6);
    \fill[green!60!black] (3.6,0.5) circle (2pt);
    \draw[->, green!60!black, thick] (7, 2) -- (7, 1.1);
    \coordinate[label=below:\begin{tiny}\textcolor{green!60!black}{$a'_i=[c_*,a_i]$}\end{tiny}] (a) at (7,1);
    \fill[green!60!black] (7,1) circle (2pt);
    \draw[->, green!60!black, thick] (8.5, 2) -- (8.5, 1.1);
    \fill[green!60!black] (8.5,1) circle (2pt);
    \end{tikzpicture}

By construction, we have 
\begin{eqnarray*}
c_* &\models& \left(\bigwedge_{j\leq k} [x,b_j]=b'_j\right)\wedge\left(\bigwedge_{i>k} [x,a_i]=a'_i\right)\text{ and thus }\\
c_{*} &\models& \{\varphi(x;c_{i}) : i \leq k\} \cup \{\psi(x;c_{j}) : j > k\},
\end{eqnarray*}
so this set of formulas is consistent. \qed

Now to establish SOP$_{3}$, we have to prove the following:

\textbf{Claim 2}:  If $i < j$, 
$$
\{\psi(x;c_{i}) , \varphi(x;c_{j})\}
$$
is inconsistent.

\emph{Proof of Claim}:  Suppose towards contradiction that there is some $d$ realizing these two formulas, i.e. some $d$ with $[d,a_{i}] = a'_{i}$ and $[d,b_{j}] = b'_{j}$. Then we recall
$$
d_{i,j} = [a_{i}',b_{j}] = [[d,a_{i}],b_{j}]. 
$$
Applying the Jacobi identity to the expression on the right yields 
\begin{eqnarray*}
d_{i,j} &=& [[d,a_{i}],b_{j}] \\
&=& - [b_{j},[d,a_{i}]] \\
&=& - ([[b_{j},d],a_{i}] + [d,[b_{j}, a_{i}]])\\
&=& -([-b'_{j},a_{i}]+ 0)\\
&=& 0
\end{eqnarray*}
where the  last two lines follow from $[a_{i},b_{j}] = [a_{i},b'_{j}] = 0$. Since $d_{i,j} \neq 0$, this yields a contradiction, so we conclude this pair of formulas is inconsistent. \qed\\


\noindent\begin{minipage} {0.6\textwidth} 
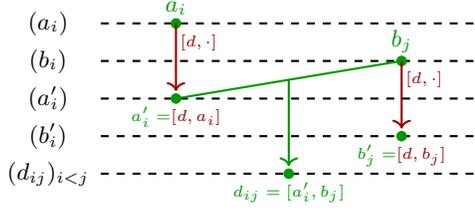
    \begin{tikzpicture}
    \draw[dashed, thick] (3,0) -- (8,0); 
    \node at (2.3,0) {$(d_{ij})_{i<j}$};
    \draw[dashed, thick] (3,0.5) -- (8,0.5);
   \node at (2.3,0.5) {$(b'_i)$};
    \draw[dashed, thick] (3,1) -- (8,1);
    \node at (2.3,1) {$(a'_i)$};
    \draw[dashed, thick] (3,1.5) -- (8,1.5);
    \node at (2.3,1.5) {$(b_i)$};
    \draw[dashed, thick] (3,2) -- (8,2);
    \node at (2.3,2) {$(a_i)$};
    \coordinate[label=above:\textcolor{green!60!black}{$b_j$}] (a) at (7,1.5);
    \fill[green!60!black] (7,1.5) circle (2pt);
    \coordinate[label=above:\textcolor{green!60!black}{$a_i$}] (a) at (4,2);
    \fill[green!60!black] (4,2) circle (2pt);
    \draw[->, red!60!black, thick] (4, 2) -- (4, 1.1);
    \node at (4.3,1.75) {\begin{tiny}\textcolor{red!60!black}{$[d,\cdot]$}\end{tiny}};
    \coordinate[label=below:\begin{tiny}\textcolor{green!60!black}{$b'_j=$\textcolor{red!60!black}{$[d,b_j]$}}\end{tiny}] (a) at (7,0.5);
    \fill[green!60!black] (7,0.5) circle (2pt);
    \coordinate[label=below:\begin{tiny}\textcolor{green!60!black}{$a'_i=$\textcolor{red!60!black}{$[d,a_i]$}}\end{tiny}] (a) at (4,1);
    \fill[green!60!black] (4,1) circle (2pt);
    \draw[->, red!60!black, thick] (7, 1.5) -- (7, 0.6);
    \node at (7.3,1.25) {\begin{tiny}\textcolor{red!60!black}{$[d,\cdot]$}\end{tiny}};
    \draw[green!60!black, thick] (4,1) -- (7,1.5);
    \draw[->, green!60!black, thick] (5.5, 1.25) -- (5.5, 0.1);
    \coordinate[label=below:\begin{tiny}\textcolor{green!60!black}{$d_{ij}=[a'_i,b_j]$}\end{tiny}] (a) at (5.5,0);
    \fill[green!60!black] (5.5,0) circle (2pt);
    \end{tikzpicture}
    \end{minipage}
    \hfill
    \begin{minipage}{0.5\textwidth}\begin{small}
    If $\textcolor{red}{d}\models\{\psi(x,c_i),\varphi(x,c_j)\}$ for $i<j$,
    then
    \begin{eqnarray*}
    d_{ij}=[a'_i,b_j]&=&[[\textcolor{red}{d},a_i],b_j]\\
    &=&[[\textcolor{red}{d},b_j],a_i]+[\textcolor{red}{d},[a_i,b_j]]\\
    &=&[b'_j,a_i]+[\textcolor{red}{d},0]\\
    &=&0+0=0.
    \end{eqnarray*} \end{small}
    \end{minipage}\\

We have proved that $\varphi$, $\psi$, and $I$ witness the two-formula version of SOP$_{3}$, completing the proof. 
\end{proof}

\subsection{NSOP$_{4}$}\label{subsec:thetheoryisNSOP4}


In this section, we argue that the theory $T_{c,\mathbb{F}}$ is NSOP$_{4}$.  To begin, we will establish two general model-theoretic lemmas.  Their statements concern an arbitrary complete theory $T$.  Recall that if $M \models T$, then coheir independence $a \ind^{u}_{M} b$ means $\mathrm{tp}(a/Mb)$ is finitely satisfiable in $M$ and heir independence $a \ind^{h}_{M} b$ means $b \ind^{u}_{M} a$.  The following is certainly well-known but, for lack of a precise reference, we give a proof.

\begin{lem} \label{heir sequence lemma}
Suppose $I = (a_{i})_{i < \omega}$ is an indiscernible sequence.  Then there is a model $M$ such that $a_{i} \indi{h}_{M} a_{<i}$ for all $i < \omega$. 
\end{lem}

\begin{proof}
 Expand the monster model to have Skolem functions and denote the resulting expansion $\mathbb{M}^{\mathrm{Sk}}$.  By Ramsey and compactness, we can extract an $L^{\mathrm{Sk}}$-indiscernible sequence $I'$ from $I$ and then stretch $I'$ so that $I' = (a'_{i})_{i < \omega + \omega}$.  Let $M = \mathrm{dcl}_{L^{\mathrm{Sk}}}(a'_{<\omega})$.  Now we claim $a'_{\omega+i} \indi{h}_{M} a'_{\omega}\ldots a'_{\omega + i-1}$ for all $i < \omega$.  Any formula in $\text{tp}_{L}(a'_{\omega}\ldots a'_{\omega + i-1}/Ma'_{\omega + i})$ can be written as $\varphi(x_{0}, \ldots, x_{i-1};a'_{\omega+i}, t(a'_{<N}))$ for an $L$-formula $\varphi$, a natural number $N$, and some Skolem term $t$. By $L^{\mathrm{Sk}}$-indiscernibility, we have 
 $$
 \mathbb{M}^{\mathrm{Sk}} \models \varphi(a'_{N},\ldots, a'_{N+i-1};a'_{\omega+i}, t(a'_{<N})),
 $$
 Working now in $\mathbb{M}$, we have shown $\text{tp}_{L}(a'_{\omega}\ldots a'_{\omega + i-1}/Ma'_{\omega + i})$ is finitely satisfiable in $M$, or, in other words, $a'_{\omega+i} \indi{h}_{M} a'_{\omega}\ldots a'_{\omega + i-1}$. Choose an automorphism $\sigma \in \mathrm{Aut}(\mathbb{M})$, with $\sigma(a'_{\omega + i}) = a_{i}$ for all $i$. Then $M_{*} = \sigma(M)$ is a model such that $a_{i} \indi{h}_{M_{*}} a_{<i}$ for all $i$, as desired. 
\end{proof}

The following result is a variant of the results of \cite{conant2017axiomatic} and \cite{mutchnik2022conant} yielding NSOP$_4$ via the existence of an independence relation with certain properties. See also \cite{johnson2023curve} for a similar approach.

\begin{theorem}\label{thm:criterionNSOP4}
Let $\ind$ be an invariant relation satisfying symmetry, full existence, stationarity over models, and the following ``weak transitivity" over models:
    \[a\ind_{Md} b \text{ and } a\indi h _M d \text{ and } b\indi u _M d  \implies a\ind_M b\]
    for all finite tuples $a,b,d$ and small model $M$. Then $T$ is NSOP$_4$.
\end{theorem}
\begin{proof}
    Let $(a_i)_{i<\omega}$ be an indiscernible sequence and $p(x,y) = \tp(a_0,a_1)$. We show that 
    \[p(x_0,x_1)\cup p(x_1,x_2)\cup p(x_2,x_3)\cup p(x_3,x_0)\]
    is a consistent partial type. By Lemma \ref{heir sequence lemma}, $a_i \indi h _M a_{<i}$ for all $i<\omega$ and some small model $M$.

    By full existence, there exists $a_0^*\equiv_{Ma_1} a_0$ such that $a_0^*\ind_{Ma_1} a_2$. By symmetry, we have $a_2\ind_{Ma_1} a_0^*$. As $a_2\indi h _M a_1$ and $a_0^*\indi u _M a_1$, we conclude $a_2\ind_{M} a_0^*$ using the weak transitivity assumption.

    We have $a_0^*\equiv_M a_2$ and $a_0^*\ind_M a_2$. Let $a$ be such that $a_0^*a_2\equiv_M a_2a$. Then by invariance, we have $a_2\ind_M a$ and by symmetry we obtain $a \ind_M a_2$. Now stationarity yields $a\equiv_{Ma_2}a_0^*$, whence $aa_2\equiv_M a_0^*a_2$, and thus $a_0^*a_2\equiv_M a_2a_0^*$. 
    
    Then, there exists $a_3^*$ such that $a_0^*a_2a_1\equiv_M a_2a_0^*a_3^*$. We claim that $(a_0^*,a_1,a_2,a_3^*)$ satisfies the type above.
    First, $a_0^*a_1\equiv a_0a_1$ hence $p(a_0^*,a_1)$. By indiscernibility, $a_0a_1\equiv_M a_1a_2$ hence $p(a_1,a_2)$. By choice, $a_2a_3^*\equiv_M a_0^*a_1$ hence $p(a_2,a_3^*)$. Finally $a_3^*a_0^*\equiv_M a_1a_2$ hence $p(a_3^*,a_0^*)$.
\end{proof}

We will use Theorem \ref{thm:criterionNSOP4} to prove that the theory $T_{c,\mathbb{F}}$ is NSOP$_4$.

\begin{lemma}\label{lm:heircoheirnicebases}
    Let $A,B,C$ be LLAs with $C\seq A\cap B$. If $A\indi u_C B$ or $A\indi h _C B$ then for every $a_1,\ldots, a_n\in A$ if
    \[B\lteq \vect{Ba_1} \lteq \vect{Ba_1 a_2}\lteq \ldots \lteq \vect{Ba_1\ldots a_n}\]
    then 
    \[C\lteq \vect{Ca_1} \lteq \vect{Ca_1 a_2}\lteq \ldots \lteq \vect{Ca_1\ldots a_n}.\]
\end{lemma}

\begin{proof}
    Let $a_1,\ldots, a_n$ be as in the hypothesis. By Lemma \ref{lm:idealspan}, we have $\vect{Ba_1,\ldots, a_m} = \Span_\F(Ba_1\ldots a_m)$ for all $m\leq n$. Fix $m\leq n$ and assume by induction that $\vect{C a_1,\ldots, a_{i-1}}$ is an ideal of $\vect{C a_1,\ldots, a_{i}}$ for all $i\leq m$. Again by Lemma \ref{lm:idealspan} we have $\vect{C a_1,\ldots, a_{m}} = \Span_\F(C a_1\ldots a_m)$. In order to prove that $\vect{Ca_1\ldots a_m}$ is an ideal of $\vect{Ca_1\ldots a_{m+1}}$, it is enough to prove that $[a_{m+1},v]\in \vect{Ca_1\ldots a_m}$ for all $v\in \vect{Ca_1\ldots a_m}$. Let $v = c+\sum_{i=1}^m \lambda_i a_i\in \vect{Ca_1\ldots a_m}$. As $\vect{Ba_1\ldots a_m}$ is an ideal of $\vect{Ba_1\ldots a_{m+1}}$, $[a_{m+1},v]\in \vect{Ba_1\ldots a_{m}}$ hence there exists $b\in B$ and $\mu_1,\ldots,\mu_{m}\in\mathbb{F}$ such that $[a_{m+1},v] = b+\sum_{i = 1}^m \mu_i a_i$. It follows that the formula $\phi(x_1,\ldots,x_{m+1}, c,b)$ defined by
    \[[x_{m+1},c+\sum_{i = 1}^m \lambda_i x_i] = b+\sum_{i = 1}^m \mu_i x_i\]
    is in $p(x_1,\ldots,x_{m+1}) = \tp(a_1,\ldots,a_{m+1}/B)$.

Assume that $A\indi u _C B$. Then $\phi(x_1,\ldots,x_{m+1},c,b)$ is satisfiable by a tuple $(c_1,\ldots c_{m+1})$ from $C$ and it follows that $b\in C$. Then $[a_{m+1},v] = b+\sum_{i = 1}^m \mu_i a_i\in \vect{Ca_1\ldots a_m}$ and we conclude.

Assume that $A\indi h _C B$. In that case, there exists $c'\in C$ such that $\models \phi(a_1,\ldots, a_{m+1},c,c')$ hence $[a_{m+1},v] = c'+\sum_{i = 1}^m \mu_i a_i\in \vect{Ca_1\ldots a_m}$ and we conclude.
\end{proof}

Recall Definition \ref{def:Malcev} of a tuple being Malcev.

\begin{corollary}\label{cor:malcevbasesheircoheir}
    Let $A,B,C$ be LLAs over $\mathbb{F}$ with $C\seq A\cap B$. If $A\indi u_C B$ or $A\indi h _C B$ then for every $a = (a_1,\ldots, a_n)$ from $A$ if $a$ is (ordered) Malcev over $B$, then it is also (ordered) Malcev over $C$.
\end{corollary}
\begin{proof}
    Let $a = (a_1,\ldots,a_n)$ be in $A$ and Malcev over $B$. We prove that $a$ is Malcev over $C$, that is $\Span_{\mathbb{F}}(CP_i(\vect{Ca})) = \Span_{\mathbb{F}}(CP_i(a))$. First, using Lemma \ref{lm:heircoheirnicebases} we have $\vect{Ca_1,\ldots,a_n} = \Span_{\mathbb{F}}(Ca_1\ldots a_n)$. If $x\in \Span_{\mathbb{F}}(CP_i(\vect{Ca}))$ then there exist $\lambda_1,\ldots,\lambda_n\in \F$ such that $x = c+\sum_j\lambda_j a_j$. By assumption, $\Span_{\mathbb{F}}(BP_i(\vect{aB})) = \Span_{\mathbb{F}}(BP_i(a))$, whence there exist $\mu_1,\ldots \mu_k$ such that for $a_{i_1},\ldots,a_{i_k}\in P_i(a)$ we have $x = b+\sum_{\ell} \mu_{\ell} a_{i_{\ell}}$. It follows that $\sum_{\ell} \mu_{\ell} a_{i_{\ell}} - \sum_j\lambda_j a_j\in B$ which implies that $\lambda_j = \mu_{i_{\ell}}$ if $i = i_{\ell}$ and $\lambda_j = 0$ if $j\notin \set{i_1,\ldots,i_k}$ so we conclude that $x\in \Span_{\mathbb{F}}(CP_i(a))$. 
\end{proof}

\begin{thm} \label{NSOP4 theorem}
The theory $T_{c,\mathbb{F}}$ is NSOP$_4$.
\end{thm}
\begin{proof}
By Corollary \ref{cor:freerelationisanSIR} the relation $\indi \otimes$ is a stationary independence relation, hence by Theorem \ref{thm:criterionNSOP4} it suffices to prove the ``weak transitivity" property. Let $a,b,d$ be finite tuples and $E$ be an LLA such that for $C = \vect{Ed}$ we have $a\indi \otimes_C b$, $a\indi h_E d$ and $b\indi u _E d$. We may assume that $a$ and $b$ are Malcev bases of $\vect{Ca}$ and $\vect{Cb}$ over $C$. As $a\indi h _E C$ and $b\indi u _E C$, by Corollary \ref{cor:malcevbasesheircoheir} we have that $a$ and $b$ are Malcev over $E$ hence $a\indi \otimes_E b$ by Theorem \ref{thm:weaktransitivitygeneralcase}.
\end{proof}

Restricting to $\mathbb{F} = \mathbb{F}_{p}$ and $c < p$, we obtain a corresponding result for groups, via Lazard correspondence.

\begin{cor} \label{nsop4cor}
For all $c < p$, the theory $\mathrm{Th}(\mathbf{G}_{c,p})$ is NSOP$_{4}$.  
\end{cor}

\begin{proof}
    This follows from Theorem \ref{NSOP4 theorem}, since $T_{c,p} = \mathrm{Th}(\mathbf{L}_{c,p})$ and Fact  \ref{Lazard facts} which states that $\mathbf{L}_{c,p}$ and $\mathbf{G}_{c,p}$ are bi-interpretable.  
\end{proof}

\subsection{$c\,$-dependence} 

In this subsection, we will show that $T_{c,\mathbb{F}}$ is $c$-dependent and $(c-1)$-independent (see Definition \ref{IPk def}). Via the Lazard correspondence, it will follow as a corollary that the theories $\mathrm{Th}(\mathbf{G}_{c,p})$ for $c < p$ furnish examples of groups showing the strictness of the NIP$_{k}$ hierarchy.   

\begin{lem} \label{term description}
 Every term $t(\overline{x})$ of $\LL_{c,\mathbb{F}}$ is equal modulo $T_{c,\mathbb{F}}$ to an $\mathbb{F}$-linear combination of Lie monomials.  
\end{lem}

\begin{proof}
    This is an easy induction on terms. Clearly the constant $0$ and the variables are of the required form.  Suppose it has been established for terms $s(\overline{x})$ and $t(\overline{x})$.  Then $s(\overline{x}) + t(\overline{x})$ is of the required form.  If we have 
    \begin{eqnarray*}
        s(\overline{x}) &=& \sum_{i <k} \alpha_{i} m_{i}(\overline{x}) \\
        t(\overline{x}) &=& \sum_{j <l} \beta_{j} m'_{j}(\overline{x}),
    \end{eqnarray*}
    for scalars $\alpha_{i}, \beta_{j} \in \mathbb{F}$ and Lie monomials $m_{i}(\overline{x}), m'_{j}(\overline{x})$ for all $i < k$, $j < l$, then, by bilinearity, we have 
    \begin{eqnarray*}
        [s(\overline{x}),t(\overline{x})] &=& \left[ \sum_{i <k} \alpha_{i} m_{i}(\overline{x}), \sum_{j <l} \beta_{j} m'_{j}(\overline{x})\right] \\
        &=& \sum_{i< k,j < l} \alpha_{i}\beta_{j} [m_{i}(\overline{x}),m'_{j}(\overline{x})],
    \end{eqnarray*}
    and each $[m_{i}(\overline{x}), m'_{j}(\overline{x})]$ is a Lie monomial. This completes the induction. 
\end{proof}

\begin{cor} \label{formula description}
    Every formula $\varphi(\overline{x})$ of $\LL_{c,\mathbb{F}}$ is equivalent modulo $T_{c,\mathbb{F}}$ to one of the form $\psi(m_{0}(\overline{x}), \ldots, m_{k-1}(\overline{x}))$ where $\psi(y_{0},\ldots, y_{k-1})$ is a quantifier-free $\LL_{\mathrm{flag},c,\mathbb{F}}$-formula and each $m_{i}(\overline{x})$ is a Lie monomial.
\end{cor}

\begin{proof}
    Immediate by quantifier elimination, Corollary \ref{cor: limit existence}, and Lemma \ref{term description}. 
\end{proof}

The following theorem of Chernikov and Hempel gives a key criterion for establishing that a theory is NIP$_{k}$.  It is the $k$-ary analogue of the earlier cited Fact \ref{2dep composition lemma}.  The proof of this theorem has not yet been disseminated, so we note that our Theorem \ref{cdep theorem} below is conditional on results which are not yet publicly available. 

\begin{fact} \label{composition lemma} \cite{ChernikovHempel3}
Let $M$ be an $\LL'$-structure such that its reduct to a language $\LL \subseteq \LL'$ is NIP.  Let $d,k$ be natural numbers and $\varphi(x_{1}, \ldots, x_{d})$ be an $\LL$-formula. Let further $y_{0}, \ldots, y_{k}$ be arbitrary $(k+1)$-tuples of variables.  For each $1 \leq t \leq d$, let 
$0 \leq i_{t,1}, \ldots, i_{t,k} \leq k$ be arbitrary and let $f_{t} : M_{y_{i_{t,1}}} \times \ldots \times M_{y_{i_{t,k}}} \to M_{x_{t}}$ be an arbitrary $k$-ary function.  Then the formula 
$$
\psi(y_{0}; y_{1}, \ldots, y_{k}) = \varphi(f_{1}(y_{i_{1,1}}, \ldots, y_{i_{1,k}}), \ldots, f_{d}(y_{i_{d,1}}, \ldots, y_{i_{d,k}}))
$$
is $k$-dependent. 
\end{fact}

\begin{thm} \label{cdep theorem}
 The theory $T_{c,\mathbb{F}}$ is $c$-dependent and $(c-1)$-independent. 
\end{thm}

\begin{proof}
 By Lemma \ref{flags are stable}, we know that $T_{\mathrm{flag},c,\mathbb{F}}$ is stable.  Moreover, in a $c$-nilpotent Lie algebra, each Lie monomial $m(\overline{x})$ is at most $c$-ary.  Therefore, the $c$-dependence follows by Fact \ref{composition lemma}. 

 Now we argue that this theory is $(c-1)$-independent.  Let $L$ denote the free $c$-nilpotent Lie algebra over $\mathbb{F}$ with generators $(b_{X})_{X \subseteq \omega^{c-1}}$ and $(a_{i,j})_{i < c-1, j < \omega}$. Let $\prec$ denote an arbitrary linear order of the monomials in these generators such that $a_{i,j} \prec a_{i',j'}$ if $(i,j) <_{lex} (i',j')$ and such that $a_{i,j} \prec b_{X}$ for all $(i,j)$ and $X$.  Then the monomials $[b_{X},a_{0,j_{0}},a_{1,j_{1}}, \ldots, a_{c-2,j_{c-2}}]$ are in the Hall basis (defined with respect to the ordering $\prec$) for all $X \subseteq \omega^{c-1}$ and $(j_{0}, \ldots, j_{c-2}) \in \omega^{c-1}$. Define $I \subseteq L$ to be the vector space spanned by 
 $$
 \left\{[b_{X},a_{0,j_{0}},a_{1,j_{1}}, \ldots, a_{c-2,j_{c-2}}] : X \subseteq \omega^{c-1}, (j_{0}, \ldots, j_{c-2}) \in X\right\}.  
 $$
 Since $L$ is $c$-nilpotent, if $x \in L$ and $y \in I$, then $[x,y] = 0$ and therefore $I$ is an ideal.  Moreover, since terms in the Hall basis are linearly independent, we know 
 $$[b_{X},a_{0,j_{0}},a_{1,j_{1}}, \ldots, a_{c-2,j_{c-2}}] \not\in I$$ when $(j_{0}, \ldots, j_{c-2}) \not\in X$.  

 It follows, then, that in the $c$-nilpotent Lie algebra $\overline{L} = L/I$,
 $$
 [b_{X},a_{0,j_{0}}, \ldots, a_{c-2,j_{c-2}}]  = 0
 $$
 if and only if $(j_{0}, \ldots, j_{c-1}) \in X$.  Since $\mathrm{Age}(\overline{L}) \subseteq \mathrm{Age}(\mathbb{M}_{c,\mathbb{F}})$, we may assume $\overline{L}$ is embedded in $\mathbb{M}$.  This shows that $T$ has the $(c-1)$-independence property. 
\end{proof}

\begin{cor} \label{groupc-dependence}
    Assume $p$ is an odd prime and $c > p$. Then the theory $\mathrm{Th}(\mathbf{G}_{c,p})$ is $c$-dependent and has the $(c-1)$-independence property. In particular, $T_{\mathbf{G}}$, defined in Section \ref{nil-2}, is $2$-dependent.
\end{cor}

\begin{proof}
    As in Corollary \ref{nsop4cor}, this follows from Theorem \ref{NSOP4 theorem}, since $T_{c,p} = \mathrm{Th}(\mathbf{L}_{c,p})$ and $\mathbf{L}_{c,p}$ and $\mathbf{G}_{c,p}$ are bi-interpretable by Fact \ref{Lazard facts}. Finally, the `in particular' part of the statement follows from the fact that $\mathbb{G}_{2,p}$ and $\mathbb{G}$ are easily seen to be bidefinable (the only difference in languages is that $\mathbb{G}$ has a single predicate while $\mathbb{G}_{2,p}$ has two additional predicates, one for the trivial subgroup and one for the entire subgroup).    
\end{proof}

\begin{rem}
    By Proposition \ref{prop:algebraicpropertiesL}, the predicates for the terms of the Lazard series, interpreted in $\mathbf{G}_{c,p}$, are definable in the language of groups.  Therefore, after taking the reducts to the language of groups, the theories $\Th(\mathbf{G}_{c,p})$ as $c$ varies give examples of pure groups witnessing the strictness of the NIP$_{k}$ hierarchy. 
\end{rem}

\subsection*{Acknowledgements}

Much of this work was completed at the Institut Henri Poincar\'e when the four authors were hosted as part of a `Research in Paris' program. The authors would also like to thank the Institut Henri Poincar\'e for their hospitality.  We would like to thank Gabriel Conant, Berthold Maier, Dugald Macpherson, and Erik Walsberg for valuable conversations and correspondence on this project. We would also like to thank Amador Martin-Pizarro and Martin Ziegler for helpful suggestions on the presentation of this article.

\bibliographystyle{plain}
\bibliography{biblio.bib}{}

\end{document}